\documentclass[11pt]{article}
\usepackage[reqno]{amsmath}
\usepackage{amssymb,amscd,latexsym,bm,paralist}
\usepackage[all]{xy}

\newcommand{\C}{{\mathbb{C}}}
\newcommand{\Pa}{{\mathbb{P}}}
\newcommand{\Q}{{\mathbb{Q}}}

\newcommand{\T}{\mathbb{T}}

\newcommand{\R}{{\mathbb{R}}}
\newcommand{\Z}{{\mathbb{Z}}}

\newcommand{\bfN}{\mathbf{N}}
\newcommand{\bfX}{\mathbf{X}}
\newcommand{\bfd}{\mathbf{d}}
\newcommand{\bs}{\mathbf{s}}

\newcommand{\dist}{\mathrm{dist}}

\newcommand{\arc}{\mathrm{arc}\,}

\newcommand{\id}{\mathrm{id}}
\newcommand{\pre}{\mathrm{pre}}
\renewcommand{\mod}{\;\mathrm{mod}\;}
\newcommand{\ord}{\mathrm{ord}}
\newcommand{\res}{\mathrm{res}}
\newcommand{\sing}{\mathrm{sing}}

\newcommand{\Imm}{\mathrm{Im}\,}

\newcommand{\oB}{\overline{B}}

\newcommand{\tM}{\tilde{M}}
\renewcommand{\th}{\tilde{h}}
\newcommand{\tsigma}{\tilde{\sigma}}
\newcommand{\tmu}{\tilde{\mu}}
\newcommand{\hmu}{\hat{\mu}}
\newcommand{\tnu}{\tilde{\nu}}
\newcommand{\hsigma}{\hat{\sigma}}
\newcommand{\tG}{\tilde{G}}

\newcommand{\RRe}{\mathrm{Re}\,}

\newcommand{\Tr}{\mathrm{Tr}}
\newcommand{\Var}{\mathrm{Var}}

\newcommand{\Ah}{{\mathcal A}}
\newcommand{\Bh}{\mathcal{B}}

\newcommand{\Dh}{{\mathcal D}}
\newcommand{\Eh}{{\mathcal E}}
\newcommand{\Fh}{{\mathcal F}}
\newcommand{\Gh}{{\mathcal G}}

\newcommand{\Ih}{{\mathcal I}}
\newcommand{\Kh}{\mathcal{K}}
\newcommand{\Lh}{{\mathcal L}}
\newcommand{\Mh}{\mathcal{M}}
\newcommand{\Nh}{\mathcal{N}}
\newcommand{\Oh}{{\mathcal O}}
\newcommand{\Ph}{\mathcal{P}}
\newcommand{\Rh}{{\mathcal R}}
\newcommand{\Sh}{{\mathcal S}}

\newcommand{\Uh}{\mathcal{U}}

\newcommand{\emm}{{\mathfrak{m}}}
\newcommand{\eo}{\mathfrak{o}}
\newcommand{\fs}{\mathfrak{s}}
\newcommand{\teo}{\tilde{\eo}}

\newcommand{\ozeta}{\overline{\zeta}}

\newcommand{\tf}{\tilde{f}}

\newcommand{\oc}{\overline{c}}
\newcommand{\silo}{\xrightarrow{\sim}}
\newcommand{\tei}{\, | \,}
\newcommand{\hullet}{\raisebox{0.05cm}{$\,\scriptscriptstyle \bullet\,$}}
\newcommand{\verk}{\mbox{\scriptsize $\,\circ\,$}}
\newcommand{\halb}{\frac{1}{2}}

\newtheorem{theorem}{Theorem}
\newtheorem{lemma}[theorem]{Lemma}
\newtheorem{prop}[theorem]{Proposition}
\newtheorem{defn}[theorem]{Definition}
\newtheorem{cor}[theorem]{Corollary}
\newtheorem{example}[theorem]{Example}
\newtheorem{remark}[theorem]{Remark}
\newtheorem{remarks}[theorem]{Remarks}

\newtheorem{conjecture}[theorem]{Conjecture}
\newtheorem{notation}[theorem]{Notation}
\newtheorem{conj}[theorem]{Conjecture}

\newenvironment{rem}{\noindent {\bf Remark}}{}
\newenvironment{rems}{\noindent {\bf Remarks}}{}
\newenvironment{quest}{\noindent {\bf Question}}{}

\newenvironment{proof}{\noindent {\bf Proof}}{\mbox{}\hfill$\Box$}
\begin{document}
\title{Invariant measures on the circle and functional equations}
\author{Christopher Deninger}
\date{\ }
\maketitle
\section{Introduction}
\label{sec:1}
For an integer $N \ge 1$ consider the endomorphism $\varphi_N$ of the unit circle $\T$ given by $\varphi_N (\zeta) = \zeta^N$. It is known that besides Haar measure there are many $\varphi_N$-invariant atomless probability measures on $\T$, see \cite{B}. 

There are natural ways to characterize a measure $\mu$ on $\T$ by an associated function defined either on $\T$ or holomorphic in the unit disc. Invariance of the measure under $\varphi_N$ translates into functional equations for the corresponding functions. For example consider the holomorphic function $f_{\mu} = \exp (-h_{\mu})$ on the unit disc $D$ where $h_{\mu}$ is the Herglotz-transform of $\mu$
\[
h_{\mu} (z) = \int_{\T} \frac{\zeta+z}{\zeta-z} \, d\mu (\zeta) \, .
\]
Then $\varphi_N$-invariance of $\mu$ is equivalent to a functional equation for $f = f_{\mu}$ 
\begin{equation} \label{eq:1.1}
 f (z^N)^N = \prod_{\zeta^N = 1} f (\zeta z) \; .
\end{equation}
In section \ref{sec:3} we study this functional equation in its own right within the Nevanlinna class $\Nh$. Theorem \ref{t33} asserts that up to a unique positive constant any non-zero function $f$ in $\Nh$ satisfying \eqref{eq:1.1} is a quotient of singular inner functions: Blaschke products and outer functions in $\Nh$ cannot satisfy \eqref{eq:1.1} unless they are constant. Using these facts we give a somewhat surprising characterization in corollary \ref{t37} of those sequences of complex numbers which arise as Fourier coefficients of $\varphi_N$-invariant measures. The relevant notions from the theory of Hardy spaces are reviewed in section \ref{sec:2}. 

In section \ref{sec:4} the cumulative mass function $\hmu$ of a measure $\mu$ on the circle is studied as well as a related analytic function $G_{\mu}$. The focus is on their behaviour under push-forward  and pullback along $\varphi_N$. In particular, $\varphi_N$-invariance of $\mu$ is characterized in terms of functional equations for $\hmu$ and $G_{\mu}$. 

Not much is known about measures on $\T$ which are invariant under at least two endomorphisms $\varphi_N$ and $\varphi_M$ with $N$ prime to $M$, but see \cite{R}. It is therefore interesting to look for holomorphic functions on $D$ which satisfy the functional equation \eqref{eq:1.1} for several integers $N$. We begin this study in section \ref{sec:5}. Consider the multiplicative monoid $\Sh$ generated by pairwise prime integers $N_1 , \ldots , N_s \ge 2$. It acts on $\T$ if we identify $N \in \Sh$ with $\varphi_N$. For a subgroup $\Gh \subset \Oh^1 = \{ f \in \Oh (D)^{\times} \tei f (0) = 1 \}$ set
\[
 H^0 (\Sh , \Gh) = \{ f \in \Gh \tei f \; \text{satisfies \eqref{eq:1.1} for all} \; N \in \Sh \} 
\]
and
\[
 Z (\Sh , \Gh) = \{ \alpha \in \Gh \tei \prod_{\zeta^N = 1} \alpha (\zeta z) = 1 \quad \text{for} \; 1 \neq N \in \Sh \} \; .
\]
Here the conditions need to be checked for $N = N_1 , \ldots , N_s$ only. The group $Z (\Sh , \Oh^1)$ is easy to describe as a certain quotient of $\Oh^1$. Moreover there are mutually inverse isomorphisms
\[
 Z (\Sh , \Oh^1) \xleftarrow[\Phi_{\Sh}]{\xrightarrow{\; \Psi_{\Sh}\; } }H^0 (\Sh , \Oh^1) \; .
\]
For $s = 1$ they are given by the formulas
\[
 \Phi_{\Sh} (f) (z) = f (z) / f (z^{N_1}) \quad \text{and} \quad \Psi_{\Sh} (\alpha) (z) = \prod^{\infty}_{\nu = 0} \alpha (z^{N^{\nu}_1}) \; .
\]
For general $\Sh$ we have
\[
 \Psi_{\Sh} (\alpha) = \prod_{N \in \Sh} \alpha (z^N) \; .
\]
See proposition \ref{t52} for details. Thus for $f \in \Oh^1$ the description of simultanous solutions of \eqref{eq:1.1} is easy. As we saw in the third section, the situation becomes interesting when one imposes growth conditions on the solutions $f$. Recall that for a probability measure $\mu$ on $\T$ the function $f_{\mu}$ lies in the Hardy space $H^{\infty} (D)$ of bounded analytic functions on $D$. 

If $\mu$ is $\varphi_N$-invariant for $N \in \Sh$ then $f_{\mu}$ lies in $H^0 (\Sh , \Oh^1)$. Consequently $\Phi_{\Sh} (f_{\mu}) \in Z (\Sh , \Nh^1)$ where $\Nh^1 = \Nh^{\times} \cap \Uh$. Here we have used that quotients of nowhere vanishing bounded holomorphic functions lie in $\Nh^{\times}$. 

It is not known which functions are of the form $f_{\mu}$ for an $\Sh$-invariant probability measure $\mu$. By the above they can be recovered from $\Phi_{\Sh} (f_{\mu})$ by applying $\Psi_{\Sh}$.  Thus it is natural to study the map $\Psi_{\Sh}$ on $Z (\Sh , \Nh^1)$. The space $Z (\Sh , \Nh^1)$ is naturally a quotient of $\Nh^1$ with a known kernel, c.f. proposition \ref{t52}. The image under $\Psi_{\Sh}$ contains the space $H^0 (\Sh , \Nh^1)$ whose structure we would like to understand but it is strictly bigger. One basic result is theorem \ref{t939n} from section \ref{sec:10} which asserts that
\[
\Psi_{\Sh} (Z (\Sh , \Nh^1)) \subset H^0 (\Sh , \Nh^1_s) \; .
\]
Here $\Nh^1_s = \Nh^{\times}_s \cap \Uh$ and $\Nh_s$ is the algebra of functions $f \subset \Oh (D)$ that can be written in the form $f = g_1 g^{-1}_2$ where $g_2$ has no zeroes and both $g_1$ and $g_2$ satisfy an estimate of the form
\begin{equation} \label{eq:1.2n}
 |g (z)| \le a_g \exp (r_g \log^s (1 - |z|)^{-1}) \quad \text{for} \; z \in D
\end{equation}
where $a_g \ge 0$ and $r_g \ge 0$ are constants. For $s = 0$ the estimate \eqref{eq:1.2n} asserts that $g \in H^{\infty} (D)$ so that $\Nh_0 = \Nh$. For $s = 1$ it asserts that
\[
 |g (z)| \le a_g (1 - |z|)^{-r_g} \; .
\]
This means that $g \in \Ah^{-\infty}$ in the notation of Korenblum \cite{K1}, \cite{K2}. The more general classes $\Nh_s$ appear in the works \cite{BL}, \cite{K4} and \cite{S} for example.

Classically the elements of $\Nh^1$ can be described by finite signed measures on $\T$. More generally, by a theorem of Korenblum the elements of $\Nh^1_s$ correspond to real premeasures of bounded $\kappa_s$-variation on the circle. Here $\kappa_s$ is the generalized entropy-function on $[0,1]$
\[
 \kappa_s (x) = x \sum^s_{\nu = 0} \frac{1}{\nu!} |\log x|^{\nu} \; .
\]
Thus $\kappa_0 (x) = x$ and $\kappa_1 (x) = x (1 + |\log x|) = x \log \frac{e}{x}$. The premeasure $\mu$ on $\T$ is of bounded $\kappa_s$-variation if there is a constant $A \ge 0$ such that
\[
 \sum_j |\mu (C_j)| \le A \sum_j \kappa_s (|C_j|)
\]
holds for all finite partitions of $\T$ into disjoint connected subsets $C_j$ (arcs). Here $|C|$ is the arc length of $C$ normalized by $|\T| = 1$. 

If the premeasure $\mu$ corresponds to $f \in \Nh^1_s$ then as for measures, $\mu$ is $\varphi_N$-invariant if and only if $f$ satisfies equation \eqref{eq:1.1}. Hence we have obtained an injection from $Z (\Sh , \Nh^1)$ into the space of premeasures of bounded $\kappa_s$-variation which are invariant under $N_1 , \ldots , N_s$ c.f. corollary \ref{t822}. One can do a little better: For suitable functions in $Z (\Sh , \Nh^1)$ one even obtains premeasures of bounded $\kappa_{s-1}$-variation invariant under $N_1 , \ldots , N_s$, c.f. proposition \ref{t925}.

Classically the atoms of a measure $\mu$ can be seen in the function $f_{\mu}$. For the Korenblum correspondence between premeasures and functions this is still true but more subtle c.f. theorem \ref{t7.19}. It rests on a positivity argument as with the F\'ej\`er kernel in Fourier analysis.\\
In the theory described up to now there are analogous assertions for spaces of atomless (pre-)measures and functions. For example, one obtains many $\varphi_N$ and $\varphi_M$ invariant atomless premeasures of bounded $\kappa_1$-variation. 

Instead of starting from the group of functions $Z (\Sh , \Nh^1)$ it is also possible to begin with the isomorphic group of finite signed measures $\sigma$ with 
\[
N_* \sigma := (\varphi_N)_* \sigma = 0 \quad \text{for} \; N \in \Sh , N \neq 1 \; .
\]
Namely, the series
\begin{equation} \label{eq:1.3n}
 \mu = \Psi_{\Sh} (\sigma) := \sum_{N \in \Sh} N^* \sigma
\end{equation}
converges on arcs to a premeasure $\mu$ of $\kappa_s$-bounded variation with $N_* \mu = \mu$ for all $N \in \Sh$. Convergence of \eqref{eq:1.3n} and $\Sh$-invariance of $\mu$ can be shown directly using cumulative mass functions. The $\kappa_s$-bounded variation of $\mu$ follows by comparison with the above analytical theory. A purely measure theoretic proof for this fact should also be possible. As with the space $Z (\Sh , \Nh^1)$, the corresponding space of measures $\sigma$ as above is easy to describe, c.f. corollary \ref{t61}. The premeasure $\mu = \Psi_{\Sh} (\sigma)$ has atoms if and only if $\sigma$ has atoms.

As part of a more general theory, Korenblum has shown that premeasures $\mu$ of $\kappa = \kappa_s$-bounded variation induce compatible measures $\mu^F$ on the Borel algebras of $\kappa$-Carleson sets $F$. These are closed subsets of $\T$ of Lebesgue measure zero such that
\[
 \sum_I \kappa (|I|) < \infty \; .
\]
Here $\T \setminus F = \amalg I$ is the decomposition into connected components $I$. The family $\mu_s = (\mu^F)$ is called the $\kappa$-singular measure attached to $\mu$. In section \ref{sec:8} we review Korenblum's theory of $\kappa$-singular measures. Moreover, using his results and general facts from measure theory we show that $\kappa$-singular measures can be interpreted as ``$\kappa$-thin measures'' $\tmu$. These live in the Grothendieck group of a semigroup of positive $\sigma$-finite measures on the Borel algebra of $\T$ (with further properties). Thus $\tmu$ is given by a class of pairs of $\sigma$-finite positive measures $\tmu_i$:
\[
 \tmu = [\tmu_1 , \tmu_2] \; .
\]
Because of a cancellation property there is equality
\[
 [\tmu_1 , \tmu_2] = [\tnu_1 , \tnu_2]
\]
if and only if $\tmu_1 + \tnu_2 = \tmu_2 + \tnu_1$. Combining the previously defined maps $\Psi_{\Sh}$ with the passage to $\kappa_s$-thin measures, we obtain for every $\alpha \in Z (\Sh , \Nh^1)$ or corresponding measure $\sigma$, pairs of $\sigma$-finite measures $\tmu_1 , \tmu_2$ with $\tmu_1 + N_* \tmu_2 = N_* \tmu_1 + \tmu_2$ for all $N \in \Sh$. The measures $\tmu_i \ge 0$ live on countable unions of $\kappa_s$-Carleson sets and are restricted by further properties. If $\tmu_1$ or $\tmu_2$ is finite then $\tmu = \tmu_1 - \tmu_2$ is a signed measure and both $\tmu^+$ and $\tmu^-$ are $\Sh$-invariant. 

We prove that every $\Sh$-invariant positive ergodic probability measure which is non-zero on some $\kappa_s$-Carleson set is $\kappa_s$-thin and can be obtained by the preceeding constructions. I think that the last condition is automatically satisfied. This is related to conjecture \ref{t1054} which asserts that non-constant cyclic elements in certain topological algebras $\Ah_{\gamma} \subset \Oh (D)$ defined by growth conditions cannot satisfy the functional equation \eqref{eq:1.1} for too many coprime integers $N$. The relation comes from Korenblum's theory \cite{K2}, \cite{K3} characterizing cyclicity in terms of vanishing $\kappa$-singular measure. For $\gamma = 0$ the conjecture is true. It is inspired by a corresponding result, theorem \ref{t926}, for functions in $\Nh_{\gamma}$ with a zero in $D$ which seems to be analogous. Theorem \ref{t926} follows from the work of Seip \cite{S}. I think that conjecture \ref{t1054} is an interesting challenge for experts in the theory of growth spaces of analytic functions on $D$. 

We would like to draw attention to the work \cite{EP} of Eigen und Prasad. They observe that for an $NM$-invariant ergodic probability measure $\nu$ on $\T$ the orbits $O_N (\nu) = \{ N^i_* \nu \tei i \ge 0 \}$ and $O_M (\nu) = \{ M^i_* \nu \tei i \ge 0 \}$ have the same cardinality and consist of mutually singular measures. A short argument shows that the $\sigma$-finite positive measure
\begin{equation} \label{eq:1.4n}
 \mu = -\nu + \sum_{\nu' \in O_N (\nu)} \nu' + \sum_{\nu'' \in O_M (\nu)} \nu''
\end{equation}
is both $N$ and $M$ invariant. With a suitable choice of initial measure $\nu$ they obtain a non-atomic $\mu$ which is even $\Sh = \langle N , M \rangle$ ergodic (for $N = 2 , M = 3$). We were unable to fit their construction into our framework. The reason may be this: Unless $\nu$ is already invariant under $N^k$ and $M^k$ for some integer $k \ge 1$ the series \eqref{eq:1.4n} will never define a finite measure. The map $\Psi_{\Sh}$ given by the series \eqref{eq:1.3n} on the other hand has all $\Sh$-invariant probability measures on $\T$ in its image (up to a scalar multiple of Haar measure). 

In writing the paper we found it useful to introduce a certain number of operations on functions, (pre-)measures and (Schwartz-)distributions and study their relations. These operations behave like Frobenius, Verschiebung and the Teichm\"uller character for Witt vectors. In the appendix we embed the ring $\Dh' (\T)$ of distributions on $\T$ under convolution into the ring of big Witt vectors of $\C$ and identify the corresponding operations on both sides. As a small example we note that the Artin--Hasse exponential for the prime $p$ is the image of a $p$-invariant premeasure on $\T$ of $\kappa_1$-bounded variation which is not a measure and whose $\kappa_1$-thin (or singular) measure is zero, c.f. proposition \ref{a31}. 

Finally we would like to point out three reasons why we have been working with the muliplicative functional equation \eqref{eq:1.1} instead of the additive functional equation satisfied by the Herglotz transform $h = h_{\mu}$
\begin{equation} \label{eq:1.5n}
 N h (z^N) = \sum_{\zeta^N = 1} h (\zeta z) \; .
\end{equation}
Firstly, for non-zero singular measures the Herglotz transform is never in $H^1 (D)$, only in the non-locally convex spaces $H^p (D)$ for $0 < p < 1$. The function $f = f_{\mu}$ on the other hand lies in the much studied Banach algebra $H^{\infty} (D)$. Secondly, although $f_{\mu}$ has no zeroes, the functional equation \eqref{eq:1.1} makes sense also for functions $f$ with zeroes in $D$ and one is led to interesting results about those, c.f. theorems \ref{t33} and \ref{t926}. Thirdly the proof of theorem \ref{t34} below uses the functional equation \eqref{eq:1.1} in a way that has no analogue for the additive functional equation \eqref{eq:1.5n}. 

Despite its length the present paper marks only a beginning in the study of the relations between $\Sh$-invariant (pre-)measures and generalized Nevanlinna theory: It contains groundwork like section \ref{sec:9} and also reviews of the required Hardy space and Nevanlinna--Korenblum theories. The deeper aspects still need to be explored. In particular it would be very desirable to introduce geometrical methods from the theory of conformal mappings.

This paper owes its existence to many interesting discussions with Wilhelm Singhof. I would like to thank him very much. I am also grateful to Klaus Schmidt for drawing my attention to invariant measures on the circle and to Boris Korenblum for a helpful email exchange.

\tableofcontents

\section{Background on Hardy spaces and the Nevanlinna class}
\label{sec:2}
In this section we introduce some notations and review a number of basic facts that will be needed in the sequel. Details can be found in the books \cite{D}, \cite{G} and \cite{CMR}.

The Banach space of finite real (or signed) Borel measures on the circle with the total variation norm is denoted by $M (\T)$. We set $M^+ (\T) = \{ \mu \in M (\T) \tei \mu \ge 0 \}$. The Haar probability measure on $\T$ will be denoted by $\lambda$. Under the standard identification $\R / 2 \pi \Z \silo \T , \theta \mapsto \exp (i \theta)$ it is given by $(2 \pi)^{-1} d \theta$. In terms of $\zeta \in \T$ we have $\lambda = (2 \pi i)^{-1} \zeta^{-1} \, d\zeta$. A measure $\mu \in M (\T)$ is called singular if it is singular with respect to $\lambda$. We write $M (\T)_{\sing}$ and $M^+ (\T)_{\sing}$ for the subsets of singular measures. The Herglotz transform $h_{\mu}$ of a measure $\mu \in M (\T)$ is the analytic function in the unit disc $D$ defined by the formula
\begin{equation}
 \label{eq:2.1}
h_{\mu} (z) = \int_{\T} \frac{\zeta + z}{\zeta - z} \, d\mu (\zeta) \; .
\end{equation}
It is connected to the Cauchy transform
\begin{equation}
 \label{eq:2.2}
K_{\mu} (z) = \int_{\T} \frac{1}{1 - \ozeta z} \, d\mu (\zeta)
\end{equation}
by the simple formula $h_{\mu} = 2 K_{\mu} - \mu (\T)$. We refer to the book \cite{CMR} for an in depth account of the Cauchy- and hence the Herglotz transform. The Cauchy integral formula implies that we have $K_{\lambda} = 1 = h_{\lambda}$. A measure $\mu \in M (\T)$ is uniquely determined by its Fourier coefficients for $\nu \in \Z$
\[
 c_{\nu} = c_{\nu} (\mu) = \int_{\T} \zeta^{-\nu} \, d\mu (\zeta) \; .
\]
They are bounded and we have the formula
\begin{equation}
 \label{eq:2.4}
h_{\mu} (z) = c_0 + 2 \sum^{\infty}_{\nu = 1} c_{\nu} z^{\nu} \; .
\end{equation}
For a measure $\mu \in M (\T)$ the Fourier coefficients satisfy the relation $c_{-\nu} = \overline{c_{\nu}}$. Hence $\mu$ is determined by $h_{\mu}$. The importance of the Herglotz transform is due to the following result:

\begin{theorem}[Herglotz]
 \label{t21}
The Herglotz transform sets up a bijection between the non-zero measures $\mu \in M^+ (\T)$ and the analytic functions $h$ on $D$ with $\RRe h \ge 0$ and $h (0) > 0$. 
\end{theorem}

See \cite{CMR} Theorem 1.8.9 for references to different proofs. Note that for $\mu \in M^+ (\T)$ the Herglotz transform is constant if and only if $\mu$ is a scalar multiple of $\lambda$. This follows from equation \eqref{eq:2.4}. In all other cases we have $\RRe h (z) > 0$ for all $z \in D$ since $h_{\mu} (D)$ is open.

The atoms of $\mu \in M (\T)$ can be recovered from the Herglotz transform since Lebesgue's dominated convergence theorem gives the formula:
\begin{equation}
 \label{eq:2.5}
\lim_{r\to 1-} (1-r) h_{\mu} (r\zeta) = 2 \mu \{ \zeta \} \quad \mbox{for all} \; \zeta \in \T \; .
\end{equation}
For real $0 < p < \infty$ the Hardy space $H^p = H^p (D)$ consists of all analytic functions $f$ on $D$ such that
\[
 \| f \|^p_{H^p} := \sup_{r < 1} \int_{\T} |f (r\zeta)|^p d \lambda (\zeta) < \infty \; .
\]
For $p = \infty$ the Hardy space $H^{\infty} = H^{\infty} (D)$ is the space of bounded analytic functions on $D$ with the $\sup$-norm. For $\mu \in M^+ (\T)$ we have $f_{\mu} = \exp (-h_{\mu}) \in H^{\infty}$ for example. For $1 \le p \le \infty$,  $H^p$ is a Banach space with the norm $\| \; \|_{H^p}$ and for $p = 2$ it is even a Hilbert space. For $0 < p < 1$ the translation invariant metric 
\[
 d_p (f,g) = \| f-g \|^p_{H^p}
\]
turns $H^p$ into a complete metric space. For an analytic function $f (z) = \sum_{\nu \ge 0} a_{\nu} z^{\nu}$ on $D$ we have $\|f \|^2_{H^2} = \sum_{\nu \ge 0} |a_{\nu}|^2$. Thus $f$ belongs to $H^2$ if and only if $\sum_{\nu \ge 0} |a_{\nu}|^2 < \infty$. The Nevanlinna class $\Nh$ is the algebra of analytic functions $f$ on $D$ which satisfy:
\[
 \sup_{r < 1} \int_{\T} \log^+ |f (r\zeta)| \, d\lambda (\zeta) < \infty
\]
or equivalently:
\[
 \sup_{r<1} \int_{\T} \log (1 + |f (r\zeta)|) \, d\lambda (\zeta) < \infty \; .
\]
Here $\log^+ x = \max (\log x , 0)$ for $x \ge 0$. It can be shown that $\Nh$ is the class of functions $f = gh^{-1}$ where $g,h \in H^{\infty}$ and $h$ has no zeroes in $D$. In particular $f_{\mu} \in \Nh$ for any $\mu \in M (\T)$. 

For any function $f$ in $\Nh$ the boundary function:
\[
 \tf (\zeta) = \lim_{r\to 1-} f (r\zeta)
\]
exists for $\lambda$-almost all $\zeta \in \T$ and we have $\log |\tf| \in L^1 = L^1 (\T , \lambda)$ unless $f = 0$. The Smirnov class $\Nh^+$ consists of all functions $f$ in $\Nh$ such that
\begin{equation}
 \label{eq:2.7}
\sup_{r < 1} \int_{\T} \log^+ |f (r\zeta)| \, d\lambda (\zeta) = \int_{\T} \log^+ |\tf| \, d \lambda \; .
\end{equation}
It is actually a subalgebra of $\Nh$. For $0 < p \le \infty$ we have $H^p \subset \Nh^+ \subset \Nh$. A function $f \in \Nh^+$ is in $H^p$ if and only if $\tf \in L^p$. The map $f \mapsto \tf$ sets up an isometry $H^p \hookrightarrow L^p = L^p (\T , \lambda)$. It is known that for $\mu \in M (\T)$ we have $h_{\mu} \in H^p$ for all $0 < p < 1$. In particular $h_{\mu}$ has radial limits $\lambda$-almost everywhere. Fatou's theorem asserts that $\lambda$-a.e. the equation $\RRe \th_{\mu} = \frac{d\mu}{d\lambda}$ holds. In particular a measure $\mu \in M (\T)$ is singular if and only if $\RRe \tilde{h}_{\mu} = 0$ holds $\lambda$-a.e. See \cite{CMR} Theorem 1.8.6 noting that we consider only real measure.

A function $\eo$ in $\Nh \setminus 0$ is called an outer function if we have
\begin{equation}
 \label{eq:2.8}
\log |\eo (0)| = \int_{\T} \log |\teo| \, d\lambda \; .
\end{equation}
In this case $\alpha \in \Nh^+$ and there is a constant $\omega \in \C$ with $|\omega| = 1$ such that we have
\begin{eqnarray}
 \eo (z) & = & \omega \exp \int_{\T} \frac{\zeta + z}{\zeta - z} \log |\teo (\zeta)| \, d\lambda (\zeta) \label{eq:2.9}\\
& = & \omega \exp h_{\log |\teo| \, \lambda} \nonumber
\end{eqnarray}
It is clear that $\omega = 1$ in case $\eo (0) > 0$. On the other hand for every measurable function $\alpha : \T \to \R$ with $\alpha \ge 0$ such that $\log \alpha \in L^1$, the function $\eo = \exp h_{\log \alpha \, d\lambda}$ is an outer function in $\Nh$ and we have $|\teo| = \alpha$ $\lambda$-almost everywhere. If in addition $\alpha \in L^p$ then $\eo$ is (an outer function) in $H^p$. There is the following functional analytic characterization of being outer. A function $f$ in $H^p$ is said to be cyclic if the linear subspace generated by the functions $f , zf , z^2f , \ldots$ is dense in $H^p$. In other words $\C [z] f$ should be dense in $H^p$. It is known that for $0 < p < \infty$ a function $f \in H^p$ is outer if and only if it is cyclic, \cite{D}.

A holomorphic function $f$ on $D$ is called inner if $|f (z)| \le 1$ for all $z \in D$ and $|\tf| = 1$  $\lambda$-almost everywhere on $\T$. It is called a singular inner function if in addition $f (z) \neq 0$ for all $z \in D$. For any $\nu \in M^+ (\T)_{\sing}$ the function $s_{\nu} = \exp (-h_{\nu})$ is a singular inner function with $s_{\nu} (0) > 0$. If $f$ is a singular inner function then we have $f = \omega s_{\nu}$ for a unique singular measure $\nu \ge 0$ and a unique constant $\omega$ of absolute value one. These assertions follow from theorem \ref{t21} and Fatou's theorem. We call $\nu$ the singular measure of $f$. 

The possible zero sets of functions in $H^p$ for $0 < p \le \infty$ or $\Nh$ are determined as follows. Consider a map $\rho : D \to \Z_{\ge 0}$. Then we have $\ord_z f = \rho (z)$ for some $f \in \Nh$ and all $z \in D$ if and only if the Blaschke condition holds:
\begin{equation}
 \label{eq:2.10}
\sum_{z\in D} (1 - |z|) \rho (z) < \infty \; .
\end{equation}
For $f \in H^p$ the condition is the same. If \eqref{eq:2.10} holds for $\rho$, the Blaschke product
\[
 b_{\rho} (z) = z^{\rho (0)} \prod_{w \in D^*} \Big( \frac{|w|}{w} \frac{w - z}{1 -\overline{w} z} \Big)^{\rho (w)}
\]
defines an inner function with $\ord_z b_{\rho} = \rho (z)$ for all $z \in D$. Every inner function $f$ is a product $f = bs$ of a Blaschke product $b$ with a singular inner function $s$. We have $b = b_{\ord f}$ and hence this representation is unique. We will need the canonical factorization theorem: $\Nh \setminus \{ 0 \}$ consists of the functions $f$ that can be written in the form:
\begin{equation}
 \label{eq:2.11}
f = b \frac{s_1}{s_2} \eo
\end{equation}
where $b$ is a Blaschke product, $\eo$ is an outer function in $\Nh$ and $s_1$ and $s_2$ are singular inner functions whose singular measures $\nu_1$ and $\nu_2$ are mutually singular, $\nu_1 \perp \nu_2$. In \eqref{eq:2.11} we have $b = b_{\ord f}$ and $s_1 , s_2$ and $\eo$ are uniquely determined up to a constant factor of modulus one. In particular the singular (signed) measure $\nu = \nu_1 - \nu_2$ is uniquely determined by $f$. It is clear that we have $\eo = \omega \exp h_{\log |\tf| \, d\lambda}$ for some $\omega \in \T$. 

The functions $f \in \Nh^+$ are those with $\nu_2 = 0$ i.e. those which can be written in the form, where $s$ is singular inner:
\begin{equation}
 \label{eq:2.12}
f = b s \eo \; .
\end{equation}

\section{A functional equation in the Nevanlinna class}
\label{sec:3}
In this section we first express the $N$-invariance of a finite measure $\mu$ on the circle in terms of an additive functional equation for its Herglotz transform $h_{\mu}$. The function $f_{\mu} (z) = \exp (-h_{\mu} (z))$ satisfies the multiplicative functional equation \eqref{eq:1.1}. We show that up to a scalar any solution of \eqref{eq:1.1} in $\Nh$ is a quotient of singular inner functions.

For an integer $N \ge 1$ define operators $N^*$ and $\Tr_N$ on functions $h : D \to \C$ or $h : \T \to \C$ by setting $(N^* h) (z) = h (z^N)$ and
\[
 (\Tr_N h) (z) = \sum_{\zeta^N = 1} h (\zeta z)
\]
where $\zeta$ runs over the $N$-th roots of unity. For $\mu \in M (\T)$ we write $N_* \mu$ for $\varphi_{N^*} (\mu)$.

\begin{prop}
 \label{t31}
For any $\mu \in M (\T)$ we have
\begin{equation}
 \label{eq:3.1}
N^* h_{N_* \mu} = \frac{1}{N} \Tr_N (h_{\mu}) \; .
\end{equation}
The measure $\mu$ is $N$-invariant if and only if either of the following functional equations holds:
\begin{equation}
 \label{eq:3.2}
h (z^N) = \frac{1}{N} \sum_{\zeta^N = 1} h (\zeta z) \quad \text{for} \; h = h_{\mu} 
\end{equation}
and
\begin{equation}
 \label{eq:3.3}
f (z^N)^N = \prod_{\zeta^N = 1} f (\zeta z) \quad \text{for} \; f = f_{\mu} \; .
\end{equation}
\end{prop}

\begin{proof}
 For every $\eta \in \C$ we have:
\begin{equation}
 \label{eq:3.4}
\frac{\eta^N + z^N}{\eta^N - z^N} = \frac{1}{N} \sum_{\zeta^N = 1} \frac{\eta + \zeta z}{\eta - \zeta z} \quad \text{in} \; \C (z) \; .
\end{equation}
Namely both sides have the same divisor on $\Pa^1 (\C) = \C \cup \{ \infty \}$ and agree at $z = 0$. Integrating \eqref{eq:3.4} with respect to $d\mu (\eta)$ gives the equation $h_{N_* \mu} (z^N) = N^{-1} \Tr_N (h_{\mu})$. It follows that \eqref{eq:3.2} and \eqref{eq:3.3} hold if we have $N_* \mu = \mu$.

Now let $\mu \in M (\T)$ be arbitrary and assume that \eqref{eq:3.3} holds for $f = f_{\mu}$. Then \eqref{eq:3.2} holds for $h = h_{\mu}$ up to an additive constant which must be zero as follows by taking $z = 0$. Equation \eqref{eq:3.2} is equivalent to the assertion $N^* h_{\mu} = N^* h_{N_* \mu}$ i.e. to $h_{N_* \mu} = h_{\mu}$. On real measures the Herglotz transform is injective. We therefore get $N_* \mu = \mu$.
\end{proof}

Since we will be concerned with $N$-invariant signed measures the following fact is interesting to know:

\begin{prop}
 \label{t33n}
For a signed (not necessarily finite) measure $\mu$ on the Borel algebra of $\T$ consider the Jordan decomposition $\mu = \mu_+ - \mu_-$. Then $N_* \mu = \mu$ is equivalent to $N_* \mu_{\pm} = \mu_{\pm}$.
\end{prop}

\begin{proof}
 The equation $\mu = N_* \mu_+ - N_* \mu_-$ shows that $\mu_{\pm} \le N_* \mu_{\pm}$ by minimality of $\mu_{\pm}$ in the Jordan decomposition \cite{E} VII 1.12. If $\mu_-$ is finite then $\mu_- (\T) = (N_* \mu_-) (\T) < \infty$ implies that $\mu_- = N_* \mu_-$. Adding this equation to $\mu = N_* \mu$ we get $\mu_+ = N_* \mu_+$ as well. If $\mu_+$ is finite one argues similarly. The other direction is clear. 
\end{proof}

Here are some useful formulas. Let $h (z) = \sum^{\infty}_{\nu=0} a_{\nu} z^{\nu}$ be a holomorphic function in $D$. Then we have:
\begin{equation}
 \label{eq:3.5}
(N^{-1} \Tr_N h) (z) = \sum_{N \tei \nu} a_{\nu} z^{\nu} = \sum^{\infty}_{\nu = 0} a_{\nu N} z^{\nu N} \; .
\end{equation}
Recall that in terms of the Fourier coefficients $c_{\nu} = c_{\nu} (\mu) = \int_{\T} \zeta^{-\nu} \, d\mu (\zeta)$ of a measure $\mu$ in $M (\T)$ we have
\begin{equation}
 \label{eq:3.6}
h_{\mu} (z) = \mu (\T) + 2 \sum^{\infty}_{\nu = 1} c_{\nu} z^{\nu} \; .
\end{equation}
The following equality is clear from the definitions:
\begin{equation}
 \label{eq:3.7}
c_{\nu} (N_* \mu) = c_{\nu N} (\mu) \quad \text{for all} \; \nu \; \text{in} \; \Z \; . 
\end{equation}
Hence we get
\[
 N^* h_{N_* \mu} (z) = h_{N_* \mu} (z^N) = \mu (\T) + 2 \sum^{\infty}_{\nu = 1} c_{\nu N} z^{\nu N} \; .
\]
Thus the formula $N^* h_{N_* \mu} = N^{-1} \Tr_N (h_{\mu})$ in proposition \ref{t31} follows again. Moreover the following assertions hold:
\begin{equation}
 \label{eq:3.8}
\begin{array}{l}
 \text{For a measure $\mu \in M (\T)$ we have $N_* \mu = \mu$ if and only if $c_{\nu} = c_{\nu N}$} \\ 
\text{for all $\nu \in \Z$}
\end{array} 
\end{equation}
\begin{equation}
 \label{eq:3.9}
\begin{array}{l}
\text{A holomorphic function $h = \sum^{\infty}_{\nu = 0} a_{\nu} z^{\nu}$ in $D$ satisfies the functional} \\ 
\text{equation \eqref{eq:3.1} if and only if $a_{\nu} = a_{\nu N}$ holds for all $\nu \ge 0$.}
      \end{array}
\end{equation}

\begin{lemma}\label{t32}
 For $N \ge 2$ let $\varphi \in \Lh^1 (\T)$ be a solution of the additive functional equation
\[
 \varphi (\eta^N) = \frac{1}{N} \sum_{\zeta^N = 1} \varphi (\zeta \eta) \quad \text{for $\lambda$-a.a. $\eta \in \T$.}
\]
Then $\varphi$ is constant $\lambda$-a.e.
\end{lemma}

\begin{proof} {\bf 1}
The Fourier coefficients $c_{\nu}$ of $\varphi$ satisfy the relation $c_{\nu N} = c_{\nu}$ for all $\nu$. Since $c_{\nu} \to 0$ for $|\nu| \to \infty$ by the Riemann-Lebesgue lemma it follows that we have $c_{\nu} = 0$ for $\nu \neq 0$ i.e. $\varphi$ is constant $\lambda$-a.e.
\end{proof}

\begin{proof} {\bf 2} 
 Set $e_N (\varphi) (\eta) = \frac{1}{N} \sum_{\zeta^N = 1} \varphi (\zeta \eta)$ i.e. $e_N = N^{-1} \Tr_N$. Jessen \cite{J} has proved that for any sequence $(N_{\nu})_{\nu \ge 1}$ tending to infinity with $N_{\nu} \tei N_{\nu + 1}$ we have for all $\varphi \in \Lh^1 (\T)$
\[
 e_{N_{\nu}} (\varphi) \longrightarrow E (\varphi) = \int_{\T} \varphi \, d\lambda \quad \text{$\lambda$-a.e. on $\T$.}
\]
In particular we have $e_{N^{\nu}} (\varphi) \to E (\varphi)$ $\lambda$-a.e. as $\nu \to \infty$ and therefore since $e_{N^{\nu}} (\varphi) = \varphi (\eta^{N^{\nu}})$ for $\lambda$-a.a. $\eta \in \T$ by the functional equation:
\begin{equation}
 \label{eq:3.22}
\varphi (\eta^{N^{\nu}}) \longrightarrow E (\varphi) \quad \text{for $\lambda$-a.a. $\eta$ as $\nu \to\infty$.}
\end{equation}
Set
\[
 M_n = \Big\{ \eta \in \T \tei |\varphi (\eta) - E (\varphi)| \ge \frac{1}{n} \Big\}\; .
\]
Then we have
\[
 M = \{ \eta \in \T \tei \varphi (\eta) \neq E (\varphi) \} = \bigcup_{n \ge 1} M_n \; .
\]
Assume that $\lambda (M) > 0$. Then $\lambda (M_n) > 0$ for some $n \ge 1$. By Poincar\'e recurrence, for almost all $\eta \in \T$ there is a sequence $(\nu_j)_{j \ge 1}$ with $\lim_{j\to \infty} \nu_j = \infty$ and $\eta^{N^{\nu_j}} \in M_n$ for all $j \ge 1$. For these $\eta$, equation \eqref{eq:3.22} cannot hold. Thus we have obtained a contradiction and therefore $\lambda (M) = 0$. 
\end{proof}

We now look at the functional equation \eqref{eq:1.1} within the Nevanlinna class.

\begin{theorem}
 \label{t33}
For some $N \ge 2$ let $0 \neq f \in \Nh$ be a solution of the multiplicative functional equation \eqref{eq:1.1}.
Then there is a unique constant $c > 0$ such that $cf$ is a quotient of singular inner functions. If in addition $f \in \Nh^+$ there is a unique $c > 0$ such that $cf$ is a singular inner function.
\end{theorem}

\begin{proof}
 The functional equation implies the relation
\[
 N^2 \ord_0 f = \sum_{\zeta^N = 1} \ord_0 f = N \ord_0 f
\]
and hence $f (0) \neq 0$. We now show that $f (a) \neq 0$ for all $0 \neq a \in D$ as well. For any $k \ge 1$ we have by induction:
\[
 f (z^{N^k})^{N^k} = \prod_{\zeta^{N^k} = 1}  f (\zeta z) \; .
\]
This gives the relation
\begin{equation}
 \label{eq:3.11}
N^k \ord_w f = \sum_{z^{N^k} = w} \ord_z f \quad \text{for} \; 0 \neq w \in D \; .
\end{equation}
Assume that $f (a) = 0$ for some $0 \neq a \in D$. Then we have
\begin{eqnarray*}
 \sum_{z \in D} (1 - |z|) \; \ord_z f & \ge & \sum^{\infty}_{k = 0} \sum_{z \in D \atop z^{N^k} = a} (1 - |z|) \; \ord_z f \\
& = & \sum^{\infty}_{k = 0} (1 - |a|^{1/N^k}) \sum_{z \in D \atop z^{N^k} = a} \ord_z f \\
& \overset{\eqref{eq:3.11}}{=} & \ord_a f \; \sum^{\infty}_{k = 0} N^k (1 - |a|^{1 / N^k}) \; .
\end{eqnarray*}
We have 
\[
 N^k (1 - |a|^{1 / N^k}) \longrightarrow \log |a|^{-1} \neq 0 \quad \text{for} \; k \to \infty \; .
\]
Hence the previous series diverges, contradicting the assumption $f \in \Nh$. Thus $f$ has no zeroes in $D$. Hence we have a unique factorization of the form $f = u \eo$ where $\eo$ is an outer function in $\Nh$ with $c^{-1} := \eo (0) > 0$ and $u$ is a quotient of singular inner functions resp. a singular inner function if $f \in \Nh^+$. By the uniqueness of such a decomposition and since $f (z^N) , f (\zeta z)$ are in $\Nh$ as well it follows that with $f$ both $\eo$ and $u$ satisfy the functional equation \eqref{eq:1.1}. Let $\teo$ be the radial extension of $\eo$ to a measurable function on $\T$. Then it is known that $\varphi = \log |\teo|$ is in $L^1$. Since $\teo$ satisfies the multiplicative functional equation \eqref{eq:3.3} $\lambda$-almost everywhere on $\T$, we have $N^* \varphi = N^{-1} \Tr_N (\varphi)$ in $L^1$. By lemma \ref{t32} $\varphi$ is constant and hence $\eo = \exp h_{\varphi \lambda}$ is constant as well, $\eo = \eo (0) = c^{-1}$. This implies the assertion of the theorem.
\end{proof}

\begin{quest}
 For $0 < p < \infty$ the outer functions in $H^p$ are exactly the cyclic vectors in $H^p$. Can one use this characterization to show that any outer function in $H^p$ satisfying \eqref{eq:1.1} is constant? This would be very desirable because it might suggest an approach to conjecture \ref{t1054} below dealing with cyclic vectors in generalized Nevanlinna classes.
\end{quest}

The following result leads to a different proof of a special case of theorem \ref{t33}.

\begin{theorem}
 \label{t34}
For $N \ge 2$ and $p > 0$ let $f \in L^p (\T)$ be a solution of the following functional equation for $\lambda$-almost all $\eta \in \T$:
\begin{equation}
 \label{eq:3.12}
f (\eta^N)^N = \prod_{\zeta^N = 1} f (\zeta \eta) \; .
\end{equation}
Then $|f|$ is constant $\lambda$-a.e.
\end{theorem}

\begin{proof}
 We may assume that $f \neq 0$ in $L^p (\T)$. Using H\"older's inequality we get:
\begin{eqnarray*}
 \int_{\T} |f (\eta)|^p \, d \lambda (\eta) & = & \int_{\T} |f (\eta^N)^N|^{p/N} d\lambda (\eta) \quad \text{since} \; N_* \lambda = \lambda \\
& \overset{\eqref{eq:3.12}}{=} & \int_{\T} \prod_{\zeta^N = 1} |f (\zeta \eta)|^{p/N} d\lambda (\eta) \\
& \overset{(*)}{\le} & \prod_{\zeta^N = 1} \Big( \int_{\T} |f (\zeta \eta)|^p \Big)^{1/N} d\lambda (\eta)\\
& = & \int_{\T} |f(\eta)|^p d\lambda (\eta) \quad \text{since} \; \zeta_* \lambda = \lambda \; .
\end{eqnarray*}
Hence we have equality in H\"older's inequality $(*)$. It follows from theorem 188 in section 6.9 of \cite{HLP} that for $\zeta^N = 1$ the functions $|f (\zeta \eta)|^p$ are pairwise linearly dependent in $L^1 (\T)$. Because of $|f|^p \neq 0$ in $L^1 (\T)$ there are therefore constants $c_{\zeta} > 0$ such that we have
\[
 |f (\zeta \eta)|^p = c_{\zeta} |f (\eta)|^p \quad \text{in} \; L^1 (\T) \; .
\]
Applying this equality $N$-times we find
\[
 |f (\zeta^N \eta)|^p = c^N_{\zeta} |f (\eta)|^p \; .
\]
Since $\zeta^N = 1$ we get $c^N_{\zeta} = 1$ and hence $c_{\zeta} = 1$. It follows that we have
\[
 |f (\zeta \eta)| = |f (\eta)|
\]
for $\lambda$-almost all $\eta \in \T$. Combining this with the functional equation \eqref{eq:3.12} we get the relation:
\[
 |f (\eta^N)| = |f (\eta)| \quad \text{$\lambda$-a.e.}
\]
Since $N$ acts ergodically on $\T$ with respect to $\lambda$ it follows that $|f|$ is constant $\lambda$-a.e.
\end{proof}

\begin{remarks}
 \label{t35} \em
1) If $\log |f|$ is integrable on $\T$, theorem \ref{t34} also follows from lemma \ref{t32}.\\
2) For $p > 0$ we have a natural inclusion $H^p (\T) \hookrightarrow L^p (\T)$ given by mapping $f$ to its boundary function $\tf$. Moreover the function $f$ is inner if and only if $|\tf| = 1$ $\lambda$-a.e. Thus theorem \ref{t34} implies the following weaker version of theorem \ref{t33}:
\end{remarks}

\begin{cor}
 \label{t36}
For $p > 0$ let $0 \neq f \in H^p (\T)$ satisfy the functional equation \eqref{eq:1.1}. Then there is a unique constant $c > 0$, such that $cf$ is an inner function.
\end{cor}

The following curious fact was found in discussions with Wilhelm Singhof. Let us call a measure $\mu \in M (\T)$ special if it has the form $\mu = \alpha \lambda + \sigma$ where $\alpha \in \R$ and $\sigma \in M_+ (\T)$ is singular. The $N$-invariant special measures can be characterized by their Fourier coefficients as follows:

\begin{cor}
 \label{t37}
Fix $N \ge 2$. A bounded sequence $(c_{\nu})_{\nu \in \Z}$ of complex numbers is the sequence of Fourier coefficients of a (uniquely determined) $N$-invariant special measure if and only if the following conditions hold:
\begin{enumerate}
 \item [i)] $c_{-\nu} = \oc_{\nu}$ for $\nu \in \Z$
\item[ii)] $c_{\nu N} = c_{\nu}$ for all $\nu \in \Z$
\item[iii)] The series $\sum^{\infty}_{n=0} |b_n|^2$ converges, where $(b_n)$ is defined by the formula:
\[
 \sum^{\infty}_{n=0} b_n z^n = \exp \Big( -c_0 -2 \sum^{\infty}_{\nu = 1} c_{\nu} z^{\nu} \Big) \; .
\]
\end{enumerate}
Condition iii) may be replaced by the following weaker one:
\begin{enumerate}
 \item [iii')] For some $\varepsilon > 0$ the series $\sum^{\infty}_{n=0} |b_n (\varepsilon)|^2$ converges where $(b_n (\varepsilon))$ is defined by the formula:
\[
 \sum^{\infty}_{n=0} b_n (\varepsilon) z^n = \exp \Big( -\varepsilon \sum^{\infty}_{\nu = 1} c_{\nu} z^{\nu} \Big) \; .
\]
\end{enumerate}
\end{cor}

\begin{proof}
 For any $N$-invariant measure $\mu \in M (\T)$ conditions i) and ii) are satisfied. The Herglotz-transform of $\mu$ is given by 
\[
 h_{\mu} (z) = c_0 + 2 \sum^{\infty}_{\nu = 1} c_{\nu} z^{\nu} \; .
\]
Hence we have
\[
 f_{\mu} (z) = \sum^{\infty}_{n=0} b_n z^n = \exp (-h_{\mu} (z)) \; .
\]
Writing $\mu = \alpha \lambda + \sigma$ as above in the definition of ``special'' we have $h_{\mu} = \alpha + h_{\sigma}$ and hence $\RRe h_{\mu} \ge \alpha$. This implies that $f_{\mu} \in H^{\infty} (D) \subset H^2 (D)$. Therefore the Taylor coefficients $b_n$ of $f_{\mu}$ are square integrable, i.e. iii) holds.

Now assume that we are given a bounded sequence $(c_{\nu})$ with i), ii), iii). Then
\[
 h (z) = c_0 + 2 \sum^{\infty}_{\nu = 1} c_{\nu} z^{\nu}
\]
defines a holomorphic function in the unit disc which because of ii) satisfies the additive functional equation \eqref{eq:3.2}. Hence $f = \exp (-h)$ satisfies equation \eqref{eq:1.1}. Condition iii) asserts that we have $f \in H^2 (D)$. Now it follows from theorem \ref{t33} or even from corollary \ref{t36} (since $f(z) \neq 0$ for all $z \in D$) that $cf$ is a singular inner function for some $c > 0$. Hence there are a constant $\lambda \in \C^*$ and a singular measure $\sigma \in M_+ (\T)$ such that $\lambda f = \exp (-h_{\sigma})$. For some $\alpha \in \C$ we therefore have
\begin{equation}
 \label{eq:3.13}
h = \alpha + h_{\sigma} = h_{\alpha \lambda + \sigma} \; .
\end{equation}
Since $\sigma$ is real valued, condition i) for $\nu = 0$ implies that $\alpha \in \R$. Hence $\mu = \alpha \lambda + \sigma$ is a special measure. Equation \eqref{eq:3.13} implies that $c_{\nu} (\mu) = c_{\nu}$ for $\nu \ge 0$. Since $\mu$ is real valued, it follows from condition i) that we have $c_{\nu} (\mu) = c_{\nu}$ for $\nu \le 0$ as well. Thus we have found the desired special measure. It is $N$-invariant because of ii). The final remark follows by multiplying $(c_{\nu})$ resp. $\mu$ by a positive constant.
\end{proof}

\begin{rem}
 For $g \in L^1 (\T)$ the measure $g\lambda \in M (\T)$ is $N$-invariant by \eqref{eq:3.8} if and only if the Fourier coefficients $c_{\nu}$ of $g$ satisfy the relations $c_{\nu N} = c_{\nu}$ for all $\nu \in \Z$. For $N \ge 2$, because of the Riemann--Lebesgue lemma, this means that $g$ is constant. It follows that the $N$-invariant special measures agree with the $N$-invariant measures of the form $\mu = \mu_a + \sigma$ where $\mu \in M (\T)$ is absolutely continuous and $\sigma \in M_+ (\T)$ is singular.
\end{rem}

\section{The cumulative mass function of an invariant measure}
\label{sec:4}
In this section we characterize $N$-invariant measures by their cumulative mass functions and develop a certain amount of formalism that will be used later. 

In the following we will use the identifications $[0, 2\pi) \equiv \R /2 \pi \Z \equiv \T$ without further comment. In particular for a function $\alpha : \T \to \C$ we will write $\alpha (\theta)$ for $\alpha (\exp (i\theta))$ if $\theta \in [0,2\pi)$. A measure $\mu \in M (\T)$ gives rise to a function $\hmu : [0,2\pi) \to \R$ by setting $\hmu (\theta) = \mu [0,\theta)$. It is well known that in this way one obtains an isomorphism between $M (\T)$ and the space $V (\T)$ of functions $\hmu : [0, 2\pi) \to \R$ which satisfy the following conditions
\begin{enumerate}
 \item $\hmu (0) = 0$ and $\hmu (2 \pi -)$ exists.
\item $\hmu$ is left continuous.
\item $\hmu$ has bounded variation.
\end{enumerate}
For a function $\alpha : \T \to \C$ and a point $\zeta \in \T$ let us write $\alpha (\zeta \pm) = \lim_{\varepsilon \to 0 \pm} \alpha (\zeta e^{i\varepsilon})$ if the limit exists. We call $\alpha$ left-continuous in $\zeta$ if $\alpha (\zeta -) = \alpha (\zeta)$.

When viewed as a function $\hmu : \T \to \R$ the above conditions are these:
\begin{enumerate}
 \item $\hmu (1) = 0$ and $\hmu (1-)$ exists.
\item $\hmu$ is left continuous in every point $\zeta \in \T \setminus \{ 1 \}$.
\item $\hmu$ has bounded variation on $\T$.
\end{enumerate}
Note that condition 3 implies that $\hmu$ has right limits in all points of $\T$. Moreover $\hmu (\zeta +) = \hmu (\zeta) + \mu \{ \zeta \}$ for $\zeta \in \T$. For a measure $\mu \in M (\T)$ we have $\hmu (2\pi -) = \mu (\T)$. Hence the space $M^0 (\T)$ corresponds to the subspace $V^0 (\T)$ of functions $\hmu$ in $V (\T)$ with $\hmu (2\pi -) = 0$. Viewing $\hmu$ as a function on $\T$ the condition is that $\hmu$ is left-continuous on all of $\T$. The measures $\mu \in M_+ (\T)$ correspond to the cone $V_+ (\T)$ of real valued left continuous nondecreasing bounded functions $\hmu : [0,2\pi) \to \R$ with $\hmu (0) = 0$. The inverse map $V (\T) \to M (\T)$ attaches to $\hmu$ the Lebesgue--Stieltjes measure $d\hmu$. 

 For every $\mu \in M (\T)$ and every $g \in C^1 [0,2\pi]$ Fubini's theorem implies the following formula of partial integration: 
\begin{equation}
 \label{eq:4.25}
\int_{[0,2\pi)} g \, d\mu = - \int^{2\pi}_0 \hmu (\theta) g' (\theta) d\theta + g (2\pi) \mu (\T) \; .
\end{equation}
In particular we have
\begin{equation}
 \label{eq:4.26}
\mu_0 := \int_{\T} \hmu d\lambda = \mu (\T) - \frac{1}{2\pi} \int_{[0,2\pi)} \theta \, d\mu (\theta) \; .
\end{equation}
We also introduce the holomorphic function on $D$ defined by:
\begin{equation}
 \label{eq:4.27}
G_{\mu} (z) = \frac{1}{2\pi i} \int^z_0 (h_{\mu} (w) - \mu (\T)) \frac{dw}{w} \; .
\end{equation}
Note here that $h_{\mu} (0) = \mu (\T)$. Set $\log (1-z) = -\sum^{\infty}_{\nu=1} \frac{z^{\nu}}{\nu}$ for $|z| < 1$. The following relation between $\mu$ and $h_{\mu}$ is well known, c.f. \cite{HK} (1.5) in case $\mu (\T) = 0$.

\begin{prop}
 \label{t4.8}
a) For $\mu \in M (\T)$ we have $G_{\mu} \in H^2$ and the formula
\begin{equation}
 \label{eq:4.28}
h_{\hmu \lambda} = G_{\mu} + \mu_0 + \frac{\mu (\T)}{\pi i} \log (1-z) \; .
\end{equation}
b) For $\mu \in M (\T)$ we have $\lambda$-a.e. on $\T$:
\begin{equation}
 \label{eq:4.29}
\hmu = \RRe \tG_{\mu} + \mu_0 + \frac{\mu (\T)}{2\pi} (\theta - \pi) \; .
\end{equation}
\end{prop}

\begin{proof}
 The function $\hmu$ is of bounded variation, hence in $L^{\infty}$. Using \eqref{eq:4.25} its Fourier coefficients are seen to be $c_0 (\hmu) = \mu_0$ and
\begin{equation}
 \label{eq:4.30}
c_{\nu} (\hmu) = \frac{1}{2\pi i \nu} c_{\nu} (\mu) - \frac{1}{2 \pi i \nu} \mu (\T) \quad \text{for} \; \nu \neq 0 \; .
\end{equation}
According to \eqref{eq:2.4} we have:
\begin{equation}
 \label{eq:4.31}
G_{\mu} (z) = \frac{1}{\pi i} \sum^{\infty}_{\nu = 1} \frac{c_{\nu} (\mu)}{\nu} z^{\nu} \; .
\end{equation}
Since the sequence $(c_{\nu} (\mu))$ is bounded we see that the Taylor coefficients of $G_{\mu}$ form an $l^2$-sequence and hence $G_{\mu} \in H^2$. In particular $\tG_{\mu} \in L^2$ exists. Applying \eqref{eq:2.4} to $h_{\hmu \lambda}$, formula \eqref{eq:4.28} follows. For $\mu \in M (\T)$ we have $c_{-\nu} (\mu) = \overline{c_{\nu} (\mu)}$. Using \eqref{eq:4.31} this implies the relations $c_0 (\RRe \tG_{\mu}) = 0$ and for $0 \neq \nu \in \Z$:
\[
 c_{\nu} (\RRe \tG_{\mu}) =  \frac{c_{\nu} (\mu)}{2\pi i \nu} \; .
\]
Together with the formulas
\[
 c_{\nu} \Big( \frac{\pi - \theta}{2 \pi} \Big) = \frac{1}{2\pi i \nu} \quad \text{for} \; 0 \neq \nu \in \Z \quad \text{and} \quad c_0 \Big( \frac{\pi - \theta}{2\pi} \Big) = 0
\]
it follows from \eqref{eq:4.30} that both sides of \eqref{eq:4.29} have the same Fourier coefficients.
\end{proof}

Whether a measure $\mu \in M (\T)$ is $N$-invariant can be seen from either function $\hmu$ or $G_{\mu}$ as follows:

\begin{prop}
 \label{t49}
a) A measure $\mu \in M (\T)$ satisfies $N_* \mu = \mu$ for some $N \ge 1$ if and only if the following functional equation holds for some constant $c_0$
\begin{equation}
 \label{eq:4.32}
\hmu (\eta^N) = \sum_{\zeta^N = 1} \hmu (\zeta \eta) - c_0 \quad \text{for} \; \lambda\text{-a.a.} \; \eta \in \T \; .
\end{equation}
In this case \eqref{eq:4.32} holds for {\em all} $\eta \in \T$ and we have
\[
 c_0 = \sum_{\zeta^N = 1} \hmu (\zeta) = (N-1) \mu_0 \; .
\]
b) A measure $\mu \in M (\T)$ is $N$-invariant if and only if for all $z \in D$ we have:
\begin{equation}
 \label{eq:4.33}
G_{\mu} (z^N) = \sum_{\zeta^N = 1} G_{\mu} (\zeta z) \; .
\end{equation}
\end{prop}

\begin{proof}
 a) For $0 \le \theta < 2 \pi$ the disjoint decomposition
\[
 N^{-1} [0,\theta) = \bigcup^{N-1}_{\nu=0} \Big[ \frac{2\pi \nu}{N} , \frac{\theta + 2 \pi \nu}{N} \Big)
\]
implies the formula:
\begin{eqnarray*}
 \widehat{N_* \mu} (\theta) & = & \sum^{N-1}_{\nu = 0} \mu \Big[ \frac{2 \pi \nu}{N} , \frac{\theta + 2\pi \nu}{N} \Big) \\
& = & \sum^{N-1}_{\nu =0} \hmu \Big( \frac{\theta + 2 \pi \nu}{N} \Big) - c_0 \; ,
\end{eqnarray*}
where $c_0 = \sum^{N-1}_{\nu=0} \hmu  (\frac{2\pi \nu}{N} )$. It follows that setting $\eta = e^{i\theta}$ we have
\begin{equation}
 \label{eq:4.34}
\widehat{N_* \mu} (\eta) = \sum_{\zeta^N = 1} \hmu (\zeta \eta^{1/N}) - c_0
\end{equation}
where $\eta^{1/N} = e^{i \theta / N}$. Since the right hand side of \eqref{eq:4.34} is independent of the chosen branch of $\eta^{1/N}$ we get the equation
\begin{equation}
 \label{eq:4.35}
\widehat{N_* \mu} (\eta^N) = \sum_{\zeta^N=1} \hmu (\zeta \eta) - c_0 \quad \text{for} \; \eta \in \T \; .
\end{equation}
If $\mu$ is $N$-invariant, equation \eqref{eq:4.32} therefore holds for {\it all} $\eta \in \T$ with $c_0 = \sum_{\zeta^N = 1} \hmu (\zeta)$ and integrating it we also get $c_0 = (N-1) \mu_0$. 

Now assume that \eqref{eq:4.32} holds $\lambda$-a.e. for some constant $c_0$. Then we have $c_{\nu} (\hmu) = N c_{\nu N} (\hmu)$ for $\nu \neq 0$ and $c_0 = (N-1) \mu_0$. Now equation \eqref{eq:4.30} implies that for $\nu \neq 0$ we have $c_{\nu} (\mu) = c_{\nu N} (\mu)$ i.e. that $N_* \mu = \mu$ by assertion \eqref{eq:3.8}.\\
b) Equation \eqref{eq:4.33} is equivalent to the functional equation
\[
 N h_{\mu} (z^N) = \sum_{\zeta^N = 1} h_{\mu} (\zeta z) \; .
\]
Now the claim follows from proposition \ref{t31}.
\end{proof}

Given a measure $\mu \in M (\T)$, there is a measure $N^* \mu$ such that $h_{N^* \mu} = N^* h_{\mu}$. The following construction of $N^* \mu$ is a special case of a more general construction in ergodic theory. Its relevance for the Herglotz transform was noted in \cite{LH}. For $0 \le k < N$ set $I_k = [\frac{2\pi k}{N} , \frac{2\pi (k+1)}{N})$ also viewed as an arc of $\T$. Then $\varphi_N : I_k \silo \T$ is a bijection. Let $i_k : I_k \hookrightarrow \T$ be the inclusion and set
\[
 \mu_k = i_{k*} (\varphi_N \, |_{I_k})^{-1}_* \mu \; ,
\]
a real Borel measure on $\T$ with support on $I_k$. More explicitely:
\[
 \mu_k (E) = \mu (\varphi_N (E \cap I_k)) \quad \text{for all Borel sets} \; E \subset \T \; .
\]
The measure $N^* \mu \in M (\T)$ is defined by the formula
\[
 N^* \mu = \frac{1}{N} \sum^{N-1}_{k=0} \mu_k \; .
\]
Note that $N^* \mu$ is in $M^+ (\T) , M (\T)_{\sing}$ if and only if this is the case for $\mu$. For a measurable function $\alpha : \T \to \C$ or a holomorphic function $\alpha \in \Oh (D)$ the function $N_* \alpha$ on $\T$ resp. $D$ defined by
\[
 (N_* \alpha) (\eta) = \frac{1}{N} \sum_{\zeta^N = 1} \alpha (\zeta \eta^{1/N})
\]
is well defined and again measurable resp. holomorphic. We have $(N_1 N_2)_* \alpha = N_{1*} (N_{2*} \alpha)$ for all $N_1 , N_2 \ge 1$. If $\mu \ge 0$ and $\alpha \ge 0$ or if $\mu \in M (\T)$ and $N_* |\alpha|$ is $|\mu|$-integrable the following formula follows from the definitions:
\begin{equation} \label{eq:437}
 \int_{\T} \alpha d (N^* \mu) = \int_{\T} (N_* \alpha) d \mu \; .
\end{equation}
An immediate calculation using \eqref{eq:437} and the formula
\[
 N_* (\alpha N^* \beta) = (N_* \alpha) \beta
\]
for measurable functions $\alpha , \beta : \T \to \C$ gives the ``projection formula''
\begin{equation} \label{eq:438}
 N_* (\alpha N^* \mu) = (N_* \alpha) \mu \; .
\end{equation}
It is valid under the conditions on $\alpha$ and $\mu$ in equation \eqref{eq:437}. In the next proposition we set $\arg \eta = \theta$ if $\eta = e^{i\theta} \in \T$ with $0 \le \theta < 2 \pi$. 

\begin{prop}
 \label{t4.10}
For $\mu \in M (\T)$ and $\zeta \in \T$ we have the formulas:
\begin{align}
h_{N_* \mu} & = N_* h_{\mu} \label{eq:4.39nn}\\
 h_{N^* \mu} & =  N^* h_{\mu}  \label{eq:4.36}\\ 
N_* N^* \mu & =  \mu \label{eq:4.37}\\ 
N^* N_* \mu & = \frac{1}{N} \Tr_N (\mu) \quad \text{where} \; \Tr_N (\mu) = \sum_{\zeta^N = 1} \zeta_* \mu \label{eq:4.38}\\ 
c_{\nu} (N^* \mu) & = c_{\nu / N} (\mu) \quad\text{if} \; N \tei \nu \; \quad \text{and} \; c_{\nu} (N^* \mu) = 0 \; \text{if} \; N \not| \nu \label{eq:4.39}\\ 
N^* M_* \mu & = M_* N^* \mu \quad \text{if} \; N, M \ge 1 \; \text{are coprime} \; .\label{eq:4.40}\\
\widehat{N_* \mu} (\eta) & = \sum_{\zeta^N=1} \hmu (\zeta \eta^{1/N}) - c_0 \quad \text{where}\; c_0 = \sum_{\zeta^N =1} \hmu (\zeta) \label{eq:48a} \\ 
\widehat{N^* \mu} (\eta) & = \frac{1}{N} [N \frac{\arg\eta}{2\pi} ] \mu (\T) + \frac{1}{N} \hmu (\eta^N) \quad \text{for} \; \eta \in \T \label{eq:4.41}\\
 \widehat{N^* \mu} & = \frac{1}{N} N^* \hmu \quad \text{if} \; \mu (\T) = 0 \; .\label{eq:4.42} \\
\widehat{\zeta^{-1}_* \mu} & = \zeta^* \hmu - (\zeta^* \hmu) (1) \quad \text{if} \; \mu (\T) = 0 \label{eq:4.42a} \\
\widehat{\Tr_N (\mu)} & = \Tr_N (\hmu) - c_0 \quad \text{where}\; c_0 = \sum_{\zeta^N = 1} \hmu (\zeta) \label{eq:4.42b}
\end{align}
Finally:
\begin{align}
G_{N^* \mu} & = \frac{1}{N} N^* G_{\mu} \label{eq:4.43} \\
G_{N_* \mu} & = N (N_* G_{\mu}) \label{eq:4.49}
\end{align}

\end{prop}

\begin{proof}
 Equation \eqref{eq:4.39nn} is a restatment of \eqref{eq:3.1}. Equation \eqref{eq:4.36} follows by the calculation in the proof of \cite{LH} theorem 2.1 or alternatively using formulas \eqref{eq:3.4} and \eqref{eq:437}. Formula \eqref{eq:4.37} is a special case of \eqref{eq:438}. Formula \eqref{eq:4.38} follows from \eqref{eq:4.36} and \eqref{eq:3.1} or by a short calculation. Formula \eqref{eq:4.39} is a consequence of \eqref{eq:437}. Equation \eqref{eq:4.40} follows from \eqref{eq:3.7} and \eqref{eq:4.39} or directly. Formula \eqref{eq:48a} is a restatement of \eqref{eq:4.35}. As for \eqref{eq:4.41}, we argue as follows: For $0 \le \theta < 2 \pi$ we have:
\begin{eqnarray*}
 \widehat{N^* \mu} (\theta) & = & N^* \mu [0,\theta) = \frac{1}{N} \sum^{N-1}_{k=0} \mu (N (I_k \cap [0,\theta))) \\
& = & \frac{1}{N} \sum_{0 \le k < [N \frac{\theta}{2\pi}]} \mu \Big( N \Big[ \frac{2\pi k}{N} , \frac{2 \pi (k+1)}{N} \Big) \Big) + \frac{1}{N} \mu \Big( N \Big[ \frac{2\pi}{N} \Big[ N \frac{\theta}{2\pi} \Big] , \theta\Big)\Big)\\
& = & \frac{1}{N} \Big[ N \frac{\theta}{2\pi} \Big] \mu (\T) + \frac{1}{N} \mu [0, \langle N \theta \rangle) \quad \text{where} \; \langle \alpha \rangle = \alpha - 2\pi \Big[ \frac{\alpha}{2\pi} \Big]\; .
\end{eqnarray*}
This implies equations \eqref{eq:4.41} and \eqref{eq:4.42}. Formula \eqref{eq:4.42a} follows from the definitions. It implies \eqref{eq:4.42b}. Equations \eqref{eq:4.43} and \eqref{eq:4.49} follow from the definition of $G_{\mu}$ using \eqref{eq:4.39nn} and \eqref{eq:4.36} and the equality $(N^* \mu) (\T) = \mu (\T)$ which follows from \eqref{eq:4.37} for example.
\end{proof}

\begin{remarks} \label{t412}
\rm a) If $N^*\mu = \mu$ for some $N \ge 2$ we have $h_{\mu} (z^N) = h_{\mu} (z)$ and hence
\[
 h_{\mu} (z) = \lim_{\nu \to\infty} h_{\mu} (z^{N^{\nu}}) = h_{\mu} (0) = \mu (\T) = h_{\mu (\T) \lambda} (z) \; .
\]
It follows that $\mu = \mu (\T) \lambda$. This conclusion also follows from a consideration of Fourier-coefficients using formula \eqref{eq:4.39}. \\[0.2cm]
b) For $\mu \in M (\T)$ let $\Gh \subset \T$ be a Borel set with $\mu (B) = \mu (B \cap \Gh)$ for all Borel sets $B$. Then for the Borel sets $\varphi_N (\Gh)$ and $\varphi^{-1}_N (\Gh)$ and all $B$ we have
\[
 (N_* \mu) (B) = (N_* \mu) (B \cap \varphi_N (\Gh)) \quad \text{and} \quad (N^* \mu) (B) = (N^* \mu) (B \cap \varphi^{-1}_N (\Gh)) \; .
\]
c) If $\mu \in M^+ (\T)$ satisfies $N_* \mu = \mu$ we have $N^* \mu = N^* N_* \mu = N^{-1} \Tr_N (\mu)$ and hence $\mu \ll N^* \mu$. Thus there is a density $\alpha$ such that $\mu = \alpha N^* \mu$. This $\alpha$ corresponds to $S'$ in \cite{P}~\S\,1 where $S = \varphi_N$. Moreover $w \mapsto N_* (\alpha \omega)$ is the Perron--Frobenius operator denoted by $L_{-f}$ in loc.~cit. up to the factor $-N$. The density $\alpha$ is quite useful in the ergodic study of $\varphi_N , \varphi_M$. I do not see how to phrase it or at least its existence in terms of the Herglotz transform.
\end{remarks}
\section{Simultanous functional equations}
\label{sec:5}

Consider pairwise coprime integers $N_1 , \ldots , N_s$ with $s \ge 1$ and all $N_i \ge 2$. They generate the multiplicative monoid $\Sh$ of numbers $N^{\nu_1}_1 \cdots N^{\nu_s}_s$ with all $\nu_i \ge 0$. Set $\Oh = \Oh (D)$. 

For a left $\Sh$-module $\Mh$ we set
\[
 H^0 (\Sh , \Mh) = \{ m \in \Mh \tei \bs m = m \; \text{for all} \; \bs \in \Sh \}
\]
and
\[
 Z (\Sh , \Mh) = \{ m \in \Mh \tei \bs m = 0 \; \text{for all} \; \bs \in \Sh , \bs \neq 1 \} \; .
\]
For a right $\Sh$-module $\Mh$ we let $\Mh\Sh$ be the submodule of $\Mh$ generated by the elements $m\bs$ for $m \in \Mh$ and $\bs \in \Sh , \bs \neq 1$. Then we have
\[
 \Mh\Sh = N^*_1 (\Sh) + \ldots + N^*_s (\Sh)
\]
if we write $N^*$ for right operation with $N \in \Sh$. If the group law of $\Mh$ is written multiplicatively we also write $\Mh^{\Sh}$ for $\Mh\Sh$ etc.

We now introduce some useful operations. For $\zeta \in \T$ let $[\zeta]$ be the rotation of $\T$ defined by $[\zeta] \eta = \zeta \eta$. The maps $[\zeta]$ and $\varphi_N$ for $N \ge 2$ generate a submonoid $S$ of the monoid of self maps of $\T$ or of $D$ under composition. The only relations are the following
\[
 \varphi_N \varphi_M = \varphi_{NM} \; , \; [\zeta][\eta] = [\zeta\eta] \; , \; \varphi_N [\zeta] = [\zeta^N] \varphi_N \; .
\]
Identifying $N_i$ with $X_i = \varphi_{N_i}$ we will view $\Sh$ as a submonoid of $S$. The monoid $S$ acts from the right on functions on $\T$ (resp. $D$) by pullback. It also acts from the left on $\Oh^{\times}$ by the formulas
\begin{align*}
 ([\eta] f)(z) = & f (\eta^{-1} z) \quad \text{and} \\
(\varphi_N f) (z) = & \Big( \prod_{\zeta^N = 1} f (\zeta z^{1/N}) \Big)^{1/N} \quad \text{with} \; (\varphi_N f) (0) = f (0)  \; .
\end{align*}
The product $\prod_{\zeta^N = 1} f (\zeta z^{1/N})$ is independent of the choice of an $N$-th root of $z$ and it defines an analytic function on $D$. Its value at $z = 0$ is given by $f (0) ^N$ and for the definition of $\varphi_N f$ we take the unique branch of the $N$-th root for which $(\varphi_N f) (0) = f (0)$. Usually we will write $N_* f = \varphi_N f$. The definition of $N_*$ on $\Oh^{\times}$ is compatible via the exponential map with the definition of $N_*$ on $\Oh$ given before proposition \ref{t4.10}. With these notations we have for example
\[
 H^0 (\Sh , \Oh^{\times}) = \{ f \in \Oh^{\times} \tei f (z^N)^N = \prod_{\zeta^N = 1} f (\zeta z) \quad \text{for} \; N \in \Sh \}
\]
and
\[
 Z (\Sh , \Oh^{\times}) = \{ u \in \Oh^{\times} \tei \prod_{\zeta^N = 1} u (\zeta z) = 1 \quad \text{for} \; N \in \Sh , N \neq 1 \} \; .
\]
The monoid $S$ acts from the right on $M (\T)$ by the formulas:
\[
 \mu [\eta] = [\eta^{-1}]_* \mu \quad \text{and} \quad \mu \varphi_N = N^* \mu \; .
\]
It acts from the left by push foreward. The natural map
\[
 M (\T) \longrightarrow \Oh^{\times} \; , \; \mu \longmapsto f_{\mu} = \exp (-h_{\mu})
\]
is left- and right $S$-equivariant by formulas \eqref{eq:4.39nn} and \eqref{eq:4.36}. Let $\Rh = \Z S$ be the semigroup ring of $S$. The left and right actions of $S$ on functions and measures extend to left and right $\Rh$-actions. For $f \in \Oh^{\times}$ and $r \in \Rh$ we usually write $f^r$ for $f\cdot r$. We set $\Rh_{\Q} = \Q S$ and $\Rh_{\R} = \R S$.

The polynomial ring $\Z \Sh = \Z [X_1 , \ldots , X_s]$ is contained in $\Rh$. The following elements of $\Rh$ and $\Rh_{\Q}$ will be useful:
\begin{align}
 \Phi_{\Sh} & = \prod^s_{i=1} (1-X_i) \label{eq:5.1} \\
\Tr_N & = \sum_{\zeta^N =1} [\zeta] \quad \text{and} \; e_N = N^{-1} \Tr_N \label{eq:5.2}\\
\Omega_{\Sh} & = \prod^s_{i=1} (1 - e_i) \quad \text{where} \; e_i = N^{-1}_i \Tr_{N_i} \quad \text{for} \; 1 \le i \le s \label{eq:5.3}
\end{align}
Note that the $e_i$ are commuting idempotents in $\Rh_{\Q}$ for $1 \le i \le s$. We have the relation
\begin{equation}
 \label{eq:5.4}
\Phi_{\Sh} e_i = (e_i - X_i) \prod_{j\neq i} (1 - X_j) \quad \text{in} \; \Rh \; .
\end{equation}
We also need the element
\begin{equation}
 \label{eq:5.5}
\Psi_{\Sh} = \sum_{N \in \Sh} \varphi_N = \sum^{\infty}_{\nu_1 = 0} \cdots \sum^{\infty}_{\nu_s=0} X^{\nu_1}_1 \cdots X^{\nu_s}_s \quad \text{in}\; \Z [[X_1 , \ldots , X_s]]
\end{equation}
and the relations
\begin{equation}
 \label{eq:5.6}
\Phi_{\Sh} \Psi_{\Sh} = 1 = \Psi_{\Sh} \Phi_{\Sh} \quad \text{in} \; \Z [[X_1 , \ldots , X_s]] \; .
\end{equation}
In some cases $\Psi_{\Sh}$ acts on functions:

\begin{lemma}
 \label{t51}
For any $\alpha \in \Oh^{\times}$ with $\alpha (0) = 1$ the infinite product
\[
 \alpha^{\Psi_{\Sh}} := \prod_{N \in \Sh} N^* \alpha
\]
converges locally uniformly to a function in $\Oh^{\times}$.
\end{lemma}

\begin{rem}
 More explicitly, for $z \in D$ we have
\[
 \alpha^{\Psi_{\Sh}} (z) = \prod_{\nu_1 , \ldots , \nu_s} \alpha (z^{N^{\nu_1}_1 \cdots N^{\nu_s}_s})
\]
where $\nu_1 , \ldots , \nu_s$ run over the non-negative integers.
\end{rem}

\begin{proof}
 Fix any $0 < r < 1$. Because of $\alpha (0) = 1$ there is a constant $c = c_r \ge 0$ such that we have
\[
 |\alpha (z) - 1| \le c |z| \quad \text{in} \; |z| \le r \; .
\]
Since the series
\[
 \sum_{\nu_1 , \ldots , \nu_s} r^{N^{\nu_1}_1 \cdots N^{\nu_s}_s} \le \sum_{k \ge 1} r^k
\]
converges, it follows that the product for $\alpha^{\Psi_{\Sh}}$ converges absolutely and uniformly in $|z| \le r$.
\end{proof}

Let $\Mh$ be a left- and right $\Rh_{\Q}$-module such that for all $N \ge 1$ the relations
\[
 N_* N^* = \id \quad \text{and} \quad N^* N_* = e_N
\]
hold, where $N_*$ and $N^*$ denote left and right multiplication with $\varphi_N$ and where left and right multiplication by $e_N$ agree on $\Mh$. Then it follows that 
\[
 N^* (\Mh) = e_N \Mh = \Mh e_N \quad \text{for all} \; N \ge 1
\]
and in particular that
\begin{equation}
 \label{eq:57n}
\Mh \Sh = \Mh e_1 + \ldots + \Mh e_s = e_1 \Mh + \ldots + e_s \Mh \; .
\end{equation}
Moreover we have the following alternative descriptions of $H^0 (\Sh , \_)$ and $Z (\Sh , \_)$ in terms of the right action of $\Sh$
\begin{align}
 H^0 ( \Sh , \Mh) & = \{ m \in \Mh \tei m \varphi_N = m e_N \quad \text{for all} \; N \in \Sh \} \label{eq:559n} \\
Z (\Sh , \Mh) & = \{ m \in \Mh \tei m e_N = 0 \quad \text{for all} \; N \in \Sh , N \neq 1 \} \label{eq:558n} 
\end{align}
The conditions in the descriptions \eqref{eq:558n} and \eqref{eq:559n} have to be checked only for $N = N_i$ with $i = 1 , \ldots , s$.

\begin{prop} \label{t15n}
 For a left and right $\Rh_{\Q}$-module $\Mh$ as above, the map
\[
 \pi : Z (\Sh , \Mh) \longrightarrow \Mh / \Mh \Sh \; , \; \pi (m) = m + \Mh \Sh
\]
is an isomorphism. The inverse of $\pi$ is given by the map
\[
 \Omega_{\Sh} : \Mh / \Mh \Sh \longrightarrow Z (\Sh , \Mh) \; , \; \Omega_{\Sh} (m + \Mh\Sh) =  m \Omega_{\Sh} = \Omega_{\Sh} m \; .
\]
\end{prop}

\begin{proof}
 The relation $\Omega_{\Sh} e_i = 0$ for $i = 1 , \ldots , s$ where $e_i = e_{N_i}$ shows that $m \Omega_{\Sh} \in Z (\Sh , \Mh)$ for $m \in \Mh$. The relation $e_i \Omega_{\Sh} = 0$ implies that $(\Mh\Sh) \Omega_{\Sh} = 0$. Hence right multiplication by $\Omega_{\Sh}$ induces a well defined map
\[
 \Omega_{\Sh} : \Mh / \Mh \Sh \longrightarrow Z (\Sh , \Mh) \; .
\]
It remains to show that it is inverse to $\pi$. For $m \in Z (\Sh , \Mh)$ we have $m e_i = 0$, hence $m (1-e_i) = m$ and hence $m  \Omega_{\Sh} =m$. This gives $\Omega_{\Sh} \verk \pi = \id$. Write
\[
 \Omega_{\Sh} = 1 + \sum^s_{i=1} r_i e_i \quad \text{with} \; r_i \in \Rh_{\Q} \; .
\]
For $m \in \Mh$ this implies:
\[
 m \Omega_{\Sh} -m = \sum^s_{i=1} (mr_i) e_i \in \Mh e_1 + \ldots + \Mh e_s = \Mh \Sh \; .
\]
Hence we have $\pi \verk \Omega_{\Sh} = \id$ as well.
\end{proof}

These facts apply in particular to the following uniquely divisible groups $\Mh$ of functions $\Ih \subset \, \Nh^1 \subset \Oh^1 = \{ u \in \Oh^{\times} \tei u (0) =1 \}$ where $\Nh^1 = \Nh^{\times} \cap \Oh^1$ and
\[
 \Ih = \{ \frac{\alpha}{\alpha  (0)} \tei \alpha = \frac{\alpha_1}{\alpha_2} \; \text{with singular inner functions} \; \alpha_i \} 
\]

\begin{prop}
 \label{t52}
The following diagram of injections and isomorphisms is commutative:
\[
 \xymatrix{
H^0 (\Sh, \Ih) \ar@{^{(}->}[r]^{\Phi_{\Sh}} \ar[d]^{\wr} & Z (\Sh , \Ih) \ar[r]^{\overset{\pi}{\sim}} \ar@{_{(}->}[d] & \Ih / \Ih^{\Sh} \ar@{_{(}->}[d] \\
H^0 (\Sh , \Nh^1) \ar@{^{(}->}[r]^{\Phi_{\Sh}} \ar@{_{(}->}[d] & Z (\Sh ,  \Nh^1) \ar[r]^{\overset{\pi}{\sim}} \ar@{_{(}->}[d] & \Nh^1 / \Nh^{1\Sh} \ar@{_{(}->}[d] \\
H^0 (\Sh , \Oh^1)  \ar@<1.5ex>[r]^{\overset{\Phi_{\Sh}}{\sim}} & Z (\Sh , \Oh^1) \ar@<1.5ex>[l]_-{\sim}^-{\Psi_{\Sh}} \ar[r]^{\overset{\pi}{\sim}} & \Oh^1 / \Oh^{1\Sh}
}
\]
Here the maps are defined as follows: $\Phi_{\Sh} (f) = f^{\Phi_{\Sh}}$ and $\Psi_{\Sh} (\alpha) = \alpha^{\Psi_{\Sh}}$ where we use lemma \ref{t51}. The maps $\pi$ are induced by inclusions. We have
\[
 \pi^{-1} ([u]) = u^{\Omega_{\Sh}} \quad \text{for} \; u \in \Oh^1 , \Nh^1 , \Ih \; .
\]
\end{prop}

\begin{rems}
1) For $s = 1,2,3$ the map $\Phi_{\Sh}$ looks as follows explicitely:
\[
 \Phi_{\Sh} (f) = \frac{f (z)}{f (z^{N_1})} \quad \text{resp.} \quad \Phi_{\Sh} (f) = \frac{f (z^{N_1 N_2}) f (z)}{f (z^{N_1}) f (z^{N_2})}
\]
resp.
\[
 \Phi_{\Sh} (f) = \frac{f (z^{N_1 N_2}) f (z^{N_1 N_3}) f (z^{N_2 N_3}) f(z)}{f (z^{N_1}) f (z^{N_2}) f (z^{N_3}) f (z^{N_1 N_2 N_3})} \; .
\]
2) For $f \in H^0 (\Sh , \Oh^1)$ we have $\Phi_{\Sh} (f) = \Omega_{\Sh} (f)$. Hence the composition
\[
 \pi \verk \Phi_{\Sh} : H^0 (\Sh , \Oh^1) \silo \Oh^1 / \Oh^{1\Sh}
\]
maps $f$ to $f  \mod \Oh^{1\Sh}$.

3) The map $P_{\Sh} = \Psi_{\Sh} \verk \Omega_{\Sh} : \Oh^1 \to \Oh^1$ is a projector with image $H^0 (\Sh, \Oh^1)$ and kernel $\Oh^{1\Sh}$. We thus get a canonical decomposition $\Oh^1 = \Oh^{1\Sh} \times H^0 (\Sh, \Oh^1)$. However there is no such decomposition if we replace $\Oh^1$ by $\Ih$.
\end{rems}

\begin{proof}
Using lemma \ref{t51} and the formal relations \eqref{eq:5.6} it follows that the maps $\Phi_{\Sh} (f) = f^{\Phi_{\Sh}}$ and $\Psi_{\Sh} (\alpha) = \alpha^{\Psi_{\Sh}}$ define mutually inverse automorphisms of $\Oh^1$. Note here that for $h \in \Oh^1$ we have:
\[
h^{X^{\nu_1}_1 \ldots X^{\nu_s}_s} \longrightarrow 1 \quad \text{locally uniformly as} \; |\nu| = \nu_1 + \ldots + \nu_s \longrightarrow \infty \; .
\]
Since $\Oh^1$ is an $\Rh_{\Q}$-module relation \eqref{eq:5.4} holds on $\Oh^1$:
\[
\Phi_{\Sh} e_i = (e_i - X_i) \prod_{j \neq i} (1 - X_j) \; .
\]
For $f \in H^0 (\Sh , \Oh^1)$ we therefore get
\[
\Phi_{\Sh} (f)^{e_i} = f^{\Phi_{\Sh} e_i} = 1 \quad \text{for} \; 1 \le i \le s \; .
\]
Hence $\Phi_{\Sh}$ maps $H^0 (\Sh , \Oh^1)$ to $Z (\Sh , \Oh^1)$. The relations $X_i e_i = X_i$ and $X_j e_i = e_i X_j$ for $i \neq j$ imply the formal relation:
\[
\Psi_{\Sh} (e_i - X_i) = e_i \Psi_{\Sh_i} \; .
\]
Here $\Sh_i \subset \Sh$ is generated by the numbers $N_1 , \ldots , N_s$ except for $N_i$. Using lemma \ref{t51} this leads to the formula
\[
\Psi_{\Sh} (\alpha)^{e_i - X_i} = \alpha^{\Psi_{\Sh} (e_i - X_i)} = \alpha^{e_i \Psi_{\Sh_i}}
\]
for any $\alpha \in \Oh^1$. For $\alpha \in Z (\Sh , \Oh^1)$ we therefore have
\[
\Psi_{\Sh} (\alpha)^{e_i - X_i} = 1 \quad \text{for} \; 1 \le i \le s
\]
and hence $\Psi_{\Sh}$ maps $Z (\Sh , \Oh^1)$ to $H^0 (\Sh , \Oh^1)$. It follows that $\Phi_{\Sh} : H^0 (\Sh , \Oh^1) \to Z (\Sh , \Oh^1)$ is an isomorphism with inverse $\Psi_{\Sh}$. It is clear that $\Phi_{\Sh}$ maps $H^0 (\Sh , \Gh)$ to $Z (\Sh, \Gh)$ for $\Gh = \Ih$ and $\Nh^1$. The injective map
\[
 H^0 (\Sh , \Ih) \silo H^0 (\Sh , \Nh^1)
\]
is an isomorphism because of theorem \ref{t33}. Hence we are done with the left hand side of the diagram. The right hand side is commutative by proposition \ref{t15n}. 
\end{proof}

\begin{cor}
 \label{t513}
Assume that $\alpha \in \Oh^{\times}$ with $\alpha (0) = 1$ satisfies the relation $\prod_{\zeta^N = 1} \alpha (\zeta z) = 1$ for some $N \ge 2$ and that the function
\[
 f(z) = \prod^{\infty}_{\nu = 0} \alpha (z^{N^{\nu}})
\]
is in the Nevanlinna class $\Nh$. Then $\alpha = u / u (0)$ where $u$ is a quotient of singular inner functions. The function $f$ satisfies the functional equation \eqref{eq:1.1}.
\end{cor}

\begin{proof}
 The assertion follows from theorem \ref{t33}. It is also a formal consequence of proposition \ref{t52} for $\Sh = \{ N^{\nu} \tei \nu \ge 0 \}$ the monoid generated by $N$ because $\alpha \in Z (\Sh , \Oh^1)$ and $f = \Psi_{\Sh} (\alpha)$. 
\end{proof}

\begin{example} \label{t514}
 For any $a \in \C^*$ the function
\[
 f (z) = \prod^{\infty}_{\nu = 0} \exp (a z^{N^{\nu}}) = \exp a \sum^{\infty}_{\nu =0} z^{N^{\nu}}
\]
satisfies the functional equation \eqref{eq:1.1} and $f \notin \Nh$. 
\end{example}

This follows from corollary \ref{t513} since $|\exp a z|$ is a non-constant function of $z \in \T$ if $a \neq 0$. 

We now discuss atoms. For a function $f \in \Oh^{\times}$ and $\eta \in \T$ we set
\[
 A (f) (\eta) = - \halb \lim_{r\to 1-} (1-r) \log |f (r\eta)|
\]
if the limit exists. We say that $A (f)$ exists if $A (f) (\eta)$ exists for every $\eta \in \T$. Note that for $\mu \in M (\T)$ we have by equation \eqref{eq:2.5}
\begin{equation}
 \label{eq:5.7}
A (f_{\mu}) (\eta) = \mu \{ \eta \} \; .
\end{equation}
In particular $A (f)$ exists for every $f \in \Nh^{\times}$. We call $f \in \Oh^{\times}$ atomless if $A (f)$ exists and equals zero. Set $\Phi^1_{\Sh} = \prod^s_{i=1} (1 - N^{-1}_i X_i) \in \Rh_{\Q}$.

\begin{prop}
 \label{t53}
a) For $f \in \Oh^{\times}$ assume that $A (f)$ exists. Then $A (f^{\Phi_{\Sh}})$ exists as well and we have
\[
 A (f^{\Phi_{\Sh}}) = A (f) \Phi^1_{\Sh} \; .
\]
In particular, if $f$ is atomless then $f^{\Phi_{\Sh}}$ is atomless as well.\\
b) For $\alpha \in \Oh^{\times}$ with $\alpha (0) = 1$ set $f = \alpha^{\Psi_{\Sh}}$ as in Lemma \ref{t51}. Assume that $A (f)$ exists and that we have
\begin{equation}
 \label{eq:5.8}
\lim_{N\to\infty\atop N \in \Sh} N^{-1} A (f) (\eta^N) = 0 \quad \text{for all} \; \eta \in \T \; .
\end{equation}
 Then $A(\alpha)$ exists as well and we have pointwise on $\T$:
\begin{equation}
 \label{eq:5.9}
A (f) = \lim_{k\to \infty} A (\alpha) \sum^{k-1}_{\nu_1=0} \cdots \sum^{k-1}_{\nu_s=0} \frac{X^{\nu_1}_1 \cdots X^{\nu_s}_s}{N^{\nu_1}_1 \cdots N^{\nu_s}_s} \; .
\end{equation}
In particular, under these assumptions, if $\alpha$ is atomless, then $\alpha^{\Psi_{\Sh}}$ is atomless as well.
\end{prop}

\begin{proof}
 a) Using the formula $\lim_{r\to 1-} (1-r)^{-1} (1-r^N) = N$ it follows that with $A (f)$ also $A (f^{X_i})$ exists and that for $\eta \in \T$ we have
\begin{equation}
 \label{eq:5.10}
A (f^{X_i}) (\eta) = \frac{1}{N_i} A (f) (\eta^{N_i}) \; .
\end{equation}
Hence we get
\[
 A (f^{1-X_i}) = A (f) \Big( 1 - \frac{X_i}{N_i} \Big)
\]
and the assertion follows by induction.\\
b) In the proof of proposition \ref{t52} it was shown that $\alpha = f^{\Phi_{\Sh}}$. Thus $A (\alpha)$ exists by a) and we have
\[
 A (\alpha) = A (f) \prod^s_{i=1} (1 - N^{-1}_i X_i) \; .
\]
This gives
\begin{align*}
 A (\alpha) \sum_{0 \le \nu_i < k \atop i = 1 , \ldots , s} \frac{X^{\nu_1}_1 \cdots X^{\nu_s}_s}{N^{\nu_1}_1 \cdots N^{\nu_s}_s} & = A (f) \prod^s_{i=1} \Big( 1 - \frac{X^k_i}{N^k_i} \Big) \\
& = A (f) + \sum_{\bfd \in \{ 0 , 1\}^s \atop \bfd \neq 0} (-1)^{|\bfd|} \bfN^{-\bfd k} A (f) \cdot \bfX^{\bfd k} \; .
\end{align*}
Here we have set $\bfN = (N_1 , \ldots , N_s)$ and $\bfX = (X_1 , \ldots  X_s)$ and used multiindex notation. Since for $k \to \infty$ the right hand side tends to $A (f)$ pointwise by assumption \eqref{eq:5.8} the assertions follow.
\end{proof}

Set $M^0 (\T) = \{ \mu \in M (\T) \tei \mu (\T) = 0 \}$ and
\begin{equation} \label{eq:63n}
 M^{\kappa_0} (\T) = \{ \mu - \mu (\T) \lambda \tei \mu \in M (\T) \; \text{singular to} \; \lambda \} \; .
\end{equation}
Then the map $\mu \mapsto f_{\mu}$ defines isomorphisms
\begin{equation} \label{eq:562n}
 M^0 (\T) \silo \Nh^1 \quad \text{and} \quad M^{\kappa_0} (\T) \silo \Ih \; .
\end{equation}
Let $M^0_{c} (\T)$ resp. $M^{\kappa_0}_c (\T)$ be the subspaces of atomless measures in $M^0 (\T)$ resp. $M^{\kappa_0} (\T)$. For functions $f$ in $\Nh^1$ the function $A (f)$ exists by equation \eqref{eq:5.7}. Let $\Nh^1_c$ resp. $\Ih_c$ be the subgroups of atomless functions in $\Nh^1$ resp. $\Ih$. By our definitions the isomorphisms \eqref{eq:562n} restrict to isomorphisms:
\begin{equation} \label{eq:563}
 M^0_{c} (\T) \silo \Nh^1_c \quad \text{and} \quad M^{\kappa_0}_c (\T) \silo \Ih_c \; .
\end{equation}

\begin{cor}
 \label{t54}
a) Under the map $\Phi_{\Sh} : H^0 (\Sh , \Nh^1) \hookrightarrow Z (\Sh , \Nh^1)$ an element $f \in H^0 (\Sh , \Nh^1)$ is atomless if and only if $\Phi_{\Sh} (f)$ is atomless.\\
b) The natural map
\[
 \pi : Z (\Sh , \Ih_c) \silo \Ih_c / \Ih^{\Sh}_c
\]
is an isomorphism with inverse induced by right multiplication with $\Omega_{\Sh}$. \\
The same assertion holds for $\Nh^1_c$ instead of $\Ih_c$.
\end{cor}

\begin{proof}
 a) For $f \in \Nh^1$ the function $A (f)$ exists and is bounded by \eqref{eq:5.7} and \eqref{eq:562n}. Hence condition \eqref{eq:5.8} is satisfied. The assertion follows from proposition \ref{t53} a), b).\\
b) $\Ih_c$ is a left- and right sub- $\Rh_{\Q}$-module of $\Ih$. Now one applies proposition \ref{t15n}. The argument in the case of $\Nh^1_c$ is the same.
\end{proof}

\section{Back to measures and on to premeasures}
\label{sec:6}
In this section we reinterpret and extend the preceeding considerations in terms of measures and cumulative mass functions. The study of the map $\Psi_{\Sh}$ in this context leads to premeasures and functions of bounded $\kappa$-variation in the sense of Korenblum \cite{K1}, \cite{K4}. Proposition \ref{t52} and Corollary \ref{t54} imply the following assertion:

\begin{cor}
\label{t61} 
Let $\Mh$ be either $M^0 (\T)$ or $M^{\kappa_0} (\T)$. The map $\Phi_{\Sh} : H^0 (\Sh , \Mh) \hookrightarrow Z (\Sh , \Mh)$ sending $\mu$ to
\[
 \Phi_{\Sh} (\mu) = \mu \Phi_{\Sh} = \prod^s_{i=1} (1 - N^*_i) \mu
\]
is injective. For $\mu \in H^0 (\Sh , \Mh)$ we also have $\mu \Phi_{\Sh} = \mu \Omega_{\Sh}$. A measure $\mu \in H^0 (\Sh , \Mh)$ is atomless if and only if $\Phi_{\Sh} (\mu)$ is atomless. The inclusion $Z (\Sh , \Mh) \hookrightarrow \Mh$ induces an isomorphism $Z (\Sh , \Mh) \silo \Mh / \Mh\Sh$ whose inverse map sends $\mu \mod \Mh\Sh$ to
\[
\mu \Omega_{\Sh} = \Omega_{\Sh} \mu = \prod^s_{i=1} (1-e_i) \mu \; .
\]
The subgroup $\Mh \Sh$ consists of the measures $\mu$ with $c_{\nu} (\mu) = 0$ if $N_i \nmid \nu$ for all $1 \le i \le s$.
\end{cor}

\begin{proof}
Only the last assertion still needs proof. A measure $\mu \in \Mh$ is in $\Mh \Sh$ if and only if $c_{\nu} (\Omega_{\Sh} \mu) = 0$ for all $\nu \in \Z$. We have $c_{\nu} ((1 - e_N) \mu) = c_{\nu} (\mu)$ if $N \nmid \nu$ and $c_{\nu} ((1 - e_N) \mu) = 0$ for $N \tei \nu$. Hence $c_{\nu} (\Omega_{\Sh} \mu) = c_{\nu} (\mu)$ if $N_i \nmid \nu$ for all $1 \le i \le s$ and $c_{\nu} (\Omega_{\Sh} \mu) = 0$ if $N_i \tei \nu$ for some index $i$. 
\end{proof}

\begin{rem}
 We have
\[
 Z (\Sh , M (\T)) = Z (\Sh , M^0 (\T))
\]
and
\[
 Z (\Sh , M (\T)_{\sing}) = Z (\Sh , M^0 (\T)_{\sing}) = Z (\Sh , M^{\kappa_0} (\T)) \; .
\]
\end{rem}

Now let $\sigma$ be a measure in $Z (\Sh , \Mh)$. The formal series
\begin{equation}
 \label{eq:6.54}
\mu = \sigma \Psi_{\Sh} = \sum_{N \in \Sh} N^* \sigma = \sum_{\nu_1 , \ldots , \nu_s \ge 0} (N^{\nu_1}_1)^* \cdots (N^{\nu_s}_s)^* \sigma
\end{equation}
will not converge to a measure in general. If it does, we have $N_{i*} \mu = \mu$ for all $1 \le i \le s$ by the following computation where we use that $N_{i*} \sigma = 0$:
\begin{align*}
 N_{i*} \mu & = N_{i*} \sum_{(N,N_i)=1} N^* \sigma + N_{i*} \sum_N (N N_i)^* \sigma \\
& \overset{\eqref{eq:4.40}}{=} \sum_{(N , N_i) = 1} N^* N_{i*} \sigma + \sum_N N_{i*} N^*_i N^* \sigma \\
& \overset{\eqref{eq:4.37}}{=} \sum_N N^* \sigma = \mu \; .
\end{align*}
In order to discuss the convergence of this series \eqref{eq:6.54} is is best to evaluate it on the intervalls $[0,\theta)$ i.e. to look at the corresponding series of cumulative mass functions. Recall the correspondence $M (\T) \silo V (\T) , \mu \mapsto \hmu$ between measures on $\T$ and left-continuous real-valued functions $\hmu$ of bounded variation on $[0, 2\pi) \equiv \R / 2 \pi \Z \equiv \T$ with $\hmu (0) = 0$ and for which $\hmu (2 \pi -)$ exists. Here $M^0 (\T)$ corresponds to the subspace $V^0 (\T)$ of functions $\hmu$ with $\lim_{\theta \to 2\pi -} \hmu (\theta) = 0$ or equivalently with $\hmu$ left-continuous when viewed as a function on $\T$. The space $M^{\kappa_0} (\T)$ corresponds to the space $V^{\kappa_0} (\T)$ of functions $\hmu$ in $V^0 (\T)$ with $\frac{d\hmu}{d\theta}$ constant $\lambda$-a.e. We now introduce a new right $\Rh$-modul structure denoted by $\hullet$ on the space of functions $\omega : \T \to \C$ which have left limits in every point $\eta \in \T$ by setting:
\begin{equation}
 \label{eq:6.55} 
\omega \hullet [\zeta] = \zeta^* \omega - (\zeta^* \omega) (1) \quad \text{for} \; \zeta \in \T
\end{equation}
and
\begin{equation}
 \label{eq:6.56}
\omega \hullet \varphi_N = N^{-1} N^* \omega + \frac{1}{N} \Big[ N \frac{\arg \eta}{2\pi} \Big] \omega (1-) \; .
\end{equation}
Note that we have
\begin{equation}
 \label{eq:6.57}
\omega \hullet \Tr_N = \Tr_N (\omega) - \sum_{\zeta^N = 1} \omega (\zeta) \; .
\end{equation}
Because of equations \eqref{eq:4.41} and \eqref{eq:4.42a} the space $V (\T)$ is an $\Rh$-submodule and the natural map $M (\T) \to V (\T)$ is an isomorphism of right $\Rh$-modules. For any right $\Rh$-submodule $V$ of $V (\T)$ we set c.f. \eqref{eq:558n}: 
\[
 Z (\Sh , V) = \{ \hmu \in V \tei \hmu \hullet e_N = 0 \quad \text{for all} \; N \in \Sh \setminus \{ 1 \} \} \; .
\]
The formalism implies that $Z (\Sh , V) \subset V^0 (\T) $. This can be seen directly as follows. The equation $\hmu \hullet \Tr_N = 0$ implies that
\[
\sum_{\zeta^N = 1} \hmu (\zeta-) - \hmu (\zeta) = 0 \; .
\]
Since functions $\hmu \in V (\T)$ are left-continuous in any point $\zeta \in \T , \zeta \neq 1$ this implies $\hmu (1-) = \hmu (1) = 0$. 

The space $V^0 (\T)$ is right $\Rh$-invariant and the operation simplifies somewhat. We have
\begin{equation}
 \label{eq:6.58}
\hsigma \hullet \varphi_N = N^{-1} N^* \hsigma \quad \text{for} \; N \ge 1 \; \text{and} \; \hsigma \in V^0 (\T) \; . 
\end{equation}
Under the isomorphism $M(\T) \silo V (\T)$ the subspace $Z (\Sh , M^{\kappa_0} (\T)) \subset M (\T)$ corresponds to the subspace $Z (\Sh , V^{\kappa_0} (\T))$. On $Z (\Sh , M^0 (\T))$ it was not clear how to make sense of the map $\Psi_{\Sh}$. On $Z (\Sh , V^0 (\T))$ and even on $V^0 (\T)$ the situation is more lucid. For $\hsigma \in V^0 (\T)$ which is in particular bounded, the series
\begin{equation}
 \label{eq:6.59}
\hsigma \hullet \Psi_{\Sh} = \sum_{N \in \Sh} \hsigma \hullet \varphi_N = \sum_{N \in \Sh} N^{-1} N^* (\hsigma)
\end{equation}
is absolutely and uniformly convergent. Since $\hsigma$ is left-continuous on $\T$ the function $\Psi_{\Sh} (\hsigma) = \hsigma \hullet \Psi_{\Sh}$ is therefore left-continuous on $\T$ as well. Since $\hsigma$ is of bounded variation, $\hsigma$ has right limits in every point of $\T$. By the uniform convergence the same is true for $\Psi_{\Sh} (\hsigma)$. Finally we have $\Psi_{\Sh} (\hsigma) (1) = 0$ for $1 \in \T$ since $\hsigma (1) = 0$. It follows that $\hmu = \Psi_{\Sh} (\hsigma)$ is the cumulative mass function of a premeasure $\mu$ on $\T$ (always with $\mu (\T) = 0$). Let us recall this notion following \cite{K1} \S\,2 and \cite{K4}. Let $\Kh$ be the set of all arcs i.e. connected subsets $C$ of $\T$. A function $\mu : \Kh \to \R$ is called a premeasure on $\T$ if the following conditions are satisfied:
\begin{compactenum}[a)]
 \item $\mu (\emptyset) = \mu (\T) = 0$
\item $\mu (C_1 \cup C_2) = \mu (C_1) + \mu (C_2)$ for all $C_1 , C_2 \in \Kh$ with $C_1 \cap C_2 = \emptyset$
\item $\lim_{\nu\to \infty} \mu (C_{\nu}) = 0$ \quad for every sequence $(C_{\nu})_{\nu \ge 1}$ of arcs $C_{\nu} \in \Kh$ with $C_1 \supset C_2 \supset \ldots $ and $\bigcap_{\nu \ge 1} C_{\nu} = \emptyset$.
\end{compactenum}
By finite additivity $\mu$ extends to a real valued function on the ring in $\T$ generated by the arcs. 

The vector space of premeasures on $\T$ is denoted by $P (\T)$. It contains $M^0 (\T)$. 

As for measures, the cumulative mass function $\hmu : [0, 2\pi) \to \R$ is defined by setting $\hmu (\theta) = \mu (C_{\theta})$ where $C_{\theta} = \exp i [0, \theta)$. In this way one obtains an isomorphism between $P (\T)$ and the space $V_{\pre} (\T)$ of functions $\hmu : [0, 2\pi) \to \R$ which satisfy the following conditions:
\begin{compactenum}[a)]
 \item $\hmu (0) = 0 = \lim_{\theta \to 2 \pi -} \hmu (\theta)$
\item $\hmu$ is left continuous on $[0,2\pi)$
\item $\hmu$ has right limits in every point of $[0,2\pi)$.
\end{compactenum}
When viewed as a function $\hmu : \T \to \R$ the conditions read as follows:
\begin{compactenum}[a)]
 \item $\hmu (1) = 0$
\item $\hmu$ is left continuous on $\T$
\item $\hmu$ has right limits in every point of $\T$.
\end{compactenum}

\begin{rems}
1. We have $\mu \{ \zeta \} = \hmu (\zeta +) - \hmu (\zeta)$ for every $\mu \in P (\T)$ and $\zeta \in \T$. Hence $\mu$ is atomless if and only if $\hmu : \T \to \R$ is continuous.\\
2. It is known that the functions on $\T$ which have left and right limits in every point are exactly the uniform limits of step functions. In particular they are bounded, a fact that is also easily seen directly. It follows that for every premeasure $\mu$ there is a constant $A \ge 0$ with $|\mu (C)| \le A$ for all arcs $C \subset \T$. One can take $A = 2 \sup_{\zeta \in \T} |\hmu (\zeta)|$.  
\end{rems}

The right and left $\Rh$-module structures on $M^0 (\T)$ extend to $P (\T)$. For $\mu \in P (\T)$ and arcs $C \subset \T$ we set for $N \ge 1 , \zeta \in \T$:
\begin{align*}
(\varphi_{N^*} \mu) (C) & = \mu (\varphi^{-1}_N (C))\\
([\zeta]_* \mu) (C) & = \mu ([\zeta]^{-1} (C))\\
\mu_k (C) & = \mu (\varphi_N (C \cap I_k)) \quad \text{for} \; 0 \le k \le N-1\\
\varphi^*_N \mu & = N^{-1} \sum^{N-1}_{k=0} \mu_k \; .
\end{align*}

Here $I_k = [\frac{2\pi k}{N} , \frac{2\pi (k+1)}{N})$ is considered as an arc of $\T$. In this way we obtain premeasures $\varphi_{N^*} \mu = N_* \mu , \varphi^*_N \mu = N^* \mu$ and $[\zeta]_* \mu = \zeta_* \mu$. The left and right $S$- and hence $\Rh$-module structures of $P (\T)$ are defined as before:
\begin{align*}
 \varphi_N \mu & = \varphi_{N^*} \mu \quad \text{resp.} \; \mu \varphi_N = \varphi^*_N \mu \\
[\zeta] \mu & = [\zeta]_* \mu \quad \text{resp.} \; \mu [\zeta] = [\zeta^{-1}]_* \mu \; .
\end{align*}

\begin{prop}
 \label{t62}
For $\mu \in P (\T)$ and $\eta \in \T$ we have the formulas:
\begin{align}
 N_* N^* \mu & = \mu \label{eq:6.61} \\
N^* N_* \mu & = \frac{1}{N} \Tr_N (\mu) \quad \text{where} \; \Tr_N (\mu) = \sum_{\zeta^N=1} [\zeta]_* \mu \label{eq:6.62}\\
N^* M_* \mu & = M_* N^* \mu \quad \text{if} \; N, M \ge 1 \; \text{are coprime} \label{eq:6.63}\\
\widehat{N^* \mu} & = \frac{1}{N} N^* \hmu \label{eq:6.64}\\
\widehat{N_* \mu} (\eta) & = \sum_{\zeta^N=1} \hmu (\zeta \eta^{1/N}) - c_0 \quad{where} \; c_0 = \sum_{\zeta^N = 1} \hmu (\zeta) \label{eq:6.65}\\
\widehat{\eta^{-1}_* \mu} & = \eta^* \hmu - (\eta^* \hmu) (1) \label{eq:6.66}\\
\widehat{\Tr_N (\mu)} & = \Tr_N \hmu - c_0 \label{eq:6.67}
\end{align}
\end{prop}
In particular the isomorphism $P (\T) \silo V_{\pre} (\T)$ mapping $\mu$ to $\hmu$ becomes a map of right $\Rh$-modules where $\Rh$ acts on $V_{\pre} (\T)$ by using the action $\hullet$ defined in \eqref{eq:6.55}, \eqref{eq:6.56}. The proofs are similar as for measures, c.f. proposition \ref{t4.10}.

We have seen above that for $\hsigma \in V^0 (\T)$ the function $\hmu = \hsigma \hullet \Psi_{\Sh}$ is in $V_{\pre} (\T)$. Let $\mu \in P (\T)$ be the premeasure corresponding to $\hmu$. Since the series in \eqref{eq:6.59} converge absolutely we see that for $C = C_{\theta}$ with $0 \le \theta < 2\pi$ we have:
\begin{equation}
 \label{eq:6.60}
\mu (C) = \sum_{N \in \Sh} (\sigma \varphi_N) (C) = \sum_{N \in \Sh} N^* (\sigma) (C) \; .
\end{equation}
Note here that
\[
(\hsigma \hullet \varphi_N) (\theta) = \widehat{\sigma \varphi_N} (\theta) = (\sigma \varphi_N) (C) \; .
\]
Taking differences and using $\mu (\T) = 0 = \sigma (\T)$ it follows that \eqref{eq:6.60} holds for all arcs $C = \exp i [\theta_1 , \theta_2)$ with $\theta_1 \le \theta_2$ as well. Using the uniform convergence in \eqref{eq:6.59} we get for any $\zeta \in \T$
\[
 \hmu (\zeta +) - \hmu (\zeta) = \sum_{N \in \Sh} (\hsigma \hullet \varphi_N) (\zeta +) - (\hsigma \hullet \varphi_N) (\zeta) \; .
\]
This means that \eqref{eq:6.60} holds for one point arcs $C = \{ \zeta \}$ as well. Using additivity we finally see that equation \eqref{eq:6.60} holds for all arcs $C \in \Kh$, the series being absolutely convergent. This gives a direct description of the premeasure $\mu =: \sigma \Psi_{\Sh} =: \Psi_{\Sh} (\sigma)$ corresponding to $\hsigma \hullet \Psi_{\Sh}$. 

The preceeding constructions for measures extend to premeasures. We have
\[
 P\Sh = N^*_1 P (\T) + \ldots + N^*_s P (\T) = P (\T) e_1 + \ldots + P (\T) e_s
\]
and
\begin{align*}
 Z (\Sh , P (\T)) & = \{ \mu \in P (\T) \tei N_* \mu = 0 \quad \text{for} \; N \in \Sh \setminus \{ 1 \} \} \\
& = \{ \mu \in P (\T) \tei \mu e_N = 0 \quad \text{for} \; N \in \Sh \setminus \{ 1 \} \} \; .
\end{align*}
As usual the conditions need to be checked for $N = N_1 , \ldots , N_s$ only. Next we consider
\begin{align*}
H^0 (\Sh , P (\T))  & = \{ \mu \in P (\T) \tei N_* \mu = \mu \quad \text{for all} \; N \in \Sh \} \\
& = \{ \mu \in P (\T) \tei \mu \varphi_N = \mu e_N \quad \text{for all} \; N \in \Sh \} \; .
\end{align*}
As before one shows that for $\sigma \in P (\T)$ the series
\begin{equation}
 \label{eq:6.68}
\Psi_{\Sh} (\hsigma) := \hsigma \hullet \Psi_{\Sh} := \sum_{N \in \Sh} \hsigma \hullet \varphi_N = \sum_{N \in \Sh} N^{-1} N^* (\hsigma)
\end{equation}
is absolutely and uniformly convergent and hence defines an element of $V_{\pre} (\T)$. It follows that the series of premeasures
\begin{equation}
 \label{eq:6.69}
\Psi_{\Sh} (\sigma) := \sigma \Psi_{\Sh} := \sum_{N \in \Sh} \sigma \varphi_N = \sum_{N \in \Sh} N^* (\sigma)
\end{equation}
converges absolutely on arcs $C \subset \T$ to the premeasure with $\widehat{\Psi_{\Sh} (\sigma)} = \Psi_{\Sh} (\hsigma)$. For $\mu \in P (\T)$ and $\hmu \in V_{\pre} (\T)$ we set $\Phi_{\Sh} (\mu) = \mu \Phi_{\Sh}$ resp. $\Phi_{\Sh} (\hmu) = \mu \hullet \Phi_{\Sh}$. Since the series \eqref{eq:6.68} and \eqref{eq:6.69} converge it follows from the formal relations \eqref{eq:5.6} that $\Phi_{\Sh}$ and $\Psi_{\Sh}$ define mutually inverse automorphisms of $V_{\pre} (\T)$ and hence of $P (\T)$. These automorphisms respect the $\Rh \otimes \R$-submodules $P_c (\T)$ of atomless premeasures and $V_{\pre , c} (\T)$ of continuous functions in $V_{\pre} (\T)$. Note that $P_c (\T) \silo V_{\pre , c} (\T)$ under the map $\mu \mapsto \hmu$, as observed before. Using similar arguments as in the proof of proposition \ref{t52} we get the following result where the index $c$ refers to ``atomless''.

\begin{prop}
 \label{t63}
The following diagrams of injections and isomorphisms are commutative
\[
 \xymatrix{
H^0 (\Sh , M^0 (\T)) \ar@{^{(}->}[r]^{\Phi_{\Sh}} \ar@{_{(}->}[d] & Z (\Sh , M^0 (\T))  \ar@{_{(}->}[d] \ar[r]^{\overset{\pi}{\sim}} & M^0 (\T) / M^0 (\T) \Sh \ar@{_{(}->}[d] \\
H^0 (\Sh , P (\T))  \ar@<1.5ex>[r]^{\overset{\Phi_{\Sh}}{\sim}} & Z (\Sh , P (\T)) \ar@<1.5ex>[l]_-{\sim}^-{\Psi_{\Sh}} \ar[r]^-{\overset{\pi}{\sim}} & P (\T) / P (\T) \Sh
}
\]
and
\[
\xymatrix{
H^0 (\Sh , M^0_c (\T)) \ar@{^{(}->}[r]^{\Phi_{\Sh}} \ar@{_{(}->}[d] & Z (\Sh , M^0_c (\T))  \ar@{_{(}->}[d] \ar[r]^{\overset{\pi}{\sim}} & M^0_c (\T) / M^0_c (\T) \Sh \ar@{_{(}->}[d] \\
H^0 (\Sh , P_c (\T))  \ar@<1.5ex>[r]^{\overset{\Phi_{\Sh}}{\sim}} & Z (\Sh , P_c (\T)) \ar@<1.5ex>[l]_-{\sim}^-{\Psi_{\Sh}} \ar[r]^-{\overset{\pi}{\sim}} & P_c (\T) / P_c (\T) \Sh
}
\]
Here $\pi$ denotes the map induced by inclusion. Its inverse is given by right (or left) multiplication with $\Omega_{\Sh} = \prod^s_{i=1} (1 - e_i) \in \Rh_{\Q}$. On $H^0 (\Sh , P (\T))$ we have $\Phi_{\Sh} = \Omega_{\Sh}$.
\end{prop}
 
The space $P_c (\T)$ corresponds to $V_{\pre , c} (\T)$ which consists of all continuous functions $\hmu : \T \to \R$ with $\hmu (1) = 0$. The $l^1$-sequences of complex numbers $(c_{\nu})_{\nu \in \Z}$ with $c_{\nu} = 0$ if $N_i \tei \nu$ for some $1 \le i \le s$ and with $\oc_{\nu} = c_{-\nu}$ for all $\nu \in \Z$ give rise via their Fourier series $\hmu (\eta) = \sum_{\nu \in \Z} c_{\nu} \eta^{\nu}$ to pairwise different functions in $V_{\pre , c} (\T) / V_{\pre , c}  (\T) \Sh$. Hence $P_c (\T) / P_c (\T) \Sh$ has uncountable dimension as a real vector space and it follows that there are very many atomless $\Sh$-invariant premeasures on $\T$. The premeasures obtained from {\it measures} by the map $\Psi_{\Sh}$ are not arbitrary however. Using complex analysis we will show that they belong to certain very interesting and restrictive classes introduced by Korenblum. 

\section{The Herglotz transform of a premeasure}
\label{sec:7}
Korenblum extended the notion of the Herglotz transform from measures to premeasures. We extend some of the preceeding formulas to this more general context.

For a distribution $T \in \Dh' (\T)$ on the circle $\T \equiv \R / 2 \pi \Z$ c.f. \cite{Sch} the Herglotz transform $h_T$ is defined by the formula
\[
 h_T (z) = T (k_z) \quad \text{where}\; k_z (\theta) = \frac{e^{i\theta} + z}{e^{i\theta} -z } \quad \text{and} \; z \in D \; .
\]
We have
\[
 h_T (z) = c_0 (T) + 2 \sum^{\infty}_{\nu = 1} c_{\nu} (T) z^{\nu}
\]
where $c_{\nu} (T) = T (e^{-i\nu \theta})$. It is known that for some $k \ge 0$ depending on $T$ we have $c_{\nu} (T) = O (|\nu|^k)$ as $|\nu| \to \infty$. Hence $h_T$ is analytic in $D$. Like for measures one introduces the analytic function in $D$ defined by
\[
 G_T (z) = \frac{1}{2 \pi i} \int^z_0 (h_T (w) - c_0 (T)) \frac{dw}{w} \; .
\]
A measure $\mu \in M (\T)$ defines a distribution $T_{\mu}$ of order zero on $\T$ by setting $T_{\mu} (\varphi) = \int_{\T} \varphi \, d\mu$ for any test function $\varphi \in \Dh (\T) = C^{\infty} (\T)$. It is clear that we have $h_{\mu} = h_{T\mu}$. Partial integration \eqref{eq:4.25} shows that we have:
\[
 T_{\mu} = 2 \pi T'_{\hmu \lambda} + \mu (\T) \delta_{2\pi} \quad \text{in} \; \Dh' (\R / \Z) \; .
\]
In particular, for $\mu \in M^0 (\T)$ we have $T_{\mu} = 2 \pi T'_{\hmu \lambda}$. Following Korenblum and Hayman \cite{HK} we define the distribution $T_{\mu}$ attached to a premeasure $\mu \in P (\T)$ by this formula
\[
 T_{\mu} := 2 \pi T'_{\hmu \lambda} \; .
\]
Note that $T_{\mu}$ has order $\le 1$. In particular the Fourier coefficients $c_{\nu} (\mu) := c_{\nu} (T_{\mu})$ satisfy the estimate $c_{\nu} (\mu) = O (|\nu|)$ as $|\nu| \to \infty$. The distribution $T_{\mu}$ determines $\mu$ uniquely because $T'_{\hmu \lambda} = 0$ implies that we have $(\hmu - c) \lambda = 0$ for some constant $c$. Hence $\hmu = c$ $\lambda$-a.e. on $\T$ and therefore $\hmu = c$ everywhere on $\T$ since $\hmu$ is left-continuous. It follows that $0 = \hmu (1) = c$ hence $\hmu = 0$ and therefore $\mu = 0$ as well. The Herglotz transform of $\mu \in P (\T)$ is defined by
\[
 h_{\mu} (z) := h_{T_{\mu}} (z) = -2\pi T_{\hmu \lambda} \Big( \frac{d}{d\theta} k_z (\theta) \Big) = 2 i z \int^{2\pi}_0 \frac{e^{i\theta}}{(e^{i\theta} -z)^2} \hmu (\theta) \, d\theta \; .
\]
This extends the definition for measures $\mu$ in $M^0 (\T)$. We have
\[
 G_{\mu} (z) := G_{T_{\mu}} (z) = \frac{1}{2\pi i} \int^z_0 h_{\mu} (w) \frac{dw}{w} \quad \text{for} \; \mu \in P (\T) \; .
\]
Note here that $h_{\mu} (0) = 0$. There are similar relations as in proposition \ref{t4.10} where the first formula is from \cite{HK}:

\begin{prop} \label{t71}
 a) For $\mu \in P (\T)$ we have the formula
\begin{equation}
 \label{eq:7.72}
h_{\hmu \lambda} = G_{\mu} + \mu_0 \quad \text{where} \; \mu_0 = \int_{\T} \hmu \, d\lambda \; .
 \end{equation}
b) For $\mu \in P (\T)$ we have $G_{\mu} \in H^2$ and
\begin{equation} \label{eq:7.73}
 \hmu = \RRe \tG_{\mu} + \mu_0 \quad \text{holds $\lambda$-a.e. on} \; \T \; .
\end{equation}
c) A premeasure $\mu \in P (\T)$ satisfies $N_* \mu = \mu$ for some $N \ge 1$ if and only if the following functional equation holds for some constant $c_0$
\begin{equation}
 \label{eq:7.74}
\hmu (\eta^N) = \sum_{\zeta^N = 1} \hmu (\zeta \eta) - c_0 \quad \text{for $\lambda$-a.a.} \; \eta \in \T \; .
\end{equation}
In this case \eqref{eq:7.74} holds for all $\eta \in \T$ and we have
\[
 c_0 = \sum_{\zeta^N = 1} \hmu (\zeta) = (N-1) \mu_0 \; .
\]
d) A premeasure $\mu \in P (\T)$ is $N$-invariant if and only if for all $z \in D$ we have
\begin{equation}
 \label{eq:7.75}
G_{\mu} (z^N) = \sum_{\zeta^N = 1} G_{\mu} (\zeta z) \; .
\end{equation}
e) For $\mu \in P (\T)$ and $N \ge 1$ we have
\begin{align}
h_{N_* \mu} & = N_* h_{\mu} \label{eq:7.78n}\\
 h_{N^* \mu} & = N^* h_{\mu} \label{eq:7.76} \\
N^* h_{N_* \mu} & = N^{-1} \Tr_N (h_{\mu}) \label{eq:7.77} \\
G_{N^* \mu} & = \frac{1}{N} N^* G_{\mu} \; . \label{eq:7.78}\\
G_{N_* \mu} & = N (N_* G_{\mu}) \label{eq:7.82n}
\end{align}
\end{prop}

\begin{proof}
 a) Since $h_{\hmu \lambda} (0) = \mu_0$ it suffices to show that we have
\[
 z h'_{\hmu \lambda} (z) = (2 \pi i)^{-1} h_{\mu} (z) \; .
\]
For $\nu \ge 1$ the $\nu$-th Taylor coefficient is given by $2 \nu c_{\nu} (\hmu \lambda)$ on the left and by $(\pi i)^{-1} c_{\nu} (T_{\mu})$ on the right. For any $T \in \Dh' (\T)$ we have $c_{\nu} (T') = i \nu c_{\nu} (T)$ for all $\nu \in \Z$. Hence we get
\begin{equation}
 \label{eq:7.79}
(\pi i)^{-1} c_{\nu} (T_{\mu}) = -2 i c_{\nu} (T'_{\hmu \lambda}) = 2 \nu c_{\nu} (T_{\hmu \lambda}) = 2 \nu c_{\nu} (\hmu \lambda) \; .
\end{equation}
b) We have
\begin{equation}
 \label{eq:7.80}
G_{\mu} (z) = \frac{1}{\pi i} \sum^{\infty}_{\nu=1} \frac{c_{\nu} (T_{\mu})}{\nu} z^{\nu}
\end{equation}
and 
\[
 \sum^{\infty}_{\nu = 1} \Big| \frac{c_{\nu} (T_{\mu})}{\nu} \Big|^2 = 4 \pi^2 \sum^{\infty}_{\nu = 1} |c_{\nu} (\hmu \lambda)|^2 \; .
\]
This series is convergent by Parseval's formula since $\hmu \in L^{\infty} \subset L^2$. Hence $G_{\mu} \in H^2$ and the boundary function $\tG_{\mu}$ exists $\lambda$-a.e. on $\T$. We have $c_{-\nu} (T_{\mu}) = \overline{c_{\nu} (T_{\mu})}$ since $\mu$ is real. Now we argue as in the proof of proposition \ref{t4.8} b).\\
c) It follows from \eqref{eq:6.65} that $N_* \mu = \mu$ implies equation \eqref{eq:7.74}. If on the other hand \eqref{eq:7.74} holds $\lambda$-a.e. for some $c_0$ then using \eqref{eq:6.65} we get
\[
\widehat{N_* \mu} (\eta^N) = \hmu (\eta^N) + c_0 - d_0 \quad \text{for $\lambda$-a.a.} \; \eta \in \T \; .
\]
Here we have set $d_0 = \sum_{\zeta^N=1} \hmu (\zeta)$. It follows that we have
\[
 \widehat{N_* \mu} = \hmu + c_0 - d_0 \quad \text{$\lambda$-a.e. on} \; \T \; .
\]
Since $\widehat{N_* \mu}$ and $\hmu$ are both left continuous this equality holds everywhere on $\T$. Evaluating at $\eta = 1$ we find $c_0 = d_0$ and hence $\widehat{N_* \mu} = \hmu$ i.e. $N_* \mu = \mu$.\\
d) Because of \eqref{eq:7.72}, relation \eqref{eq:7.75} is equivalent to
\begin{align*}
 h_{\hmu \lambda} (z^N) & = \sum_{\zeta^N = 1} h_{\hmu \lambda} (\zeta z) - c_0 \quad \text{where} \; c_0 = (N-1) \mu_0 \\
& = \Tr_N (h_{\hmu \lambda}) - c_0  \\
& \overset{\eqref{eq:3.1}}{=} N h_{N_* (\hmu \lambda)} (z^N) - c_0 = h_{N N_* (\hmu \lambda) - c_0 \lambda} (z^N) \; .
\end{align*}
Since the measures in question are real, this is equivalent to the relation
\begin{equation} \label{eq:7.81}
 \hmu \lambda = N N_* (\hmu \lambda) - c_0 \lambda \; .
\end{equation}
By \eqref{eq:438} we have $N N_* (\hmu \lambda) = \alpha \lambda$ where $\alpha (\eta) = \sum_{\zeta^N = 1} \hmu (\eta^{1/N} \zeta)$. 

Hence \eqref{eq:7.81} is equivalent to the formula
\[
 \hmu (\eta^N) = \sum_{\zeta^N = 1} \hmu (\zeta \eta) - c_0 \quad \text{for $\lambda$-a.a.} \; \eta \in \T \; ,
\]
and hence by part c) to $N_* \mu = \mu$. Alternatively, d) follows from equation \eqref{eq:7.82n}. \\
e) To show \eqref{eq:7.76} it is easiest to compare the Taylor coefficients of both sides. Thus we must show that we have $c_{\nu} (T_{N^* \mu}) = c_{\nu / N} (T_{\mu})$ for $\nu \ge 1$ where we set $c_{\alpha} (T) = 0$ for $\alpha \in \Q \setminus \Z$. We calculate:
\begin{align*}
 c_{\nu} (T_{N^* \mu}) & \overset{\eqref{eq:7.79}}{=} 2 \pi i \nu c_{\nu} (\widehat{N^* \mu} \lambda) \overset{\eqref{eq:6.64}}{=} 2 \pi i \frac{\nu}{N} c_{\nu} ((N^* \hmu) \lambda) \\
& = 2\pi i \frac{\nu}{N} c_{\nu / N} (\hmu \lambda) \overset{\eqref{eq:7.79}}{=} c_{\nu / N} (T_{\mu}) \; . 
\end{align*}
Equation \eqref{eq:7.78n} follows similarly using \eqref{eq:6.65}.

As for \eqref{eq:7.77} note the following:
\[
 N^* h_{N_* \mu} \overset{\eqref{eq:7.76}}{=} h_{N^* N_* \mu} \overset{\eqref{eq:6.62}}{=} h_{N^{-1} \Tr_N (\mu) } = N^{-1} \Tr_N (h_{\mu}) \; .
\]
Finally \eqref{eq:7.78}  and \eqref{eq:7.82n} follow from \eqref{eq:7.76} and \eqref{eq:7.78n}.
\end{proof}

We end this section with a discussion of atoms.

By formula \eqref{eq:2.5} it is possible to recover the atoms of a measure $\mu \in M (\T)$ from the Herglotz transform $h_{\mu}$. As we will show this is also possible if $\mu$ is a premeasure. The proof is more delicate though.

\begin{theorem} \label{t7.19}
 For $\mu \in P (\T)$ and every $\eta \in \T$ we have
\[
 \lim_{r\to 1-} \halb (1-r) \RRe h_{\mu} (r\eta) = \mu \{ \eta \} \; .
\]
\end{theorem}

\begin{rem}
 I do not know whether $\lim_{r\to 1-} (1-r) \Imm h_{\mu} (r\eta)$ exists in general for premeasures $\mu$.
\end{rem}

Since $\mu \{ \eta \} = \hmu (\eta +) - \hmu (\eta) = \hmu (\eta +) - \hmu (\eta-)$ the result follows from the next theorem.

\begin{theorem} \label{t7.20}
 Let $\hmu : \T \to \R$ be a function which has left and right limits in every point. Then the function
\[
 h (z) = 2 i z \int^{2\pi}_0 \frac{e^{i\theta}}{(e^{i\theta} - z)^2} \hmu (\theta) \, d\theta
\]
satisfies the limit formula
\begin{equation} \label{eq:7.82}
 \lim_{r\to 1-} \halb (1-r) \RRe h (r\eta) = \hmu (\eta +) - \hmu (\eta -) \quad \text{for every} \; \eta \in \T \; .
\end{equation}
\end{theorem}

\begin{proof}
 Setting $h = h^{\hmu}$ we have $\eta^* h = h^{\eta^* \hmu}$ for every $\eta \in \T$. Hence it suffices to prove \eqref{eq:7.82} for $\eta = 1$ i.e., viewing $\hmu$ as a $2\pi$-periodic function on $\R$, to show that we have
\begin{equation}
 \label{eq:7.83}
\lim_{r\to 1-} \halb (1-r) \RRe h (r) = \hmu (0+) - \hmu (0-) \; .
\end{equation}
Set 
\begin{equation}
 \label{eq:7.84}
\psi_r (\theta) = \RRe \frac{ri e^{i\theta}}{(e^{i\theta}-r)^2} = r (1-r^2) \frac{\sin \theta}{|e^{i\theta}-r|^4} \; .
\end{equation}
Then we have $\psi_r (-\theta) = -\psi_r (\theta)$ and hence:
\begin{align}
 \RRe h (r) & = 2 \int^{\pi}_{-\pi} \psi_r (\theta) \hmu (\theta) \, d\theta \notag \\
& = 2 \int^{\pi}_0 \psi_r (\theta) (\hmu (\theta) - \hmu (-\theta)) \, d \theta \; . \label{eq:7.85}
\end{align}
The relation
\[
 \psi_r (\theta) = - \halb \RRe \Big( \frac{d}{d\theta} \, \frac{e^{i\theta} +r}{e^{i\theta} -r} \Big)
\]
implies the formula
\begin{equation}
 \label{eq:7.86}
\int^{\pi}_0 \psi_r (\theta) \, d\theta = \frac{2r}{1-r^2} \; .
\end{equation}
Hence we get 
\begin{equation}
 \label{eq:7.87}
\frac{1-r^2}{4r} \RRe h (r) - (\hmu (0+) - \hmu (0-)) = \frac{1-r^2}{2r} \int^{\pi}_0 \psi_r (\theta) \varepsilon (\theta) \, d\theta
\end{equation}
where
\[
 \varepsilon (\theta) = (\hmu (\theta) - \hmu (0+)) - (\hmu (-\theta) - \hmu (0-)) \; .
\]
For every $0 < \delta < \pi$ we have
\begin{equation}
 \label{eq:7.88}
\lim_{r\to 1-} \frac{1-r^2}{2r} \int^{\pi}_{\delta} \psi_r (\theta) \varepsilon (\theta) \, d\theta = 0 \; .
\end{equation}
Now fix some $\varepsilon > 0$. Choose $0 < \delta < \pi$ such that
\[
 |\hmu (\pm \theta) - \hmu (0\pm)| < \frac{\varepsilon}{4} \quad \text{and hence} \; |\varepsilon (\theta)| < \frac{\varepsilon}{2} \; \text{for} \; 0 < \theta \le \delta \; .
\]
Because of \eqref{eq:7.88} we find some $0 \le r (\varepsilon) < 1$ (depending only on $\delta = \delta (\varepsilon)$) such that for $r (\varepsilon) \le r < 1$ we have:
\[
 \Big| \frac{1-r^2}{2r} \int^{\pi}_{\delta} \psi_r (\theta) \varepsilon (\theta) \Big| \le \frac{\varepsilon}{2} \; .
\]
Using \eqref{eq:7.87} and \eqref{eq:7.88} we get the estimate
\begin{align*}
 \Big| \frac{1-r^2}{4r} \RRe h (r) - (\hmu (0+) - \hmu (0-)) \Big| & \le \frac{1-r^2}{2r} \int^{\delta}_0 \psi_r (\theta) |\varepsilon (\theta)| \, d\theta + \frac{\varepsilon}{2} \\
& \le \frac{\varepsilon}{2} \frac{1-r^2}{2r} \int^{\pi}_0 \psi_r (\theta) \, d\theta + \frac{\varepsilon}{2} = \varepsilon \; .
\end{align*}
The important point here is that $\psi_r (\theta) \ge 0$ for $0 \le \theta \le \pi$. The corresponding assertion does not hold in a similar analysis of the limit behaviour of $\Imm h (r)$. This resembles the difference between the Fej\'er and the Dirichlet kernel in the theory of Fourier series. In any case, we have shown that:
\[
 \lim_{r\to1-} \frac{1-r^2}{4r} \RRe h (r) = \hmu (0+) - \hmu (0-) \; .
\]
Formula \eqref{eq:7.83} and hence the theorem follow. 
\end{proof}

\section{Background on Nevanlinna--Korenblum theory}
\label{sec:8}
In this section we review some of Korenblum's work on generalized Nevanlinna theory. For our purposes it is necessary to extend parts of \cite{K1}, \cite{K2} in a direction suggested by Seip \cite{S} \S\,1 Remark 3. We also reinterpret the $\kappa$-singular measures of Korenblum as pairs of ``$\kappa$-thin'' mutually singular possibly infinite positive measures on $\Bh (\T)$. 

In the following $\kappa$ will always denote a continuous nondecreasing concave function $\kappa$ on $[0, 1]$ with $\kappa (0) = 0$ and $\kappa (1) = 1$. For real $\alpha \ge 1$ and $0 \le x \le 1$ we have:
\begin{equation} \label{eq:890}
 \kappa \Big( \frac{x}{\alpha} \Big) = \kappa \Big( \frac{x}{\alpha} + \Big( 1 -\frac{1}{\alpha} \Big) 0 \Big) \ge \frac{1}{\alpha} \kappa (x) + \Big( 1 - \frac{1}{\alpha} \Big) \kappa (0) = \frac{1}{\alpha} \kappa (x) \, .
\end{equation}
Hence $\kappa (x) / x$ is nonincreasing and in particular $\kappa (x) \ge x$ since $\kappa (1) = 1$. Moreover
\begin{equation} \label{eq:893n}
 \kappa (x+y) \le \kappa (x) + \kappa (y) \quad \text{for} \; x,y \ge 0 \; \text{with} \; x+y \le 1
\end{equation}
because of \eqref{eq:890}
\[
 \kappa (x) + \kappa (y) \ge \frac{x}{x+y} \kappa (x+y) + \frac{y}{x+y} \kappa (x+y) = \kappa (x+y) \; .
\]
In \cite{K4} section 4, Korenblum introduced the notion of a premeasure on $\T$ with bounded $\kappa$-variation as follows. For an $\arc C \subset \T$ let $|C| = \lambda (C) \in [0, 1]$ be the normalized arc length of $C$. For $\mu \in P (\T)$ the $\kappa$-variation norm $\| \mu \|_{\kappa} = \Var_{\kappa} \mu$ is defined as
\[
 \| \mu \|_{\kappa} = \sup \frac{\sum_j |\mu (C_j)|}{\sum_j \kappa (|C_j|)} \; .
\]
Here the supremum is taken over all finite partitions $\{ C_j \}$ of $\T$ into disjoint arcs. If $\| \mu \|_{\kappa} < \infty$ the premeasure $\mu$ is said to have bounded $\kappa$-variation. For $\kappa (t) = t$ one obtains the usual notions of bounded variation and total variation norm. Note that since $\sum_j \kappa (|C_j|) \ge \sum_j |C_j| = 1$ we have 
\[
\|\mu \|_{\kappa} \le \sup \sum_j |\mu (C_j)| \; .
\]
 The premeasure $\mu$ may have bounded $\kappa$-variation even if its ordinary variation is infinite. Let $P_{\kappa} (\T)$ denote the set of all premeasures in $P (\T)$ with bounded $\kappa$-variation. It is a Banach space with norm $\| \; \|_{\kappa}$. The distribution $T_{\mu}$ allows the ``integration'' of smooth functions on $\T$ with respect to $\mu$. More generally, continuous functions ``of bounded $\kappa$-norm'' can be integrated along $\mu$ and $P_{\kappa} (\T)$ is the dual of this separable Banach space of functions, c.f. \cite{K4} theorem 5.1. 

A premeasure $\mu$ is called $\kappa$-bounded (from above) if there is a constant $a_{\mu} \ge 0$ such that
\[
 \mu (C) \le a_{\mu} \kappa (|C|) \quad \text{for all arcs} \; C \subset \T \; .
\]
A $\kappa$-bounded premeasure has $\kappa$-bounded variation and Korenblum has proved that every premeasure $\mu$ of bounded $\kappa$-variation can be written in the form $\mu = \mu_1 - \mu_2$ where $\mu_1$ and $\mu_2$ are $\kappa$-bounded premeasures. They can be chosen in such a way that we have $a_{\mu_1} , a_{\mu_2} \le \| \mu \|_{\kappa}$, see \cite{K4} p. 542. Let $P^-_{\kappa} (\T) \subset P_{\kappa} (\T)$ be the cone of $\kappa$-bounded premeasures and set $P^+_{\kappa} (\T) = - P^-_{\kappa} (\T)$.

For us the relevant convex functions $\kappa$ on $[0,1]$ are given by:
\[
 \kappa_{\gamma} (x) = c^{-1}_{\gamma} \int^x_0 |\log t|^{\gamma} \, dt \quad \text{for real} \; \gamma \ge 0 \; .
\]
Here
\[
 c_{\gamma} = \int^1_0 |\log t|^{\gamma} \, dt = \Gamma (\gamma+1) \; .
\]
We have $\kappa_{\gamma} (0) = 0$ and $\kappa_{\gamma} (1) = 1$. Moreover $\kappa_{\gamma}$ is continuous and nondecreasing on $[0,1]$. On $(0,1]$ it is $C^{\infty}$ and we have:
\[
 \kappa''_{\gamma} (x) = -\gamma c^{-1}_{\gamma} x^{-1} |\log x|^{\gamma-1} \le 0 \; .
\]
Thus $\kappa_{\gamma}$ is concave on $[0,1]$. For $0 < x \le 1$ there is some $0 < \xi_x < x$ such that we have
\[
 x^{-1} \kappa_{\gamma} (x) = c^{-1}_{\gamma} |\log \xi_x|^{\gamma} \; .
\]
For $\gamma > 0$ this implies condition b) on p. 531 of \cite{K4}
\[
 \lim_{x \to 0+} x^{-1} \kappa_{\gamma} (x) = + \infty \; .
\]
For $0 < x < 1$ we have:
\[
 \kappa'_{\gamma} (x^2) = c^{-1}_{\gamma} |\log x^2|^{\gamma} = 2^{\gamma} \kappa'_{\gamma} (x) \; .
\]
In particular $\kappa_{\gamma}$ is of type $(S)$ in the sense of \cite{K4} p. 542. For integers $\gamma \ge 1$ partial integration gives:
\[
 \kappa_{\gamma} (x) = \frac{x}{\gamma!} |\log x|^{\gamma} + \kappa_{\gamma-1} (x) \; .
\]
Hence for integer $\gamma \ge 0$ we have
\[
 \kappa_{\gamma} (x) = x \sum^{\gamma}_{\nu =0} \frac{1}{\nu!} |\log x|^{\nu} \; .
\]
For example $\kappa_0 (x) = x$ and $\kappa_1 (x) = x (1 + |\log x|) = x \log \frac{e}{x}$ the Shannon entropy function.

\begin{prop} \label{t26n}
 For $0 \le \delta < \gamma$ we have $\kappa_{\delta} (x) \le \kappa_{\gamma} (x)$ for $0 \le x \le 1$ with strict inequality for $0 < x < 1$.
\end{prop}

\begin{proof}
 It suffices to show that for $0 < x < 1$ we have:
\[
 \int^1_0 |\log t|^{\delta} \, dt \int^x_0 |\log t|^{\gamma} \, dt > \int^1_0 |\log t|^{\gamma} \, dt \int^x_0 |\log t|^{\delta} \, dt \; ,
\]
or equivalently 
\[
\int_{(|\log x| ,\infty) \times (0 , |\log x|)} e^{-(y_1 + y_2)} (y^{\gamma}_1 y^{\delta}_2 - y^{\delta}_1 y^{\gamma}_2) d\lambda (y_1 , y_2) > 0  \; . 
\]
On $U = (|\log x| , \infty) \times (0, |\log x|)$ we have $0 < y_2 < y_1$. Hence
\[
 y^{\gamma}_1 y^{\delta}_2 - y^{\delta}_1 y^{\gamma}_2 = y^{\gamma+\delta}_2 \Big( \big( \frac{y_1}{y_2} \big)^{\gamma} - \big( \frac{y_1}{y_2} \big)^{\delta} \Big)
\]
is positive and it follows that $\Dh > 0$.
\end{proof}

For $\gamma = 0$ the set $P^+_{\kappa_0} (\T)$ consists of all premeasures $\mu$ with $\mu (C) \ge -a_{\mu} |C|$ for some $a_{\mu} \ge 0$ and all arcs $C$. Hence the following map is surjective:
\begin{equation}
 \label{eq:8.89n}
\omega : M^+ (\T) \twoheadrightarrow P^+_{\kappa_0} (\T) \quad \text{with} \; \omega (\mu') = \mu' - \mu' (\T) \lambda\; . 
\end{equation}
Being real the premeasures $\mu = \omega (\mu') \in P^+_{\kappa_0} (\T)$ are in bijection with their Herglotz-transforms $h_{\mu} = h_{\mu'} - \mu' (\T)  = h_{\mu'} - h_{\mu'} (0)$. Mapping $\mu$ to $h_{\mu}$ we thus get a bijection:
\begin{equation}
 \label{eq:8.89}
P^+_{\kappa_0} (\T) \silo \{ h \in \Oh (D) \tei h(0) = 0 \quad \text{and} \quad \RRe h \ge c \; \text{for some} \; c \} \; .
\end{equation}
Note that $c = c_h \le 0$ since $h (0) = 0$ and $c < 0$ unless $h = 0$ since a non-constant holomorphic map is open.

In \cite{K4} 5.3 Theorem, Korenblum has generalized this bijection to $\kappa$-bounded premeasures for more general convex functions $\kappa$ than $\kappa_0$. The proof has not been published but it runs along the same lines at the one for $\kappa = \kappa_1$ which can be found in \cite{K1} 5.2 Theorem 3 or \cite{HKZ} Theorem 7.4. For the case $\kappa = \kappa_{\gamma}$ his result is stated below. For $\gamma \ge 0$ let $\Oh^-_{\kappa_{\gamma}}$ be the class of holomorphic functions $h \in \Oh = \Oh (D)$ with $h (0) = 0$ which satisfy an estimate of the form
\begin{equation}
 \label{eq:8.90}
\RRe h (z) \le c |\log (1-|z|)|^{\gamma} \quad \text{for} \; \halb \le |z| < 1 \; .
\end{equation}
 Here $c \in \R$ is a constant depending on $h$. Note that $\Oh_{\kappa_{\gamma}} = \Oh^-_{\kappa_{\gamma}} - \Oh^-_{\kappa_{\gamma}}$ is a real sub-vectorspace of $\Oh$.

\begin{theorem}[Korenblum] \label{t821}
For every $\gamma \ge 0$ the Herglotz transform $\mu \mapsto h_{\mu}$ gives a bijection $P^-_{\kappa_{\gamma}} (\T) \silo \Oh^-_{\kappa_{\gamma}}$ and an isomorphism of $\R$-vector spaces $P_{\kappa_{\gamma}} (\T) \silo \Oh_{\kappa_{\gamma}}$.
\end{theorem}

Formula \cite{K4} (5.3.5) gives a way to recover $\mu$ from $h_{\mu}$. A related possibility is to use formula \eqref{eq:7.73}.

Let $\Ah_{\gamma}$ be the $\C$-algebra of analytic functions $f \in \Oh$ which satisfy an estimate of the following form  where $a_f \ge 0$ and $r_f \ge 0$ are constants
\begin{equation}
 \label{eq:8.91}
|f (z)| \le a_f \exp (r_f |\log (1-|z|)|^{\gamma}) \quad \text{for}\; z \in D \; .
\end{equation}
Then $\Ah_0 = H^{\infty} (D)$ and $\Ah_1 = \Ah^{-\infty}$ is the class studied in \cite{K1}, \cite{K2}. Works dealing with $\Ah_{\gamma}$ are for example \cite{BL}, \cite{K4}, \cite{S}. Note that 
\[
\Ah^1_{\gamma} := \exp \Oh^-_{\kappa_{\gamma}} = \{ f \in \Ah_{\gamma} \tei f (0) = 1 \; \text{and} \; f (z) \neq 0 \; \text{for all} \; z \in D \} \; . 
\]
Let $\Nh_{\gamma}$ be the algebra of functions $f \in \Oh$ that can be written in the form $f = g_1 g^{-1}_2$ with $g_1 , g_2 \in \Ah_{\gamma}$ where $g_2$ has no zeroes. Thus $\Nh_0 = \Nh$ is the Nevanlinna class. We set
\[
 \Nh^1_{\gamma} := \exp \Oh_{\kappa_{\gamma}} = \{ f \in \Nh^{\times}_{\gamma} \tei f (0) = 1 \} = \{ g_1 / g_2 \tei g_i \in \Ah^1_{\gamma} \} \; .
\]
In this multiplicative setting Korenblum's theorem \ref{t821} amounts to the following bijections for $\gamma \ge 0$, given by mapping $\mu$ to $f_{\mu} = \exp (-h_{\mu})$:
\begin{equation}
 \label{eq:8.92}
P^+_{\kappa_{\gamma}} (\T) \silo \Ah^1_{\gamma} \quad \text{and} \quad P_{\kappa_{\gamma}} (\T) \silo \Nh^1_{\gamma} \; .
\end{equation}
Using theorem \ref{t7.19} it follows that $A (f)$ exists for every function $f \in \Nh^1_{\gamma}$. Let $\Nh^1_{\gamma , c}$ be the subgroup of atomless functions and define $\Ah^1_{\gamma , c}$ accordingly. 

Starting with the work of Korenblum on $\Ah_1$, the Blaschke condition has been generalized to all the spaces $\Ah_{\gamma}$ for $\gamma \ge 0$ c.f. \cite{S}. Not only the absolute values but also the arguments of the zeroes enter the conditions if $\gamma > 0$. For the absolute values one has a necessary condition which follows from theorem 1.1, remark 2 and the proof of lemma 3.2, all from \cite{S}.

\begin{theorem}[Seip] \label{t22}
Let $\rho : D \to \Z_{\ge 0}$ be a map. If a function $f \in \Ah_{\gamma}$ exists with $\ord_z f = \rho (z)$ for all $z \in D$ then there are constants $c_1 , c_2 \ge 0$ such that for all $1/2 \le r < 1$ we have
\[
 \sum_{|z|\le r} (1 - |z|) \rho (z) \le c_1 |\log (1-r)|^{\gamma} + c_2 \; . 
\]
\end{theorem}

We now review the theory of the singular measure attached to a premeasure of $\kappa$-bounded variation c.f. \cite{K1} 4.3, \cite{K2} \S\,1 and \cite{K4} \S\,4. The $\kappa$-entropy of a finite set $\emptyset \neq E \subset \T$ is defined to be
\begin{equation} \label{eq:895n}
 \kappa (E) = \sum_I \kappa (|I|) \; .
\end{equation}
Here the arcs $I$ are the connected components of $\T \setminus E$. Note that since $\kappa (x) \ge x$ we have $\kappa (E) \ge 1$. One sets $\kappa (\emptyset) = 0$. The definition is extended to arbitrary closed sets $\emptyset \neq E \subset \T$ by setting
\begin{equation}
 \label{eq:893}
\kappa (E) = \sup \{ \kappa (E_1) \tei E_1 \subset E \; \text{finite} \} \in [1,\infty] \; .
\end{equation}
Formula \eqref{eq:895n} remains true for any closed subset $\emptyset \neq E \subset \T$ with $\lambda (E) = 0$. To see this one uses an approximation argument and the inequality
\[
 \sum_i \kappa (|I_i|) \ge \kappa (|J|) \; .
\]
Here $J$ is an arc and $\{ I_i \}$ is an at most countable family of pairwise disjoint arcs with
\[
 \coprod_i I_i \subset J \quad \text{and} \quad \sum_i |I_i| = |J|\; .
\]
For closed sets $E \subset E'$ we have $\kappa (E) \le \kappa (E')$. For closed $E \neq \emptyset$ there is the formula
\begin{equation}
 \label{eq:894}
\kappa (E) = \int_{\T} \kappa' (\dist (\zeta , E)) \, \lambda (\zeta) \; .
\end{equation}
Here $\dist (\zeta, \eta)$ is the length of the shorter arc between $\zeta , \eta \in \T$ normalized by $\dist (1,-1) = 1$. Moreover
\[
 \dist (\zeta , E) = \inf \{ \dist (\zeta , \eta) \tei \eta \in E \} \; .
\]
The definition of the derivative $\kappa' \ge 0$ for general $\kappa$ is given in \cite{K4} 2.3. We are only interested in the case where $\kappa$ is differentiable in $(0,1]$ and we set $\kappa' (0) = \lim_{x\to 0+} x^{-1} \kappa (x) \le \infty$. In particular for $\gamma \ge 0$ and closed $\emptyset \neq E \subset \T$ one has
\begin{equation}
 \label{eq:895}
\kappa_{\gamma} (E) = \Gamma (\gamma +1)^{-1} \int_{\T} |\log (\dist (\zeta , E))|^{\gamma} \, d\lambda (\zeta) \; .
\end{equation}
Premeasures $\mu \in P_{\kappa} (\T)$ can be extended $\sigma$-additively to more general open and closed subsets of $\T$ than just finite disjoint unions of arcs. For an open set $U \subset \T$ with $\kappa (\partial U) < \infty$ the series
\[
 \mu (U) := \sum_J \mu (J)
\]
over the connected components $J$ of $U$ is absolutely convergent. The $J$'s are (open) arcs and hence $\mu (J)$ is defined. 

For a closed set $F \subset \T$ with $\kappa (\partial F) < \infty$ set
\begin{equation}
 \label{eq:896}
\mu (F) := - \mu (\T \setminus F) \; .
\end{equation}
Note that $\mu (\T) = 0$ by definition of a premeasure. One has the inequality
\begin{equation}
 \label{eq:897}
|\mu (F)| = |\mu (\T \setminus F)| \le \halb \| \mu \|_{\kappa} \kappa (\partial F) \; .
\end{equation}
The closed sets $E \subset \T$ with $\lambda (E) = 0$ and $\kappa (E) < \infty$ are called $\kappa$-Carleson sets. As noted above we have $\kappa (E) = \sum \kappa (|I|)$ if $\T \setminus E = \coprod I$ is the decomposition into connected components and $E \neq \emptyset$. Closed subsets of $\kappa$-Carleson sets are again $\kappa$-Carleson sets. Let $\Fh_{\kappa}$ be the set of $\kappa$-Carleson sets. According to \eqref{eq:896}, $\mu$ extends to a function on $\Fh_{\kappa}$. If $\kappa' (0) = + \infty$ as for $\kappa = \kappa_{\gamma}$ with $\gamma > 0$ the condition $\kappa (E) < \infty$ already implies that $\lambda (E) = 0$. This follows from equation \eqref{eq:894}.

\begin{example} \label{t826n}
\rm For the $3$-adic Cantor set $E \subset \T$ we have 
\[
 \kappa (E) = \sum^{\infty}_{n=1} 2^{n-1} \kappa (3^{-n})\; .
\]
Hence $E$ is a $\kappa_{\gamma}$-Carleson set for every $\gamma \ge 0$. Note that we have $\varphi_3 (E) \subset E$. For $0 < \alpha < 1$ consider the function $\kappa^{\alpha} (x) = x^{\alpha}$. Then $E$ is a $\kappa^{\alpha}$-Carleson set if and only if $\alpha > \log 2 / \log 3$ the Hausdorff dimension of $E$. Similar facts hold for any integer $N \ge 2$ instead of $N = 3$. Thus for $s = 1$ there are non-empty closed subsets $E \subset \T$ with $\Sh E \subset E$ which are $\kappa_{\gamma}$-Carleson sets for any $\gamma \ge 0$. For $s \ge 2$ on the other hand every non-empty closed subset $E \subset \T$ with $\Sh E \subset E$ is equal to $\T$ by a theorem of Furstenberg \cite{F}. In particular there are no (forward) $\Sh$-invariant $\kappa$-Carleson sets for any $\kappa$ if $s \ge 2$. 
\end{example}

Let $\Bh_{\kappa} \supset \Fh_{\kappa}$ be the set of all Borel sets $B \subset \T$ with $\oB \in \Fh_{\kappa}$. Clearly $\Bh_{\kappa}$ is the union of all Borel algebras $\Bh (E)$ for $E \in \Fh_{\kappa}$. A $\kappa$-singular measure is a function $\sigma : \Bh_{\kappa} \to \R$ such that:\\
1) $\sigma^E = \sigma \, |_{\Bh (E)}$ is a finite real measure on $\Bh (E)$ for every $\kappa$-Carleson set $E$.\\
2) There is a constant $c = c_{\sigma} \ge 0$ such that
\[
 |\sigma (F)| \le c \kappa (F) \quad \text{for every} \; F \in \Fh_{\kappa} \; .
\]
By the uniqueness of the Jordan decomposition, the total variation measures $|\sigma^E|$ glue to a $\kappa$-singular measure denoted by $|\sigma|$ with $c_{|\sigma|} \le 2c_{\sigma}$. The usual norms $\| \sigma^E \| = |\sigma^E| (E)$ may be unbounded as $E$ varies.

Let $M_{\kappa} (\T)$ denote the real vector space of $\kappa$-singular measures. It becomes a Banach space with norm $\|\sigma \|_{\kappa}$ the infimum of the constants $c_{|\sigma|}$ in the above estimate. Let $M^+_{\kappa} (\T)$ be the cone of $\kappa$-singular measures $\sigma$ with $\sigma^E \ge 0$ i.e. $\sigma^E \in M^+ (E)$ for all $E \in \Fh_{\kappa}$. Any $\sigma \in M_{\kappa} (\T)$ has a unique decomposition $\sigma = \sigma_+ - \sigma_-$ with $\sigma_{\pm} \in M^+_{\kappa} (\T)$ singular to each other i.e. such that all $\sigma^E_+$ and $\sigma^E_-$ are mutually singular for $E \in \Fh_{\kappa}$. We have $\sigma_{\pm} = (\sigma^E_{\pm})$ and $|\sigma| = \sigma_+ + \sigma_-$. Any finite positive measure on the Borel algebra of a locally compact Hausdorff space with a countable basis of the topology is regular, c.f. \cite{E} VIII 1.12 Korollar. In particular for $\sigma \in M^+_{\kappa} (\T)$ the Borel measures $\sigma^E$ on $E$ are regular for $E \in \Fh_{\kappa}$. 

\begin{theorem}[Korenblum] \label{t825}
 For $\mu \in P_{\kappa} (\T)$ there is a unique extension of $\mu \, |_{\Fh_{\kappa}}$ to a $\kappa$-singular measure $\mu_s : \Bh_{\kappa} \to \R$. For $\mu \in P^+_{\kappa} (\T)$ we have $\mu_s \in M^+_{\kappa} (\T)$. On the other hand there is an absolute constant $A > 0$ such that for every $\sigma \in M^+_{\kappa} (\T)$ there is a $\mu \in P^+_{\kappa} (\T)$ with $\mu_s = \sigma$ and $\| \mu_s \|_{\kappa} \le A \| \sigma \|_{\kappa}$.
\end{theorem}

\begin{proof}
 The uniqueness is clear since $\mu^E_s (F) = \mu (F)$ for all closed subsets $F$ of a given $\kappa$-Carleson set $E$ and these $F$ form a $\cap$-stable generator of $\Bh (E)$. For $\kappa = \kappa_1$, Korenblum proves the existence of $\mu_s$ and its properties in \cite{K1} 4.3.~Theorem 6 and Corollary and \cite{K2} Theorem 2.3. The proof generalizes without changes to general $\kappa$.
\end{proof}

For $\mu \in P_{\kappa} (\T)$ we have $|\mu_s (F)| \le \halb \| \mu \|_{\kappa} \kappa (F)$ for $F \in \Fh_{\kappa}$ by \eqref{eq:897}. Hence $\|\mu_s\|_{\kappa} \le \halb \| \mu\|_{\kappa}$ and therefore the linear map of Banach spaces
\[
 \fs : P_{\kappa} (\T) \longrightarrow M_{\kappa} (\T) \quad \text{sending $\mu$ to $\mu_s$}
\]
is continuous. Since $M_{\kappa} (\T) = M^+_{\kappa} (\T) - M^+_{\kappa} (\T)$ it follows from the theorem that $\fs$ is surjective.

\begin{remark}
 \label{t826}
The composition
\[
 \xymatrix{
M^+ (\T) \ar@{->>}[r]^{\omega} & P^+_{\kappa_0} (\T) \ar@{^{(}->}[r] & P^+_{\kappa} (\T) \ar[r]^{\fs} & M^+_{\kappa} (\T)
} 
\]
sends $\mu'$ to $\mu' \, |_{\Bh_{\kappa}} = \mu'_{\sing} \, |_{\Bh_{\kappa}}$. Here $\omega (\mu') = \mu' - \mu' (\T) \lambda$ is the map \eqref{eq:8.89n}.
\end{remark}

The next fact follows from proposition \ref{t26n}.

\begin{prop}
 \label{t31n}
For $0 \le \delta \le \gamma$ the inclusion $\Bh_{\kappa_{\gamma}} \subset \Bh_{\kappa_{\delta}}$ induces a bounded linear restriction operator
\[
 \res_{\delta,\gamma} : M_{\kappa_{\delta}} (\T) \longrightarrow M_{\kappa_{\gamma}} (\T) \quad \text{with} \; \| \res_{\delta , \gamma} \| \le 1\; .
\]
It maps $M^+_{\kappa_{\delta}} (\T)$ to $M^+_{\kappa_{\gamma}} (\T)$.
\end{prop}

The next theorem of Korenblum \cite{K2} Theorem 2.2 asserts that any $\kappa$-singular measure is concentrated on a countable union of $\kappa$-Carleson sets. He proves the statement for $\kappa = \kappa_1$ but again it is valid in general.

\begin{theorem}[Korenblum]
\label{t827}
For $\sigma \in M_{\kappa} (\T)$ there is a sequence $F_1 \subset F_2 \subset \ldots$ in $\Fh_{\kappa}$ such that for every $F \in \Fh_{\kappa}$ we have
\[
 \sigma (F) = \lim_{\nu\to\infty} \sigma (F \cap F_{\nu}) \quad \text{and} \quad |\sigma| (F) = \lim_{\nu\to\infty} |\sigma| (F \cap F_{\nu}) \; .
\]
\end{theorem}

\begin{cor} \label{t828n}
 In the situation of the theorem we have
\[
 \sigma (B) = \lim_{\nu \to\infty} \sigma (B \cap F_{\nu})
\]
for every $B \in \Bh (E)$ with $E \in \Fh_{\kappa}$. 
\end{cor}

\begin{proof}
 We may assume that $\sigma \in M^+_{\kappa} (\T)$. The finite Borel measure $\sigma^E = \sigma \, |_{\Bh (E)}$ is regular as pointed out above. In particular we have
\[
 \sigma (B) = \sup_{K \subset B} \sigma (K)
\]
where $K$ ranges over all compact subsets of $B$. Since $K \subset E$ is closed we have $K \in \Fh_{\kappa}$ and hence theorem \ref{t827} gives
\[
 \sigma (K) = \sup_{\nu \ge 1} \sigma (K \cap F_{\nu}) \; .
\]
On the other hand, since $K \cap F_{\nu}$ ranges over all compact subsets of $B \cap F_{\nu}$ if $K$ varies over the compact subsets of $B$ we get by regularity:
\[
 \sigma (B \cap F_{\nu}) = \sup_{K \subset B} \sigma (K \cap F_{\nu}) \; .
\]
Hence we find:
\begin{align*}
 \sigma (B) & = \sup_{K \subset B} \sup_{\nu \ge 1} \sigma (K \cap F_{\nu}) \\
& = \sup_{\nu \ge 1} \sup_{K \subset B} \sigma (K \cap F_{\nu}) \\
& = \sup_{\nu \ge 1} \sigma (B \cap F_{\nu}) \\
& = \lim_{\nu \to 0} \sigma (B \cap F_{\nu}) \; .
\end{align*}
\end{proof}

\begin{defn} \label{t830n}
A possibly infinite positive measure $\tsigma \ge 0$ on $\Bh (\T)$ is called $\kappa$-thin if the following conditions hold:\\
i) There  is a sequence $F_1 \subset F_2 \subset \ldots$ in $\Fh_{\kappa}$ such that for every Borel set $B \subset \T$ we have:
\[
 \tsigma (B) = \lim_{\nu\to \infty} \tsigma (B \cap F_{\nu}) \; .
\]
Equivalently, $\tsigma$ is of the form $\tsigma = i_* \tsigma_G$ where $\tsigma_G \ge 0$ is a measure on $\Bh (G)$ for $G = \bigcup_{\nu \ge 1} F_{\nu}$ and $i : G \hookrightarrow \T$ is the inclusion. \\
ii) There is a constant $c = c_{\tsigma} \ge 0$ such that
\[
 |\tsigma (F)| \le c \kappa (F) \quad \text{for all} \; F \in \Fh_{\kappa} \; .
\]
The real cone of such measures is denoted by $\tilde{M}^+_{\kappa} (\T)$.
\end{defn}

Any $\kappa$-thin measure $\tsigma$ on $\T$ is $\sigma$-finite and singular with respect to Lebesgue measure. Note here that the sets $B_{\nu} = (\T \setminus G) \cup F_{\nu}$ form a countable Borel exhaustion of $\T$ with $\tsigma (B_{\nu}) < \infty$. In proposition \ref{t939n} below we will see that many ergodic measures on $\T$ are $\kappa$-thin.

\begin{theorem} \label{t829n}
 Every $\sigma \in M^+_{\kappa} (\T)$ has a unique extension $\tsigma$ to a $\kappa$-thin measure $\tsigma \ge 0$ on $\T$. The measure $\tsigma$ is finite if and only if there is a constant $c \ge 0$ with $\sigma (F) \le c$ for all $F \in \Fh_{\kappa}$. In this case we have $\tsigma (\T) \le c$. The restriction of a $\kappa$-thin measure $\tsigma$ on $\Bh (\T)$ to $\Bh_{\kappa}$ defines a $\kappa$-singular measure $\sigma = \tsigma \, |_{\Bh_{\kappa}}$ in $M^+_{\kappa} (\T)$. 
\end{theorem}

\begin{proof}
 Choose a sequence $F_1 \subset F_2 \subset \ldots$ in $\Fh_{\kappa}$ for $\sigma$ as in theorem \ref{t827} and set $G = \bigcup_{\nu \ge 1} F_{\nu}$. Then $\sigma$ defines a $\sigma$-additive and $\sigma$-finite positive content $\sigma \, |_{\Eh}$ on the ring $\Eh = \bigcup_{\nu \ge 1} \Bh (F_{\nu})$ in $G$. The $\sigma$-algebra generated by $\Eh$ in $G$ is $\Bh (G)$. Hence $\sigma \, |_{\Eh}$ has a unique extension to a $\sigma$-finite measure $\tsigma_G$ on $G$, c.f. \cite{E} II Korollar 5.7. We set $\tsigma = i_* \tsigma_G$ where $i : G \hookrightarrow \T$ is the inclusion. 
For $F \in \Fh_{\kappa}$ we have
\begin{align*}
 \tsigma (F) & = \tsigma_G (F \cap G) = \lim_{\nu\to\infty} \tsigma_G (F \cap F_{\nu}) \quad \text{since $\tsigma_G$ is a measure} \\
& = \lim_{\nu\to\infty} \sigma (F \cap F_{\nu}) \quad \text{since} \; \tsigma_G = \sigma \; \text{on} \; \Bh (F_{\nu}) \subset \Eh \\
& = \sigma (F) \quad \text{by theorem \ref{t827}.} 
\end{align*}
It follows that $\tsigma$ is a $\kappa$-thin measure on $\T$. Moreover we have $\tsigma \, |_{\Fh_{\kappa}} = \sigma \, |_{\Fh_{\kappa}}$ and therefore $\tsigma \, |_{\Bh_{\kappa}} = \sigma$ by the uniqueness assertion in theorem \ref{t825}. 

Now let $\tsigma'$ be another $\kappa$-thin extension of $\sigma$. Choose a sequence $F'_1 \subset F'_2 \subset \ldots $ in $\Fh_{\kappa}$ such that
\[
 \tsigma' (B) = \lim_{\nu\to\infty} \tsigma' (B \cap F'_{\nu}) \quad \text{for Borel sets}\; B \subset \T \; .
\]
Set $G' = \bigcup_{\nu \ge 1} F'_{\nu}$. For $B \subset \T$ Borel we have:
\begin{align*}
 \tsigma' (B) & = \lim_{\nu \to\infty} \tsigma' (B \cap F'_{\nu}) \\
& = \lim_{\nu\to\infty} \sigma (B \cap F'_{\nu}) \quad \text{since} \; B \cap F'_{\nu} \in \Bh_{\kappa} \; .
\end{align*}
Applying corollary \ref{t828n} to the Borel set $B \cap F'_{\nu}$ in $E = F'_{\nu}$ we get
\[
 \sigma (B \cap F'_{\nu}) = \lim_{n \to \infty} \sigma (B \cap F'_{\nu} \cap F_n)
\]
and hence
\[
 \tsigma' (B) = \lim_{\nu\to\infty} \lim_{n \to\infty} \sigma (B \cap F'_{\nu} \cap F_n) \; .
\]
Similarly
\[
 \tsigma (B) = \lim_{n \to\infty} \lim_{\nu\to\infty} \sigma (B \cap F'_{\nu} \cap F_n) \; .
\]
This implies $\tsigma' (B) = \tsigma (B)$ since both double limits are equal to $\tsigma (B \cap G' \cap G)$. Their equality also follows from monotonicity replacing $\lim$'s by $\sup$'s.

Thus the extension of $\sigma$ to a $\kappa$-thin measure $\tsigma$ on $\Bh (\T)$ is unique. If there is a constant $c \ge 0$ with $\sigma (F) \le c$ for all $F \in \Fh_{\kappa}$ then $\tsigma (\T) = \sup_{\nu \ge 1} \tsigma (F_{\nu}) = \sup_{\nu \ge 1} \sigma (F_{\nu}) \le c$ as well. If on the other hand $\tsigma$ is finite, then
\[
 \sigma (F) = \tsigma (F) \le \tsigma (\T) < \infty \quad \text{for all} \; F \in \Fh_{\kappa} \; .
\]
The last assertion follows from the definitions.
\end{proof}

\begin{remark} \label{t832n}
 \rm By the theorem the restriction map $\tilde{M}^+_{\kappa} (\T) \silo M^+_{\kappa} (\T)$ is an isomorphism. The additive monoid $\tilde{M}^+_{\kappa} (\T)$ has the cancellation property since $M^+_{\kappa} (\T)$ has this property. Hence its Grothendieck group
\[
 \tilde{M}_{\kappa} (\T) := K_0 (\tilde{M}^+_{\kappa} (\T))
\]
is the set of equivalence classes $[\mu_1 , \mu_2]$ of pairs $(\mu_1 , \mu_2)$ with $\mu_1 , \mu_2 \in \tilde{M}^+_{\kappa} (\T)$ where $(\mu_1 , \mu_2) \sim (\mu'_1 , \mu'_2)$ if and only if $\mu_1 + \mu'_2 = \mu_2 + \mu'_1$. The addition is defined componentwise. The natural map 
\[
 \tilde{M}^+_{\kappa} (\T) \hookrightarrow \tilde{M}_{\kappa} (\T)
\]
sending $\mu$ to $[\mu , 0]$ is injective. The group $\tM_{\kappa} (\T)$ is ordered with $\tM^+_{\kappa} (\T)$ being the set of non-negative elements. Thus we have
\[
 [\mu_1 , \mu_2] \ge 0 \quad \text{if and only if} \; \mu_1 = \mu_2 + \mu \; \text{for some} \; \mu \in \tM^+_{\kappa} (\T) \; .
\]
With the $\R$-operation defined by $\lambda [\mu_1 , \mu_2] = [\lambda \mu_1 , \lambda \mu_2]$ for $\lambda \ge 0$ and $\lambda [\mu_1 , \mu_2] = [|\lambda| \mu_2 , |\lambda| \mu_1]$ for $\lambda < 0$ the group $\tilde{M}_{\kappa} (\T)$ becomes an $\R$-vector space and the natural map $\tilde{M}_{\kappa} (\T) \to M_{\kappa} (\T)$ sending $[\mu_1 , \mu_2]$ to $\mu_1 \, |_{\Bh_{\kappa}} - \mu_2 \, |_{\Bh_{\kappa}}$ is an isomorphism of real vector spaces. We define a  Banach space structure on $\tM_{\kappa} (\T)$ such that this isomorphism becomes an isometry.
\end{remark}

As an example, we discuss the case $\kappa = \kappa_0$:

\begin{prop} \label{t35n}
 We have $\tM^+_{\kappa_0} (\T) = M^+ (\T)_{\sing}$ and $\tM_{\kappa_0} (\T) = M (\T)_{\sing}$. The natural restriction map
\[
 \fs : M (\T) \longrightarrow M_{\kappa_0} (\T) \cong \tM_{\kappa_0} (\T) = M (\T)_{\sing}
\]
 sends a measure $\mu$ to its singular part, $\fs (\mu) = \mu_{\sing}$ in the canonical decomposition $\mu = \mu_a + \mu_{\sing}$ where $\mu_a \ll \lambda$ and $\mu_{\sing} \perp \lambda$.
\end{prop}

\begin{proof}
 For every finite set $E \neq \emptyset$ we have $\kappa_0 (E) = 1$. By definition \eqref{eq:893} the same is true for any closed set $F \neq \emptyset$.  By definition $\Fh_{\kappa_0}$ thus consists of the closed Lebesgue null sets in $\T$. By definition \ref{t830n} and theorem \ref{t829n} the positive $\kappa_0$-thin measures are the bounded measures on $\Bh (\T)$ for which a set of full measure $G = \bigcup^{\infty}_{\nu=1} F_{\nu}$ exists with $F_1 \subset F_2 \subset \ldots$ a sequence of closed Lebesgue null sets. Hence $\tM^+_{\kappa_0} (\T) \subset M^+ (\T)_{\sing}$. On the other hand, every measure $\mu \in M^+ (\T)_{\sing}$ is regular. Choose a Borel set $G$ with $\lambda (G) = 0$ and $\mu (\T) = \mu (G)$. We have
\[
 \mu (G) = \sup_{K \subset G} \mu (K)
\]
where $K$ ranges over the compact subsets of $G$. Hence there exists a sequence $K_1 \subset K_2 \subset \ldots $ with $K_{\nu} \subset G$ and hence $\lambda (K_{\nu}) = 0$ such that
\[
 \mu (\T) = \lim_{\nu \to \infty} \mu (K_{\nu})\, .
\]
Thus condition i) in definition \ref{t830n} is satisfied, ii) being clear by the above. This shows the reverse inclusion $M^+ (\T)_{\sing} \subset \tM^+_{\kappa_0} (\T)$. In particular $\mu_{\sing} \in M^+ (\T)_{\sing} = \tM^+_{\kappa_0} (\T)$ is the unique (c.f. theorem \ref{t829n}) extension of $\mu_{\sing} \, |_{\Bh_{\kappa_0}}$ to a positive $\kappa_0$-thin measure on $\T$. Hence we have $\mu_{\sing} = \fs (\mu_{\sing})$ and therefore $\mu_{\sing} = \fs (\mu)$ because $\fs (\mu_a) = 0$. The extension to real  measures follows.
\end{proof}

\begin{notation} \label{t37n}
\rm By $M^{\kappa}_+ (\T)$ we denote the set of measures $\mu \in M^+ (\T)$ which satisfy condition i) in definition \ref{t830n}. We set 
\[
 M^{\kappa} (\T) = \{ \mu - \mu (\T) \lambda \tei \mu = \mu_1 - \mu_2 \; \text{for some} \; \mu_1 , \mu_2 \in M^{\kappa}_+ (\T) \} \; .
\]
This is compatible with the previous definition of $M^{\kappa_0} (\T)$ in \eqref{eq:63n}. We may view $M^{\kappa}_+ (\T) \subset \tM^+_{\kappa} (\T)$ as the set of positive bounded $\kappa$-thin measures. Note that the canonical linear map
\[
 M^{\kappa} (\T) \hookrightarrow M_{\kappa} (\T) \cong \tM_{\kappa} (\T) \; , \; \sigma \longmapsto \sigma_s \longmapsto \tsigma_s
\]
is injective. 
\end{notation}

We finally discuss the relation between cyclicity and $\kappa$-singular measures. Recall the $\C$-algebra $\Ah_{\gamma}$ defined using \eqref{eq:8.91}. For $R > 0$ let $\Ah_{\gamma} (R) \subset \Ah_{\gamma}$ be the set of holomorphic functions $f$ on $D$ with
\[
 \| f \|_{\Ah_{\gamma} (R)} := \sup_{z \in D} |f (z)| \exp (-R |\log (1-|z|)|^{\gamma}) < \infty \; .
\]
With this norm, $\Ah_{\gamma} (R)$ becomes a Banach space and for $R_1 < R_2$ the inclusion $\Ah_{\gamma} (R_1) \hookrightarrow \Ah_{\gamma} (R_2)$ is continuous. Choose a sequence $R_1 < R_2 < \ldots$ with $\lim_{\nu \to \infty} R_{\nu} = \infty$ and give $\Ah_{\gamma}$ the locally convex inductive limit topology. This is the finest locally convex topology such that all inclusion $\Ah_{\gamma} (R_{\nu}) \hookrightarrow \Ah_{\gamma}$ become continuous. It exists and is Hausdorff because the topology of uniform convergence on compact subsets of $D$ is a Hausdorff locally convex topology on $\Ah_{\gamma}$ for which all inclusions $\Ah_{\gamma} (R) \hookrightarrow \Ah_{\gamma}$ are continuous. The inductive limit topology on $\Ah_{\gamma}$ does not depend on the choice of the sequence $(R_{\nu})$. The Hausdorff $LB$-space $\Ah_{\gamma}$ is in addition a unital topological algebra. A sequence $(f_n)$ in $\Ah_{\gamma}$ converges if and only if there is some $R > 0$ such that all $f_n$ are in $\Ah_{\gamma} (R)$ and $(f_n)$ converges in $\Ah_{\gamma} (R)$. An element $f \in \Ah_{\gamma}$ is called cyclic if $f\Ah_{\gamma}$ is dense in $\Ah_{\gamma}$. For $\gamma = 0$ this topology on $\Ah_0 = H^{\infty} (D)$ is the usual one on $H^{\infty} (D)$. 

According to \eqref{eq:8.92} every $f \in \Ah^1_{\gamma}$ is of the form $f = f_{\mu} = \exp (-h_{\mu})$ for a uniquely determined premeasure $\mu \in P^+_{\kappa_{\gamma}} (\T)$. Let $\sigma_f = \mu_s \in M^+_{\kappa_{\gamma}} (\T)$ denote the corresponding $\kappa_{\gamma}$-singular measure. The following result is a special case of \cite{K2} Corollary 1.1.1:

\begin{theorem}[Korenblum] \label{t828}
An element $f \in \Ah^1_1$ is cyclic in $\Ah_1$ if and only if $\sigma_f = 0$.
\end{theorem}
I have no doubt that similar methods as the ones used in \cite{K2} will give the corresponding statement for $\Ah^1_{\gamma}$ for every $\gamma > 0$. Since I have no reference and the argument is not so short I will state this only as a conjecture:

\begin{conjecture} \label{t41a}
 For $\gamma > 0$ an element $f \in \Ah^1_{\gamma«}$ is cyclic in $\Ah_{\gamma}$ if and only if $\sigma_f = 0$. 
\end{conjecture}

Let $g$ be a singular inner function on $D$ corresponding to an ordinary singular measure $\mu_g \in M^+ (\T)_{\sing}$. Then $f = g / g (0) \in \Ah^1_{\gamma}$ corresponds to the premeasure (in fact a measure) $\mu = \omega (\mu_g) = \mu_g - \mu_g (\T) \lambda$. According to remark \ref{t826} we have $\sigma_f = \mu_g \, |_{\Bh_{\kappa_{\gamma}}}$. For $\gamma > 0$, conjecture \ref{t41a} if true would therefore tell us that $g$ was cyclic in $\Ah_{\gamma}$ if and only if $\mu_g$ vanished on the $\kappa_{\gamma}$-Carleson sets. For $\gamma = 1$ this is true by theorem \ref{t828}. 

In fact a  more precise result holds \cite{K3}. Let $\Ah^0_1 (n)$ be the subspace of $\Ah_1 (n)$ consisting of functions $f$ with 
\[
\max_{|z|=r} |f(z)| (1-|z|)^n \to 0 \quad \text{for} \; r \to 1- \; .
\]
The space $\Ph = \C [z]$ of polynomials is dense in $\Ah^0_1 (n)$ for every $n \ge 1$ and $f \in \Ah^0_1 (n)$ is called cyclic in $\Ah^0_1 (n)$ if $f \Ph$ is dense in $\Ah^0_1 (n)$.

\begin{theorem}[Korenblum]\label{t829}
 Consider a singular inner function $g$ with corresponding singular measure $\mu_g \in M^+ (\T)$. Then the following assertions are equivalent:
\begin{enumerate}
 \item [a)] $\mu_g$ vanishes on $\kappa_1$-Carleson sets i.e. $\mu_g \, |_{\Bh_{\kappa_1}} = 0$
\item [b)] $g$ is cyclic in any and hence all of the spaces $\Ah_1$ and $\Ah^0_1 (n)$ for $n \ge 1$.
\end{enumerate}

\end{theorem}

\section{Compatibilities with actions} \label{sec:9}
In this section we discuss left and right $S$-actions on the objects discussed in the last section. Here $S$ is the monoid introduced in the beginning of section \ref{sec:5}. As before $\kappa$ is a continuous nondecreasing concave function on $[0,1]$ with $\kappa (0) = 0$ and $\kappa (1) = 1$. For $\alpha \ge 1$ we have the estimates
\begin{equation} \label{eq:9n:100}
	\kappa (x) \le \alpha \kappa \Big( \frac{x}{\alpha} \Big) \le \alpha \kappa (x) \quad \text{for all} \; 0 \le x \le 1 \; .
\end{equation}
For $\kappa_{\gamma}$ we also have an estimate for $\gamma \ge 0 , \alpha \ge 1$ and $0 \le x \le 1$
\[
 \alpha \kappa_{\gamma} \Big( \frac{x}{\alpha} \Big) \le 2^{\gamma} \kappa_{\gamma} (x) (1 + c^{-1}_{\gamma} (\log \alpha)^{\gamma})
\]
where $c_{\gamma} = \Gamma (\gamma+1)^{-1}$.

\begin{prop} \label{t9n33}
 Consider an integer $N \ge 1$ and an element $\zeta \in \T$. \\
a) The maps $[\zeta]_* , \varphi_{N^*}$ and $\varphi^*_N$ on the space of premeasures $P (\T)$ respect $P_{\kappa} (\T)$ and $P^{\pm}_{\kappa} (\T)$. For $\mu \in P_{\kappa} (\T)$ we have
\begin{align*}
 \| [\zeta]_* \mu \|_{\kappa} & = \|\mu \|_{\kappa} \\
\| \varphi_{N^*} \mu \|_{\kappa} & \le N \| \mu \|_{\kappa} \\
\| \varphi^*_N \mu \|_{\kappa} & \le 2N \| \mu \|_{\kappa}
\end{align*}
For $\mu \in P^-_{\kappa} (\T)$ with $\mu (C) \le a_{\mu} \kappa (|C|)$ for arcs $C$ we have:
\begin{align*}
 ([\zeta]_* \mu) (C) & \le a_{\mu} \kappa (|C|) \\
(\varphi_{N^*} \mu) (C) & \le N a_{\mu} \kappa (|C|) \\
(\varphi^*_N \mu) (C) & \le 2N a_{\mu} \kappa (|C|)
\end{align*}
b) For $E \in \Fh_{\kappa}$ we have $\varphi^{-1}_N (E) \in \Fh_{\kappa}$ and $\varphi_N (E) \in \Fh_{\kappa}$. For $B \in \Bh_{\kappa}$ we have $[\zeta]^{-1} (B) , \varphi^{-1}_N (B) , \varphi_N (B \cap I_k) \in \Bh_{\kappa}$ for every $0 \le k \le N-1$. Here $I_k$ is the arc $I_k = [\frac{2\pi k}{N} , \frac{2\pi (k+1)}{N})$. In particular, for any map $\mu : \Bh_{\kappa} \to \C$ we can define $([\zeta]_* \mu) (B) = \mu ([\zeta]^{-1} (B)) , (\varphi_{N^*} \mu) (B) = \mu (\varphi^{-1}_N (B))$ and $\varphi^*_N \mu = \frac{1}{N} \sum^{N-1}_{k=0} \mu_k$ with $\mu_k (B) = \mu (\varphi_N (B \cap I_k))$. These operations leave $M_{\kappa} (\T)$ and $M^{\pm}_{\kappa} (\T)$ invariant and if $|\mu (F) | \le c_{\mu} \kappa (F)$ for all $F \in \Fh_{\kappa}$ we have the estimates
\begin{align*}
 |([\zeta]_* \mu) (F)| & \le c_{\mu} \kappa (F) \\
|(\varphi_{N^*} \mu) (F)| & \le Nc_{\mu} \kappa (F) \\
|(\varphi^*_N \mu) (F)| & \le (2N+1) c_{\mu} \kappa (F)
\end{align*}
c) The operations $[\zeta]_* , \varphi_{N^*}$ and $\varphi^*_N$ on the space of possibly unbounded positive measures on $\Bh (\T)$ leave  $\tilde{M}^{\pm}_{\kappa} (\T)$ invariant and hence induce operations on $\tilde{M}_{\kappa} (\T) = K_0 (\tilde{M}_{\kappa}^+ (\T))$.\\
d) The usual formulas $\varphi_N \mu = \varphi_{N^*} \mu , [\zeta] \mu = [\zeta]_* \mu$ and $\mu \varphi_N = \varphi^*_N (\mu) , \varphi [\zeta] = [\zeta^{-1}]_* \mu$ define left and right $S$-actions on $P_{\kappa} (\T) , P^{\pm}_{\kappa} (\T) , M_{\kappa} (\T) , M^{\pm}_{\kappa} (\T)$ and $\tilde{M}_{\kappa} (\T) , \tilde{M}^{\pm}_{\kappa} (\T)$. In addition the following relations hold on theses spaces:
\[
 \varphi_{N^*} \varphi^*_N = \id \; , \; \varphi^*_N \varphi_{N^*} = N^{-1} \Tr_N \; , \; \varphi^*_N \varphi_{M^*} = \varphi_{M^*} \varphi^*_N \; \text{if} \; (N,M) = 1 \; .
\]
\end{prop}

\begin{proof}
 All assertions involving $[\zeta]$ are clear since $[\zeta]$ preserves arc-length. For a closed subset $\emptyset \neq E \subset \T$ with $\lambda (E) = 0$ and decomposition into connected components $\T \setminus E = \coprod_I I$ we have $\kappa (E) = \sum_I \kappa (|I|)$. For an open arc $I \subsetneqq \T$ the inverse image $\varphi^{-1}_N (I)$ decomposes into $N$ disjoint open arcs of length $|I| / N$ each. The estimate \eqref{eq:9n:100} therefore gives:
\begin{equation} \label{eq:9n101}
 \kappa (E) \le \kappa (\varphi^{-1}_N (E)) = \sum_I N \kappa (|I|/ N) \le N \kappa (E) \; .
\end{equation}
 In particular $\varphi^{-1}_N (E)$ is a $\kappa$-Carleson set for any $E \in \Fh_{\kappa}$. For $B \in \Bh_{\kappa}$ we have $\overline{\varphi^{-1}_N (B)} \subset \varphi^{-1}_N (\overline{B})$. Hence $\overline{\varphi^{-1}_N (B)} \in \Fh_{\kappa}$ and therefore $\varphi^{-1}_N (B) \in \Bh_{\kappa}$. Moreover for $\mu \in M_{\kappa} (\T)$ with $|\mu (F)| \le c_{\mu} (F)$ for $F \in \Fh_{\kappa}$ formula \eqref{eq:9n101} implies:
\[
 |(\varphi_{N^*} \mu) (F)| = |\mu (\varphi^{-1}_N (F))| \le c_{\mu} \kappa (\varphi^{-1}_N (F)) \le N c_{\mu} \kappa (F) \; .
\]
In particular $\varphi_{N^*} \mu \in M_{\kappa} (\T)$. This settles the assertion about $\varphi^{-1}_N (B)$ and $\varphi_{N^*} \mu$ in b) The proofs of the assertions about $\varphi_{N^*} \mu$ in a) and c) are similar since for any arc $C \subset \T$ the inverse image $\varphi^{-1}_N (C)$ decomposes into $N$ disjoint arcs of length $|C|/N$.  For $B \in \Bh_{\kappa}$ we have $\overline{B} \in \Fh_{\kappa}$ and since $\overline{B \cap I_k}$ is compact:
\begin{equation} \label{eq:9n102}
 \overline{\varphi_N (B \cap I_k)} \subset \varphi_N (\overline{B \cap I_k}) \subset \varphi_N (\overline{B} \cap \overline{I}_k) \; .
\end{equation}
 In order to show that $\varphi_N (B \cap I_k) \in \Bh_{\kappa}$ it therefore suffices to show that for any $\emptyset \neq E \in \Fh_{\kappa}$ we have $\varphi_N (E \cap \overline{I}_k) \in \Fh_{\kappa}$ as well. It is clear that $\varphi_N (E \cap \overline{I}_k)$ is a closed Lebesgue zero set. Writing $\T \setminus E = \coprod I$ as above a short calculation shows
\begin{equation} \label{eq:9n103}
 \T \setminus \varphi_N (E \cap \overline{I}_k) = \coprod_I \varphi_N (I \cap \overset{\circ}{I}_k) \quad (\dot{\cup} \{ 1 \} )\; .
\end{equation}
Here $\{ 1 \}$ is added if and only if $2 \pi k N^{-1} \notin E$ and $2 \pi (k+1) N^{-1} \notin E$. The disjoint and possibly empty open sets $\varphi_N (I \cap \overset{\circ}{I})$ do not contain $\{ 1 \}$. The open set $I \cap \overset{\circ}{I}_k$ is either an open arc $J$ or the disjoint union of two open arcs $J'$ and $J''$. In the latter case by convexity we have
\begin{equation} \label{eq:9n104}
 \kappa (|\varphi_N (J')|) + \kappa (|\varphi_N (J'')|) \le 2 \kappa (|\varphi_N (I \cap \overset{\circ}{I}_k)| \; .
\end{equation}
Letting $\{ J \}$ denote the set of connected components of the sets $\varphi_N (I \cap \overset{\circ}{I}_k)$ for varying $I$ we therefore get:
\begin{align}
 \sum_J \kappa (|\varphi_N (J)|) & \overset{\eqref{eq:9n104}}{\le} 2 \sum_I \kappa (|\varphi_N (I \cap \overset{\circ}{I}_k)| = 2 \sum_I \kappa (N |I \cap \overset{\circ}{I}_k|) \notag \\
& \overset{\eqref{eq:9n:100}}{\le} 2N \sum_I \kappa (|I \cap \overset{\circ}{I}_k|) \le 2N \sum_I \kappa (|I|) \notag \\
& = 2 N \kappa (E) \; .\label{eq:9n105}
\end{align}
We have
\begin{equation} \label{eq:9n106}
 \T \setminus \varphi_N (E \cap \overline{I}_k) = \coprod_J \varphi_N (J) (\dot{\cup} \{ 1 \}) \; .
\end{equation}
In the case where $\{ 1 \}$ is added to the right hand side in the equation, the union $\coprod_J \varphi_N (J)$ must contain a pointed neighborhood of $1$ since $\T \setminus \varphi_N (E \cap \overline{I}_k)$ is open. Hence the effect of joining $\{ 1 \}$ consists of joining two adjacent intervalls $\varphi_N (J')$ and $\varphi_N (J'')$ to one interval whose length is the sum of both. By subadditivity \eqref{eq:893n} of $\kappa$ we get in this case
\begin{equation} \label{eq:9n107}
 \kappa (|\text{joined intervall}|) \le \kappa (|\varphi_N (J')|) + \kappa (|\varphi_N (J'')|) \; .
\end{equation}
Hence we find:
\begin{align*}
 \kappa (|\varphi_N (E \cap \overline{I}_k)|) & \le \sum_J \kappa (|\varphi_N (J)|) \quad \text{by \eqref{eq:9n106} and \eqref{eq:9n107}} \\
& \le 2N \kappa (E) \quad \text{by \eqref{eq:9n105}.}
\end{align*}
Note that if $|I| \le 1 /N$ for all $I$, the factor $2$ in this inequality an be omitted since in this case $I \cap \overset{\circ}{I}_k$ is always connected. Hence $\varphi_N (E \cap \overline{I}_k)$ is a $\kappa$-Carleson set. For $\mu \in M_{\kappa} (\T)$ and $F \in \Fh_{\kappa}$ we have:
\[
 \mu_k (F) = \mu (\varphi_N (F \cap I_k)) = \mu (\varphi_N (F \cap \overline{I}_k))
\]
unless $2 \pi (k+1)N^{-1} \in F$ and $2\pi k N^{-1} \notin F$. In the latter case firstly we have
\[
 \mu_k (F) = \mu (\varphi_N (F \cap \overline{I}_k)) - \mu ( \{ 1 \} ) \; .
\]
Secondly $\kappa (F) \ge 1$ since $F \neq \emptyset$ and therefore
\[
 |\mu ( \{ 1 \} )| \le c_{\mu} \kappa ( \{ 1 \} ) = c_{\mu} \le c_{\mu} \kappa (F) \; .
\]
Hence in all cases the following estimate holds:
\begin{align*}
|\mu_k (F)| & \le |\mu (\varphi_N (F \cap \overline{I}_k))| + c_{\mu} \kappa (F) \\
& \le c_{\mu} \kappa (\varphi_N (F \cap \overline{I}_k)) + c_{\mu} \kappa (F) \\
& \le (2N+1) c_{\mu} \kappa (F) \; .
\end{align*}
In particular $\mu_k$ is in $M_{\kappa} (\T)$. The assertions about $\varphi^*_N \mu$ in b) are an immediate consequence. The proofs of the claims about $\varphi^*_N \mu$ in a) and c) use similar arguments. Finally d) results from a straightforward computation.
\end{proof}

\begin{remark}
 A $\kappa$-singular measure $\mu \in M_{\kappa} (\T)$ satisfies $N_* \mu = \mu$ if and only if $N_* (\mu^{N^{-1} E}) = \mu^E$ for every $E \in \Fh_{\kappa}$. Here $N$ is viewed as a map $\varphi^{-1}_N (E) \to E$. 
\end{remark}

We set $P_{\kappa , c} (\T) = P_{\kappa} (\T) \cap P_c (\T)$ the space of atomless premeasures of $\kappa$-bounded variation, and similarly for $P^{\pm}_{\kappa , c} (\T)$. A $\kappa$-singular measure $\sigma \in M_{\kappa} (\T)$ is called atomless if $\sigma^E$ has no atoms for every $E \in \Fh_{\kappa}$. Let $M_{\kappa , c} (\T)$ and $M^{\pm}_{\kappa , c} (\T)$ denote the corresponding spaces.

Let $\tM^+_{\kappa , c} (\T)$ be the space of atomless $\kappa$-thin measures and set $\tM_{\kappa , c} (\T) := K_0 (\tM^+_{\kappa , c} (\T)) \subset \tM_{\kappa} (\T)$. All these subspaces are left- and right $S$-invariant.

\begin{cor}
 \label{t938n}
The natural maps
\[
 P^{\pm}_{\kappa} (\T) \xrightarrow{\fs} M^{\pm}_{\kappa} (\T) \xrightarrow{\sim} \tM^{\pm}_{\kappa} (\T)
\]
and
\[
 P_{\kappa} (\T) \xrightarrow{\fs} M_{\kappa} (\T) \xrightarrow{\sim} \tM_{\kappa} (\T)
\]
and
\[
 \xymatrix{
M^+ (\T) \ar@{->>}[r]^{\omega} & P^+_{\kappa_0} (\T) \ar@{^{(}->}[r] & P^+_{\kappa} (\T) \ar[r]^-{\fs} & M^+_{\kappa} (\T) \quad \text{(c.f. remark \ref{t826})}
}
\]
are left and right $S$-equivariant. The same is true for the corresponding maps between the atomless versions of these spaces. 
\end{cor} 

\begin{prop} \label{xxx}
 a) For $\gamma \ge 0$ the left and right $S$-actions on $\Oh^1$ respect $\Ah^1_{\gamma} , \Nh^1_{\gamma}$ and their atomless versions $\Ah^1_{\gamma , c}$ and $\Nh^1_{\gamma , c}$. In addition the following relations hold on $\Oh^1$. 
\[
 N_* N^* = \id , N^* N_* = N^{-1} \Tr_N , N^* M_* = M_* N^* \; \text{if} \; (N,M) = 1 \; .
\]
b) The map $P (\T) \to \Oh^1 , \mu \mapsto f_{\mu} = \exp (-h\mu)$ is left and right $S$-equivariant. Hence the same is true of the isomorphisms \eqref{eq:8.92}
\[
P^+_{\kappa_{\gamma}} (\T) \silo \Ah^1_{\gamma} \quad \text{and} \quad P_{\kappa_{\gamma}} (\T) \silo \Nh^1_{\gamma}
\]
and their atomless versions
\[
 P^+_{\kappa_{\gamma} , c} (\T) \silo \Ah^1_{\gamma , c} \quad \text{and} \quad P_{\kappa_{\gamma} , c} (\T) \silo \Nh^1_{\gamma , c} \; .
\]
\end{prop}

\begin{proof}
 a) The first assertion follows from elementary estimates. \\
b) This follows mainly from equations \eqref{eq:7.76} and \eqref{eq:7.77}. 
\end{proof}

We close this section with a relation between ergodic and $\kappa$-thin measures.

\begin{prop} \label{t939n}
 An $\Sh$-invariant ergodic measure $\mu \in M^+ (\T)$ with $\mu (F) > 0$ for some $\kappa$-Carleson set $F$ is $\kappa$-thin i.e. $\mu \in \tM^+_{\kappa} (\T)$.
\end{prop}

\begin{proof}
 Since $\Sh$ is abelian, the Borel set $\Gh = \bigcup_{s,t \in \Sh} (s^{-1} (t (F))$ satisfies $u^{-1} \Gh = \Gh$ for all $u \in \Sh$. We have $\mu (\Gh) \ge \mu (F) > 0$ and hence $\mu (\Gh) = \mu (\T)$ since $\mu$ is $\Sh$-ergodic. The sets $s^{-1} t (F)$ are in $\Fh_{\kappa}$ by proposition \ref{t9n33}\,b). Hence condition i) in definition \ref{t830n} is satisfied. Condition ii) holds because $\mu$ is bounded and $\kappa (F) \ge 1$ for every $\kappa$-Carleson set $F \neq \emptyset$. 
\end{proof}
\section{On the image of $\Psi_{\Sh}$}
\label{sec:10}

In lemma \ref{t51} we have seen that $\Psi_{\Sh} (\alpha) = \prod_{N \in \Sh} N^* \alpha$ converges locally uniformly to a holomorphic function in $D$ for any $\alpha \in \Oh^{\times}$ with $\alpha (0) = 1$. The next result gives more information if in addition $\alpha \in H^{\infty} (D)$ or $\alpha \in \Nh$. As usual $s$ denotes the number of generators of $\Sh$. 

\begin{theorem} \label{t925n}
 a) For $\alpha \in \Oh^{\times} \cap H^{\infty} (D)$ with $\alpha (0) = 1$ we have $\Psi_{\Sh} (\alpha) \in \Ah^1_s$. For $\alpha \in \Nh^1$ we have $\Psi_{\Sh} (\alpha) \in \Nh^1_s$ and $\alpha$ is atomless if and only if $\Psi_{\Sh} (\alpha)$ is atomless.\\
b) In proposition \ref{t52}, we have:
\[
 \Psi_{\Sh} Z (\Sh , \Nh^1) \subset H^0 (\Sh , \Nh^1_s) = \Big\{ f \in \Nh^1_s \tei f (z^N)^N = \prod_{\zeta^N=1} f (\zeta z) \quad \text{for} \; N \in \Sh \Big\} \; . 
\]
The function $\alpha \in Z (\Sh , \Nh^1)$ is atomless if and only if $\Psi_{\Sh} (\alpha)$ is atomless.
\end{theorem}

\begin{proof}
 Part b) is an immediate consequence of part a). As for a), if we have shown that $\Psi_{\Sh} (\alpha) \in \Nh^1_s$ then the assertion about atoms follows from proposition \ref{t53} together with theorem \ref{t7.19}. Note here that for every premeasure $\mu$ there is a constant $A > 0$ with $|\mu (C)| \le A$ for all arcs $C$, c.f. the remark before proposition \ref{t62}. Hence it suffices to prove the first claim of a). Writing $\alpha \in \Oh^{\times} \cap H^{\infty} (D)$ in the form $\alpha = \exp h$ with $h \in \Oh$ and $h (0) = 0$ we have to show that the function
\begin{equation}
 \label{eq:8.97}
\sum_{N \in \Sh} N^* h \quad \text{lies in} \; \Oh^-_{\kappa_s} \; .
\end{equation}
We first need a simple estimate for $\RRe h$. Writing $\alpha (z) = 1 + z r (z)$ it follows that $r \in \Oh (D)$ is bounded as well. Hence we have $|\alpha (z)| \le 1 + a|z|$ for some $a \ge 0$ and therefore
\[
 \RRe h (z) \le \log (1 + a |z|) \quad \text{for} \; z \in D \; .
\]
Hence we get
\begin{align}
 \RRe \sum_{N \in \Sh} N^* h & = \RRe \sum_{N \in \Sh} h (z^N) \notag \\
& \le \sum_{N \in \Sh} \log (1 + a |z|^N) \; . \label{eq:8.98}
\end{align}
In order to estimate this series we use the following version of summation by parts. Consider a function $\varphi$ on the integers $n \ge 1$ and a $C^1$-function $\psi$ on $[1,\infty)$. For $x \ge 1$ we set $M_{\varphi} (x) = \sum_{n \le x} \varphi (n)$. Then we have
\begin{equation} \label{eq:8.99}
 \sum_{n \le x} \varphi (n) \psi (n) = M_{\varphi} (x) \psi (x) - \int^x_1 M_{\varphi} (t) \psi' (t) \, dt \; .
\end{equation}
We apply this with 
\[
 \varphi (n) = |\{ (\nu_1 , \ldots , \nu_s) \tei \nu_1 , \ldots , \nu_s \ge 0 \; \text{and} \; N^{\nu_1}_1 \cdots N^{\nu_s}_s = n \} |
\]
and
\[
 \psi(x) = \log (1 + a |z|^x) \; .
\]
We have:
\begin{align*}
 M_{\varphi} (x) & = |\{ (\nu_1 , \ldots , \nu_s) \tei \nu_1 , \ldots , \nu_s \ge 0 \; \text{and} \; N^{\nu_1}_1 \cdots N^{\nu_s}_s \le x \} | \\
& \le a_1 (1+ \log^s x) \; \text{for} \; x \ge 1 \; .
\end{align*}
Here $a_1$ is a constant which depends only on $N_1 , \ldots , N_s$. 

Moreover $\psi (x) \le a |z|^x$ so that
\[
 \lim_{x\to \infty} M_{\varphi} (x) \psi (x) = 0 \quad \text{for every} \; z \in D \; .
\]
Hence we get the relations:
\begin{align*}
 \sum_{N \in \Sh} \log (1 + a |z|^N) & = \sum^{\infty}_{n=1} \varphi (n) \psi (n) \\
& = a \log \frac{1}{|z|} \int^{\infty}_1 M_{\varphi} (t) \frac{|z|^t}{1 + a |z|^t} \, dt \\
& \le a_1 \log \frac{1}{|z|} \int^{\infty}_1 (1+\log^s t) \frac{|z|^t}{1 + a |z|^t} \, dt \\
& = a_1 \int^{|z|}_0 \Big(1+ \log^s \Big( \frac{\log y}{\log |z|} \Big) \Big) \frac{1}{1+ay} \, dy \\
& \le a_1 \int^{|z|}_0 \Big( 1 + \log^s \Big( \frac{\log y}{\log |z|} \Big)  \Big) \, dy \; .
\end{align*}
The inequalities $0 < y \le |z| < 1$ give $\log y / \log |z| \ge 1$ and hence:
\begin{align*}
 0 \le \log \Big( \frac{\log y}{\log |z|} \Big) & = \log |\log y| - \log |\log |z|| \\
& \le |\log |\log y|| + |\log |\log |z||| \; .
\end{align*}
This implies the estimate
\[
 0 \le \log^s \Big( \frac{\log y}{\log |z|} \Big) \le 2^s |\log^s |\log y|| + 2^s |\log^s |\log |z||| \; .
\]
We have
\[
 \int^{|z|}_0 |\log^s |\log y|| \, dy \le \int^1_0 |\log^s (\log y^{-1})| \, dy = \int^{\infty}_0 |\log t|^s e^{-t} \, dt < \infty \; .
\]
This leads to the following estimate valid for $z \in D$:
\begin{equation} \label{eq:8.100}
 \sum_{N \in \Sh} \log (1 + a |z|^N) \le a_2 + a_3 |\log^s |\log |z||| \; .
\end{equation}
For $1/2 \le |z| < 1$ as in the definition \eqref{eq:8.90} of the class $\Oh^-_{\kappa_s}$ we have
\[
 |\log^s |\log |z||| \le |\log (1 - |z|)|^s\; .
\]
Together with \eqref{eq:8.98} and \eqref{eq:8.100} this shows that for $1/2 \le |z| <1$ we have
\[
 \RRe \sum_{N \in \Sh} h (z^N) \le a_4 |\log (1-|z|)|^s \; .
\]
Hence we have shown that assertion \eqref{eq:8.97} holds and therefore theorem \ref{t925n} is proved.
\end{proof}

Using the theory from the last section we can now say something about the image of $\Psi_{\Sh}$ on measures and in particular on the space $Z (\Sh , M^0 (\T))$.

\begin{cor} \label{t822}
 a) For any $\sigma \in M^0 (\T)$ we have $\Psi_{\Sh} (\sigma) \in P_{\kappa_s} (\T)$. \\
b) In proposition \ref{t63} we have:
\begin{equation} \label{eq:8.93}
 \Psi_{\Sh} Z (\Sh , M^0 (\T)) \subset \{ \mu \in P_{\kappa_s} (\T) \tei N_* \mu = \mu \quad \text{for} \; N \in \Sh \} \; .
\end{equation}
Moreover
\begin{equation*} 
 \Psi_{\Sh} Z (\Sh , M^0_c (\T)) \subset \{ \mu \in P_{\kappa_s} (\T) \tei \mu \; \text{is atomless and} \; N_* \mu = \mu \; \text{for alle} \; N \in \Sh \} \; .
\end{equation*}
\end{cor}

\begin{rem}
For $\sigma \in M^0 (\T)$ i.e. $\hsigma \in V^0 (\T)$ we have seen in \eqref{eq:6.69}\,ff that the premeasure $\mu = \Psi_{\Sh} (\sigma)$ is given by the absolutely convergent series
\[
 \mu (C) = \sum_{N \in \Sh} (N^* \sigma) (C) \quad \text{for} \; C \in \Kh \;.
\]
It would be interesting to prove a) directly by showing that $\sum_{N \in \Sh} N^* \sigma$ converges in the Banach space $P_{\kappa_s} (\T)$. Possibly the work \cite{CK} could be helpful for this. Instead we apply Korenblum's theory in the form of theorem \ref{t821} above together with Theorem \ref{t925n}.
\end{rem}

\begin{proof}
We first verify the formula:
\begin{equation}
 \label{eq:8.95}
T_{\Psi_{\Sh} (\sigma)} = \sum_{N\in \Sh} T_{N^* (\sigma)} \quad \text{in} \; \Dh' (\T) \; .
\end{equation}
This is done as follows where $D = \frac{d}{d\theta}$:
\[
 T_{\Psi_{\Sh} (\sigma)} = 2\pi D T_{\widehat{\Psi_{\Sh} (\sigma)} \lambda} = 2\pi D T_{\Psi_{\Sh} (\hsigma) \lambda} \; .
\]
Since the series \eqref{eq:6.59} for $\Psi_{\Sh} (\hsigma)$ is uniformly convergent and $D$ is continuous on $\Dh' (\T)$ we get
\begin{align*}
 T_{\Psi_{\Sh} (\sigma)} & = 2 \pi D \sum_{N \in \Sh} T_{N^{-1} N^* (\hsigma) \lambda} = \sum_{N \in \Sh} 2\pi D T_{\widehat{N^* (\sigma)}\lambda} \\
& = \sum_{N \in \Sh} T_{N^* (\sigma)} \; .
\end{align*}
Applying \eqref{eq:8.95} to the test function $k_z$ we get
\begin{equation} 
 \label{eq:8.96}
h_{\Psi_{\Sh} (\sigma)} = T_{\Psi_{\Sh} (\sigma)} (k_z) = \sum_{N \in \Sh} h_{N^* (\sigma)} \overset{\eqref{eq:4.36}}{=} \sum_{N \in \Sh} N^* h_{\sigma} \; .
\end{equation}
Hence we have $f_{\Psi_{\Sh}(\sigma)} = \Psi_{\Sh} (f_{\sigma})$. The function $f_{\sigma}$ is in $\Nh^1$ by \eqref{eq:562n}. Using theorem \ref{t925n} a) we get $f_{\Psi_{\Sh}} (\sigma) \in \Nh^1_s$ i.e. $h_{\Psi_{\Sh}} (\sigma) \in \Oh_{\kappa_s}$. Now theorem \ref{t821} implies that $\Psi_{\Sh} (\sigma) \in P_{\kappa_s} (\T)$.
 
Part b) follows from part a) and proposition \ref{t63}.
\end{proof}

The arguments in the proof give the following compatibilities between maps from proposition \ref{t52}, corollary \ref{t54} and proposition \ref{t63}.

\begin{cor} \label{t923}
 We have the following commutative diagrams where the vertical maps send a (pre-)measure $\mu$ to $f_{\mu}$
\[
 \xymatrix{
H^0 (\Sh , P_{\kappa_s} (\T)) \ar[d]^{\wr} & Z (\Sh , M^0 (\T)) \ar@{_{(}->}[l]_-{\Psi_{\Sh}} \ar[r]^-{\overset{\pi}{\sim}} \ar[d]^{\wr} & M^0 (\T) / M^0 (\T) \Sh \ar[d]^{\wr} \\
H^0 (\Sh , \Nh^1_s) &  Z (\Sh , \Nh^1) \ar@{_{(}->}[l]_-{\Psi_{\Sh}} \ar[r]^{\overset{\pi}{\sim}} & \Nh^1 / \Nh^{1\Sh}
}
\]
and
\[
 \xymatrix{
H^0 (\Sh , P_{\kappa_s, c} (\T)) \ar[d]^{\wr} & Z (\Sh , M^0_{c} (\T)) \ar@{_{(}->}[l]_-{\Psi_{\Sh}} \ar[r]^-{\overset{\pi}{\sim}} \ar[d]^{\wr} & M^0_{c} (\T) / M^0_{c} (\T) \Sh  \ar[d]^{\wr} \\
H^0 (\Sh , \Nh^1_{s,c}) & Z (\Sh , \Nh^1_c) \ar@{_{(}->}[l]_-{\Psi_{\Sh}} \ar[r]^{\overset{\pi}{\sim}} & \Nh^1_c / \Nh^{1\Sh}_c \; .
}
\]
The indices ``$c$'' in the second diagram refer to atomless (pre-)measures resp. functions. For $M \ge 2$ coprime to all the generators $N_1 , \ldots , N_s$ of $\Sh$ the actions of $M_* , M^*$ and $\Tr_M$ on $M (\T) , P (\T)$ and $\Oh^1$ respect all the groups and maps in these diagrams.  
\end{cor}


\begin{proof}
 It remains  to check the last assertion. It follows from the relations $\varphi_N \varphi_M = \varphi_M \varphi_N$ and $\Tr_N \Tr_M = \Tr_M \Tr_N$ in $\Rh$ for $N, M \ge 1$ and the relation $\varphi_N \Tr_M = \Tr_M \varphi_N$ if $N$ and $M$ are coprime. Also note formulas \eqref{eq:6.63}, \eqref{eq:7.78n} and \eqref{eq:7.76}.
\end{proof}

By corollary \ref{t923}, given a measure $\sigma \in M^0 (\T)$ the premeasure $\mu = \Psi_{\Sh} (\pi^{-1} (\sigma)) = \Psi_{\Sh} \Omega_{\Sh} (\sigma)$ is $\Sh$-invariant and has $\kappa_s$-bounded variation. Let $M = N_s$ and let $\Sh'$ be the monoid generated by $N_1 , \ldots , N_{s-1}$. If we start with an $M$-invariant measure $\sigma$ then the premeasure $\mu = \Psi_{\Sh'} \Omega_{\Sh'} (\sigma)$ is still $\Sh$-invariant and now even of $\kappa_{s-1}$-bounded variation. In fact, we even have the following assertion: 

\begin{prop}
 \label{t925}
If $s \ge 1$, then for $M = N_s$ and $\Sh'$ as above and for $\sigma \in M^0 (\T)$ with $M_* \sigma = \sigma$ we have
\[
 \Psi_{\Sh'} (\Omega_{\Sh'} (\sigma)) = \Psi_{\Sh} (\Omega_{\Sh} (\sigma)) \; .
\]
Correspondingly, if a function $u \in \Nh^1$ satisfies the functional equation
\[
 u (z^M)^M = \prod_{\zeta^M = 1} u (\zeta z)
\]
then we have $\Psi_{\Sh'} \Omega_{\Sh'} (u) = \Psi_{\Sh} \Omega_{\Sh} (u)$.\\
\end{prop}

\begin{proof}
 By corollary \ref{t923} it suffices to prove the second assertion. By assumption we have $MM^* (u) = \Tr_M (u)$ i.e. $u X_s = ue_s$. Hence we get:
\begin{align*}
 \Psi_{\Sh} (\Omega_{\Sh} (u)) & = u \prod^s_{i=1} (1-e_i) \Psi_{\Sh} = u (1 - e_s) \Omega_{\Sh'} \Psi_{\Sh}  \\
& = u (1-X_s) \Omega_{\Sh'} \Psi_{\Sh} = u \Omega_{\Sh'} (1-X_s) \Psi_{\Sh} \; . 
\end{align*}
We have formally:
\begin{align*}
 (1-X_s) \Psi_{\Sh} & = (1-X_s) \sum_{\nu_1 , \ldots , \nu_s \ge 0} X^{\nu_1}_1 \cdots X^{\nu_s}_s  \\
& = (1-X_s) \sum^{\infty}_{\nu=0} X^{\nu}_s \Psi_{\Sh'} \\
& = \Big( 1 - \lim_{\nu\to \infty} X^{\nu}_s \Big) \Psi_{\Sh'} \; .
\end{align*}
The function $u \Omega_{\Sh'}$ is equal to $1$ at $z = 0$. Hence we have
\[
 \lim_{\nu\to\infty} (u \Omega_{\Sh'}) X^{\nu}_s = \lim_{\nu\to\infty} (u \Omega_{\Sh'}) (z^{N^{\nu}_s}) = (u \Omega_{\Sh'} ) (0) = 1 \; .
\]
Thus we get the following equation with pointwise limits:
\begin{align*}
 \Psi_{\Sh} (\Omega_{\Sh} (u)) & = \lim_{\nu \to \infty} (u \Omega_{\Sh'}) (1 - X^{\nu}_s) \Psi_{\Sh'} = \lim_{\nu\to\infty} \frac{u \Omega_{\Sh'}}{u \Omega_{\Sh'} X^{\nu}_s} \Psi_{\Sh'}\\
& = u \Omega_{\Sh'} \Psi_{\Sh'} = \Psi_{\Sh'} (\Omega_{\Sh'} (u)) \; .
\end{align*}
\end{proof}

We will now combine the preceeding maps from measures to $\Sh$-invariant premeasures with the passage to thin measures in corollary \ref{t938n}.

For every $\kappa$ the map
\[
 \fs : P_{\kappa} (\T) \longrightarrow M_{\kappa} (\T) \silo \tM_{\kappa} (\T) = K_0 (\tM^+_{\kappa} (\T))
\]
is left- and right $\Sh$-equivariant. In particular it maps the subspace $H^0 (\Sh , P_{\kappa} (\T))$ of left $\Sh$-invariant premeasures to the space
\[
 H^0 (\Sh , \tM_{\kappa} (\T)) = \{ \mu \in \tM_{\kappa} (\T) \tei N_* \mu =\mu \; \text{for} \; N \in \Sh \}
\]
of left $\Sh$-invariant $\kappa$-thin real ``measures''. By definition, $ \mu = [\mu_1 , \mu_2] \in \tM_{\kappa} (\T)$ with $\mu_1 , \mu_2$ in $M^+_{\kappa} (\T)$ satisfies $N_* \mu = \mu$ if and only if $N_* \mu_1 + \mu_2 = \mu_1 + N_* \mu_2$. If $\mu_1 (\T)$ or $\mu_2 (\T)$ is finite then by proposition \ref{t33n} the positive measures $\mu_{\pm}$ in the Jordan decomposition of the signed measure $\mu_1 - \mu_2 = \mu_+ - \mu_-$ are both $N$-invariant: $N_* \mu_{\pm} = \mu_{\pm}$. 

The inverse of each isomorphism $\pi$ in corollary \ref{t923} is given by left (or right) multiplication with $\Omega_{\Sh} \in R_{\Q}$ c.f.~proposition~\ref{t52}. We thus consider the composition
\[
 \Lambda_{\Sh} : M^0 (\T) / M^0 (\T) \Sh \overset{\Omega_{\Sh}}{\silo} Z (\Sh , M^0 (\T)) \overset{\Psi_{\Sh}}{\hookrightarrow} H^0 (\Sh , P_{\kappa_s} (\T)) \xrightarrow{\fs} H^0 (\Sh , \tM_{\kappa_s} (\T))\, .
\]
We will also view $\Lambda_{\Sh}$ as a map defined on $M^0 (\T)$. For $s = 0$ i.e. $\Sh = \{ 1 \}$ we set $M^0 (\T) \Sh = 0$ and the map $\Lambda_{\Sh}$ becomes the following by proposition \ref{t35n}
\[
\Lambda_{\Sh} = \fs : M^0 (\T) \longrightarrow \tM_{\kappa_0} (\T) = M (\T)_{\sing} , \mu \longmapsto \mu_{\sing} \; .
\]
Restricted to the image $M^0_{c} (\T) / M^0_{c} (\T) \Sh$ of $M^0_{c} (\T)$ in $M^0 (\T) / M^0 (\T) \Sh$ the map $\Lambda_{\Sh}$ is the composition:
\[
 \Lambda_{\Sh} : M^0_{c} (\T) / M^0_{c} (\T) \Sh \overset{\Omega_{\Sh}}{\silo} Z (\Sh , M^0_{c.} (\T)) \overset{\Psi_{\Sh}}{\hookrightarrow} H^0 (\Sh , P_{\kappa_s, c} (\T)) \xrightarrow{\fs} H^0 (\Sh , \tM_{\kappa_s, c} (\T)) \; .
\]
In particular, $\Lambda_{\Sh}$ maps atomless measures to atomless ``measures''. Denoting by $\Phi_{\Sh}$ the map of right multiplication by $\Phi_{\Sh} \in \Rh_{\Q}$ we obtain the following commutative diagram:
\[
 \xymatrix{
M^0 (\T) / M^0 (\T) \Sh \ar@{_{(}->}[d] \ar@<1.5ex>[r]^{\overset{\Omega_{\Sh}}{\sim}} & Z (\Sh , M^0 (\T)) \ar@<1.5ex>[l]_{\pi} \ar@{_{(}->}[d]  \ar[r]^{\Psi_{\Sh}} & H^0 (\Sh , P_{\kappa_s} (\T)) \ar[d]^{\Phi_{\Sh}} \ar[r]^{\fs} & H^0 (\Sh , \tM_{\kappa_s} (\T)) \ar^{\Phi_{\Sh}}[d]\\
P_{\kappa_s} (\T) / P_{\kappa_s} (\T) \Sh \ar[r]^{\overset{\Omega_{\Sh}}{\sim}} & Z (\Sh , P_{\kappa_s } (\T)) \ar@{=}[r] & Z (\Sh , P_{\kappa_s} (\T)) \ar[r]^{\fs} & Z (\Sh , \tM_{\kappa_s} (\T)) 
}
\]
The middle square is commutative since by proposition \ref{t63} we have $\Phi_{\Sh} \verk \Psi_{\Sh} = \id$. By proposition \ref{t52} the maps $\Omega_{\Sh}$ are isomorphisms and hence the obvious injectivity of the second vertical arrow implies that the first one is injective as well. By $\fs$  we also denote the composition
\[
\fs : M^0 (\T) \hookrightarrow P_{\kappa_s } (\T) \xrightarrow{\fs} \tM_{\kappa_s} (\T) 
\]
and its restriction
\[
 \fs : Z (\Sh , M^0 (\T)) \longrightarrow Z (\Sh , \tM_{\kappa_s} (\T)) \; .
\]
Recall the definition \ref{t37n} of the space $M^{\kappa_s} (\T)$. 

\begin{theorem} \label{t47nnn}
 1) For $s \ge 1$ consider the map $\Lambda_{\Sh}$ defined above
\[
 \Lambda_{\Sh} : M^0 (\T) / M^0 (\T) \Sh \longrightarrow H^0 (\Sh , \tM_{\kappa_s} (\T)) \; .
\]
We have
\[
 \Phi_{\Sh} \verk \Lambda_{\Sh} \circ \pi = \fs : Z (\Sh , M^0 (\T)) \longrightarrow Z (\Sh , \tM_{\kappa_s} (\T)) \; .
\]
In particular, $\Lambda_{\Sh}$ is injective on the image
\[
 M^{\kappa_s} (\T) / M^{\kappa_s} (\T) \Sh \; \text{of} \; M^{\kappa_s} (\T) \; \text{in} \; M^0 (\T) / M^0 (\T) \Sh \; .
\]
Corresponding statements hold on the spaces of atomless measures.

2) If $s \ge 1$ set $M = N_s$ and let $\Sh' \subset \Sh$ be the submonoid generated by $N_1 , \ldots , N_{s-1}$. For $\sigma \in M^0 (\T)$ with $M_* \sigma = \sigma$ we have
\[
 \Lambda_{\Sh'} (\sigma) \in H^0 (\Sh , \tM_{\kappa_{s-1}} (\T)) \; .
\]
Under the restriction map of proposition \ref{t31n}
\[
\res_{s-1,s} : \tM_{\kappa_{s-1}} (\T) \longrightarrow \tM_{\kappa_s} (\T)
\]
we have
\[
 \Lambda_{\Sh} (\sigma) = \res_{s-1,s} (\Lambda_{\Sh'} (\sigma)) \; .
\]
3) Every $\Sh$-invariant positive ergodic measure $\mu$ that is non-zero on some $\kappa_s$-Carleson set is in the image of $\Lambda_{\Sh}$. 
\end{theorem}

\begin{proof}
1) The equation $\Phi_{\Sh} \verk \Lambda_{\Sh} \verk \pi = \fs$ follows from the above diagram. Proposition \ref{t15n} shows that the left vertical map is an isomorphism in the following diagram
\[
 \xymatrix{
Z (\Sh , M^{\kappa_s} (\T)) \ar@{^{(}->}[r]  \ar[d]^{\wr}_{\pi} & Z (\Sh , M^0 (\T)) \ar[d]^{\wr \pi} \\
M^{\kappa_s} (\T) / M^{\kappa_s} (\T) \Sh \ar[r]^i & M^0 (\T) / M^0 (\T) \Sh
}
\]
Hence the map $i$ is injective. This means that the image of $M^{\kappa_s} (\T)$ in $M^0 (\T) / M^0 (\T) \Sh$ can be identified with $M^{\kappa_s} (\T) / M^{\kappa_s} (\T) \Sh$. The map $\Lambda_{\Sh}$ is injective on this subspace since the restriction of $\fs : M^0 (\T) \to \tM_{\kappa_s} (\T)$ to $M^{\kappa_s} (\T)$ is injective. 

Part 2) follows from proposition \ref{t925} since the following diagram commutes:
\[
 \xymatrix{
P_{\kappa_{s-1}} (\T) \ar[r]^{\fs} \ar@{_{(}->}[d] & M_{\kappa_{s-1}} (\T) \ar[r]^{\sim} \ar[d]^{\res_{s-1,s}} & \tM_{\kappa_{s-1}} (\T) \ar[d]^{\res_{s-1,s}}\\
P_{\kappa_s} (\T) \ar[r]^{\fs} & M_{\kappa_s} (\T) \ar[r]^{\sim} & \tM_{\kappa_s} (\T) \; .
}
\]
As for 3), applying 2) iteratively to $\sigma = \mu - \mu (\T) \lambda$ we get
\[
 \Lambda_{\Sh} (\sigma) = \res_{0,s} (\Lambda_{\{ 1 \} } (\sigma) ) = \res_{0,s} (\sigma_{\sing}) = \res_{0,s} (\mu_{\sing})
\]
where $\res_{0,s}$ is the map
\[
 \res_{0,s} : \tM_{\kappa_0} (\T) = M (\T)_{\sing} \xrightarrow{\fs} M_{\kappa_s} (\T) \cong \tM_{\kappa_s} (\T) \; .
\]
According to proposition \ref{t939n} the measure $\mu$ is $\kappa_s$-thin and hence we have $\mu = \mu_{\sing}$. The measure $\Lambda_{\Sh} (\sigma) = \res_{s,0} (\mu)$ is the unique extension of $\mu \, |_{\Bh_{\kappa_s}}$ to a $\kappa_s$-thin measure. Since $\mu$ itself is such an extension, we have $\Lambda_{\Sh} (\sigma) = \mu$. 
\end{proof}

\begin{rems}
{\bf a} According to corollary \ref{t923} the first part $\Psi_{\Sh} \verk \Omega_{\Sh}$ of the map $\Lambda_{\Sh}$ can be interpreted as a map on functions
\[
 \Psi_{\Sh} \verk \Omega_{\Sh} : \Nh^1 / \Nh^{1\Sh} \longrightarrow H^0 (\Sh , \Nh^1_s) \; .
\]
However I do not know a function theoretic construction of the passage to $\kappa_s$-thin ``measures'' via $\fs$.

{\bf b} It is known that there is an abundance of $\varphi_3$-invariant, even ergodic atomless probability measures $\nu$ on $\T$ for which the $3$-adic Cantor set $E \subset \T$ has full measure. We can also assume that $[\zeta]_* \sigma \neq \sigma$ for $\zeta = -1$. For $\sigma = \nu - \nu (\T) \lambda \in M^0 (\T)$ we then have $\Omega_{\Sh'} (\sigma) \neq 0$ i.e. $\sigma \notin M^0 (\T) \Sh'$ for $\Sh'$ the semigroup generated by $N_1 = 2$. Let $\Sh$ be generated by $N_1 = 2$ and $N_2 = 3$. Part 2) of theorem \ref{t47nnn} implies that $\mu = \Lambda_{\Sh'} (\sigma)$ is an $\Sh$-invariant atomless $\kappa_1$-thin real ``measure'' on $\T$. It is non-zero by part 1) since $\sigma \in M^{\kappa_1} (\T)$ by Example \ref{t826n} and since $\sigma \notin M^0 (\T) \Sh'$. Incidentally, $\sigma \in M^{\kappa_{\gamma}} (\T)$ for every $\gamma \ge 0$. 
\end{rems}

For $\mu \in M^+ (\T)$ write $\dim_H \mu$ for the infimum over the Hausdorff dimensions of all Borel sets of full measure for $\mu$. Works of Billingsley and Rudolph give information about $\dim_H \mu$ for certain $\Sh$-invariant measures if $s \ge 2$. 

\begin{theorem} \label{t1052}
 Let $\mu \in M^+ (\T)$ be $\Sh$-invariant and ergodic with respect to $\varphi_N$ where $N = N_1$. Assume that $s \ge 2$ and that $\mu$ is not a scalar multiple of Lebesgue measure. Then $\dim_H \mu = 0$.
\end{theorem}

\begin{proof}
 If $\mu$ has an atom then by ergodicity it is purely atomic and hence supported on a set of Hausdorff dimension zero. Hence we may assume that $\mu$ is atomless. Since $\mu$ is $\varphi_N$-ergodic by assumption, we have the formula:
\[
 \dim_H \mu = \frac{h_{\mu} (\varphi_N)}{\log N} \; .
\]
This is a result of Billingsly \cite{Bi}, see also \cite{GP} \S\,2. If $h_{\mu} (\varphi_N) > 0$ Rudolph's theorem \cite{R} implies that $\mu$ is a scalar multiple of Lebesgue measure. Since we excluded this case it follows that $\dim_H \mu = 0$.
\end{proof}

Note that we used only the special case of Rudolph's theorem where $\mu$ is already ergodic with respect to one of the generators of the semigroup. In this case his result is easier to prove using Fourier theory.

It is not known whether a non-atomic measure as in the theorem exists at all. If it does it must be carried by a ``small'' Borel set. Hence the following question seems reasonable:

{\it Question.} Given a probability measure $\mu \ge 0$ on $\T$ with $\dim_H \mu = 0$, is $\mu$ non-zero on some $\kappa_{\gamma}$-Carleson set for some $\gamma > 0$?

Part of theorem \ref{t33} can be generalized to functions in the classes $\Nh_{\gamma}$.

\begin{theorem} \label{t926}
 For $0 \le \gamma < s$ let $f \in \Nh_{\gamma}$ be a function which satisfies the functional equations
\[
 f (z^N)^N = \prod_{\zeta^N = 1} f (\zeta z) \quad \text{for all} \; N \in \Sh \; .
\]
If $f$ has a zero in $D$ then $f$ is the zero function.
\end{theorem}

\begin{proof}
 From the proof of theorem \ref{t33} we know that $f (0) = 0$ implies that $f$ is the zero function. Assume that $f (a) = 0$ for some $0 \neq a \in D$ and that $f$ is not identically zero. Using relation \eqref{eq:3.11} for $N \in \Sh$ we get the following estimate for all $0< r < 1$
\begin{align*}
 \sum_{|z| \le r} (1 - |z|) \ord_z f & \ge \sum_{N \in \Sh} \sum_{|z| \le r \atop z^N = a} (1 - |z|) \ord_z f  \\
& = \sum_{N \in \Sh} (1 - |a|^{1/N}) \sum_{|z| \le r \atop z^N = a} \ord_z f \\
& \overset{\eqref{eq:3.11}}{=} \ord_a f \sum_{N \in \Sh \atop |a| \le r^N} N (1 - |a|^{1/N}) \\
& \ge C_1 | \{ N \in \Sh \tei |a| \le r^N \}| \; .
\end{align*}
Here $C_1$ is the positive constant
\[
 C_1 = \ord_a f \inf_{N \ge 1} N (1 - |a|^{1/N}) \; .
\]
We may assume that $N_1 \ge N_i$ for $1 \le i \le s$. For $x \ge 1$ and $0 \le \nu_i \le s^{-1} \log_{N_1} x$ we then have
\[
 N^{\nu_1}_1 \cdots N^{\nu_s}_s \le N^{\nu_1 + \ldots + \nu_s}_1 \le N^{\log_{N_1} x}_1 = x \; .
\]
It follows that
\begin{align*}
 |\{ N \in \Sh \tei N \le x \} | & = | \{ ( \nu_1 , \ldots , \nu_s) \tei \nu_1 , \ldots , \nu_s \ge 0 \; \text{and} \; N^{\nu_1}_1 \cdots N^{\nu_s}_s \le x \} |  \\
& \ge s^{-s} \log^s_{N_1} x \; .
\end{align*}
Hence for $|a| \le r < 1$ and $e^{-1} < r$ we have
\begin{align*}
 |\{ N \in \Sh \tei |a| \le r^N \} | & \ge s^{-s} \log^s_{N_1} \Big( \frac{\log |a|}{\log r} \Big) \\
& = s^{-s} ( \log_{N_1} |\log |a|| - \log_{N_1} |\log r|)^s  \\
& = s^{-s} (\log_{N_1} |\log |a|| + |\log_{N_1} |\log r||)^s \; .
\end{align*}
Hence there is a positive constant $C_2$ such that for $r < 1$ close enough to $1$ we have
\[
 \sum_{|z| \le r} (1 - |z|) \ord_z f \ge C_2 |\log (1-r)|^s \; .
\]
Comparing this estimate with the one in theorem \ref{t22} we see that our assumptions $f (a) = 0$ and $f \neq 0$ imply that $s \le \gamma$.
\end{proof}

In theorem \ref{t33} we have seen that there are neither non-constant functions with zeroes nor non-constant outer functions in $\Nh = \Nh_0$ which satisfy the functional equation \eqref{eq:1.1}. The first assertion has just been generalized to the classes $\Nh_{\gamma}$ for $\gamma \ge 0$ in theorem \ref{t926}. Replacing ``outer'' by ``cyclic'' c.f. end of \S\,8, I think the following generalization of the second assertion is true:

\begin{conj}
 \label{t1054}
For $s \ge 1$ and $0 \le \gamma < s$ let $g \in \Ah_{\gamma}$ be cyclic and satisfy the functional equations:
\[
 g (z^N)^N = \prod_{\zeta^N = 1} g (\zeta z) \quad \text{for all} \; N \in \Sh \; .
\]
Then $g$ is constant.
\end{conj}

\begin{rems}
 {\bf a} In $\Ah_0$ the cyclic elements $g$ are units because $\Ah_0 = H^{\infty} (D)$ is a unital Banach algebra. Hence $g$ is outer and thus it follows from theorem \ref{t33} that the conjecture is true for $\gamma = 0$.

{\bf b} The condition $\gamma < s$ in the conjecture is necessary: According to proposition \ref{a31} below, for $N \ge 2$ the function
\[
 g (z) = \prod^{\infty}_{\nu=0} \exp (2 z^{N^{\nu}})
\]
is a unit in $\Ah_1 = \Ah^{-\infty}$ hence cyclic and it satisfies the multiplicative functional equation \eqref{eq:1.1}. Hence for $\gamma = 1 = s$ the assertion in conjecture \ref{t1054} no longer holds. Note that according to example \ref{t513} the function $g$ is not in the Nevanlinna class. 

{\bf c} For $\rho \in D$ the evaluation map $\Ah_{\gamma} \to \C$ sending $g$ to $g (\rho)$ is continuous. Its kernel is a closed maximal ideal $\emm$. If $g \in \Ah_{\gamma}$ has a zero at $\rho$ then $g \Ah_{\gamma} \subset \emm$ and hence $g$ cannot be cyclic. It is therefore enough to prove the conjecture for elements $g \in \Ah^1_{\gamma}$ which are cyclic in $\Ah_{\gamma}$. For $\gamma = 1$ these are the functions with $\sigma_g = 0$ by theorem \ref{t37n}. For general $\gamma > 0$ this should also be true by conjecture \ref{t828}. 
\end{rems}

The following assertion is more of a program: 

\begin{theorem} \label{t1055}
 Fix an integer $s \ge 1$ and a real number $0 < \gamma < s$. Assume that the characterization of cyclic elements in $\Ah_{\gamma}$ given in conjecture \ref{t41a} is true. (This is satisfied for $\gamma = 1 < s$ by theorem \ref{t828}.) Assume also that conjecture \ref{t1054} holds for functions $g$ which are in addition singular inner. Let $\mu \neq 0$ be a finite positive singular measure on $\T$ which is $\Sh$-invariant and ergodic. Then $\mu$ is $\kappa_{\gamma}$-thin. 
\end{theorem}

\begin{rems}
 Thus we would almost get the assumption of 3) in theorem \ref{t47nnn} which would ensure that $\mu$ lies in the image of $\Lambda_{\Sh}$. If conjecture \ref{t1054} holds for $\gamma = s$ as well if in addition $g$ is singular inner, then we would get exactly the assumption of 3), and hence $\mu$ would lie in the image of $\Lambda_{\Sh}$. 
\end{rems}

\begin{proof}
 Let $g = \exp (-h_{\mu})$ be the singular inner function corresponding to $\mu$. Set $f = g / g (0)$ viewed as an element of $\Ah^1_{\gamma}$. Then we have $\sigma_f = \mu \, |_{\Bh_{\kappa_{\gamma}}}$. If $\sigma_f$ vanishes then conjecture \ref{t41a} implies that $f$ and hence $g$ is cyclic in $\Ah_{\gamma}$. Since $\mu$ is $\Sh$-invariant, the singular inner function $g$ therefore satisfies all assumptions in conjecture \ref{t1054}. Thus $g$ would have to be constant in contradiction to $\mu$ being singular and non-zero. Hence we have $\sigma_f \neq 0$ i.e. there is a $\kappa_{\gamma}$-Carleson set $F$ with $\mu (F) > 0$. Now proposition \ref{t939n} implies that $\mu$ is $\kappa_{\gamma}$-thin.
\end{proof}

Using theorem \ref{t828}, we formulate a special case of the theorem as a corollary:

\begin{cor} \label{t1056}
 Let $N ,M \ge 2$ be coprime integers. Assume that any singular inner function $g$ which is cyclic in $\Ah_1 = \Ah^{-\infty}$ and satisfies the multiplicative functional equations \eqref{eq:1.1} for $N$ and $M$ is constant. Let $\mu \neq 0$ be a finite positive singular measure on $\T$ with $N_* \mu = \mu = M_* \mu$ which is $\Sh = \langle N,M \rangle$-ergodic. Then $\mu$ is $\kappa_1$-thin. 
\end{cor}

\section*{Appendix: Premeasures and Witt vectors}

This section contains some elementary formal relations between premeasures on the circle and the ring of big Witt vectors of $\C$. We also show that for every prime number $p$, the Artin--Hasse exponential $E_p$ can be viewed as a $\varphi_p$-invariant premeasure on $\T$ of $\kappa_1$-bounded variation which is not a measure.

The real valued distributions on $\T$ form a commutative $\R$-algebra under the convolution product
\[
 T_1 * T_2 = \Sigma_* (T_1 \otimes T_2) \; .
\]
Here $\Sigma : \T \times \T \to \T$ is the sum in the abelian group $\T$, c.f. \cite{Sch}. The measures $M (\T)$ form a subalgebra of $\Dh' (\T)$. The premeasures $P (\T)$ are not a subalgebra of $\Dh' (\T)$ but they form an $M (\T)$-submodule. More precisely, for $\mu \in P (\T)$ and $\nu \in M (\T)$ the convolution $\mu * \nu$ is a premeasure given by the cumulative mass function
\begin{equation} \label{eq:100}
 \widehat{\mu * \nu} (\theta) = \int_{\T} \hmu (\theta - t) \, d\nu (t) - \int_{\T} \hmu (-t) \, d\nu (t) \; .
\end{equation}
This is seen as follows. Let $\varphi \in \Dh (\T)$ be a test function. We have:
\begin{align*}
 \langle T_{\mu} * T_{\nu} , \varphi \rangle & = \langle T_{\mu} \otimes T_{\nu} , \varphi (x+t) \rangle \\
& = \langle T_{\mu} , \int_{\T} \varphi (x+t) \, d\nu (t) \rangle \\
& = 2 \pi \langle D T_{\hmu \lambda} , \int_{\T} \varphi (x+t) \, d\nu (t) \rangle \\
& = - 2 \pi \int_{\T} \int_{\T} \varphi' (x+t) \hmu (x) \, d\lambda (x) \, d\nu (t) \\
& = - 2\pi \int_{\T} \int_{\T} \varphi' (\theta) \hmu (\theta - t) \, d\lambda (\theta) \, d\nu (t) \\
& = 2\pi \langle DT_{q\lambda} , \varphi\rangle = 2 \pi \langle DT_{(q -q (0))\lambda} , \varphi \rangle
\end{align*}
Here we have set
\[
 q (\theta) = \int_{\T} \hmu (\theta - t) \, d\nu (t) \; .
\]
By Lebesgue's dominated convergence theorem, we have $q (\theta) - q (0) \in V_{\pre} (\T)$. Hence the distribution $\mu * \nu \equiv T_{\mu} * T_{\nu}$ is the distribution attached to the premeasure with cumulative mass function $q (\theta) - q (0)$ as claimed in \eqref{eq:100}. 

For a commutative ring $\Lambda$ with unit, the (big) Witt vector ring $W (\Lambda)$ c.f. \cite{H}~III~17.2 has underlying set
\[
W (\Lambda) = 1 + T\Lambda [[T]] \; .
\]
The addition $\oplus$ in $W (\Lambda)$ is given by multiplication of formal power series. Thus $1$ is the zero element. In general the multiplication $\odot$ in $W (\Lambda)$ is less easy to describe. For $\Q$-algebras $\Lambda$ the situation is simple though. In this case the ``ghost map''
\[
 \gamma : W (\Lambda) \longrightarrow z \Lambda [[z]] \; , \; \gamma (P) = -z P' / P
\]
is an isomorphism of abelian groups and the multiplication in $W (\Lambda)$ is defined so that $\gamma$ becomes a ring isomorphism if $z \Lambda [[z]]$ is given the Hadamard product $*$ of coefficientwise multiplication:
\[
 \sum^{\infty}_{\nu =1} a_{\nu} z^{\nu} * \sum^{\infty}_{\nu =1} b_{\nu} z^{\nu} = \sum^{\infty}_{\nu = 1} a_{\nu} b_{\nu} z^{\nu} \; .
\]
In this way one gets universal polynomials for the coefficients of the product $P \odot Q$ in $W (\Lambda)$ for $P, Q \in W (\Lambda)$. It is a remarkable fact that these polynomials have coefficients in $\Z$ and define a ring structure on $W (\Lambda)$ for any ring $\Lambda$. In general the ghost map is still a ring homomorphism but it is no longer an isomorphism if $\Lambda$ is not a $\Q$-algebra. The Teichm\"uller character is the multiplicative map
\[
 [ \, ] : \Lambda \longrightarrow W (\Lambda) \quad \text{mapping $a$ to} \; [a] = 1-aT \; .
\]
Note that $[a] \odot P = P (az)$ in $W (\Lambda)$. We set $\Tr_N = \sum^{\oplus}_{\zeta^N=1} [\zeta]$ viewed as a sum of multiplication operators on  $W (\Lambda)$. 

For $N \ge 1$ the Frobenius endomorphism $F_N$ is the ring endomorphism of $W (\Lambda)$ defined by the formula
\[
 F_N (P) (z^N) = \prod_{\zeta^N = 1} P (\zeta z) = \Tr_N (P) (z) \; .
\]
The Verschiebung $V_N$ is the additive endomorphism of $W (\Lambda)$ given by
\[
 V_N (P) (z) = P (z^N) \; .
\]
The following relations are standard:
\begin{gather}
 F_N \verk F_M = F_{NM} = F_M \verk F_N \; , \; V_N \verk V_M = V_{NM} = V_M \verk V_N \label{eq:101} \\
F_N \verk V_N = N \; , \; V_N \circ F_N = \Tr_N \label{eq:102} \\
F_N \verk V_M = V_M \verk F_N \quad \text{if $N$ is prime to $M$.} \label{eq:103}
\end{gather}

\begin{prop}
 \label{t23}
The map $w : \Dh' (\T) \to W (\C)$ defined by $w (T) = \exp (-\pi i G_T)$ is a ring homomorphism. For $\zeta \in S^1$ we have $w (\delta_{\zeta}) = [\zeta^{-1}]$. Under restriction of $w$ to $M (\T)$ and $P (\T)$ the following endomorphisms correspond:

\begin{tabular}{cp{3cm}c}
 on $M (\T)$ resp. $P (\T)$ & & on $W (\C)$ \\
$[\zeta^{-1}]_*$ & & $[\zeta] \odot \_$ \\
$N_*$ & & $F_N$ \\
$N N^*$ & & $V_N$ 
\end{tabular}
\end{prop}

\begin{proof}
 By definition of addition in $W (\C)$ as multiplication of power series it is clear that the map $w$ is additive. Its composition with the ghost map $\gamma$ is given by
\[
 (\gamma w) (T) = \sum^{\infty}_{\nu = 1} c_{\nu} (T) z^{\nu} \; .
\]
Since $c_{\nu} (T_1  * T_2) = c_{\nu} (T_1) c_{\nu} (T_2)$ it follows that $\gamma w$ is multiplicative and hence also $w$. We have
\[
 G_{\delta_{\zeta}} (z) = - \frac{1}{\pi i} (\log (\zeta - z) - \log \zeta)
\]
and hence
\[
 w (\delta_{\zeta}) = \exp (-\pi i G_{\delta_{\zeta}} (z)) = 1 - \zeta^{-1} z = [\zeta^{-1}] \; .
\]
The relation $h_{[\zeta^{-1}]_* \mu} (z) = h_{\mu} (\zeta z)$ implies that we have
\[
 w ([\zeta^{-1}]_* \mu) = w (\mu) (\zeta z) = [\zeta] \odot w (\mu) \; .
\]
The relation
\[
 h_{N_* \mu} (z^N) = \frac{1}{N} \Tr_N (h_{\mu})
\]
for measures \eqref{eq:3.1} and premeasures \eqref{eq:7.77} implies the formula $G_{N_* \mu} (z^N) = \Tr_N G_{\mu}$ and hence the relation
\[
 w (N_* \mu) (z^N) = \prod_{\zeta^N = 1} w (\mu) (\zeta z) = F_N (w (\mu)) (z^N) \; .
\]
This implies $w (N_* \mu) = F_N (w (\mu))$. The formula $h_{N^* \mu} = N^* h_{\mu}$ for measures \eqref{eq:4.36} and premeasures \eqref{eq:7.76} implies that we have $G_{N N^* \mu} = G_{\mu} (z^N)$ and hence
\[
 w (NN^* \mu) = w (\mu) (z^N) = V_N (w (\mu)) \; .
\]
\end{proof}

Thus the formulas $N_* N^* \mu = \mu$ and $NN^* N_* \mu = \Tr_N (\mu)$ that we have used repeatedly and also the commutation of $N_*$ and $M^*$ for $(N,M) = 1$ correspond to the standard relations \eqref{eq:102} and \eqref{eq:103} between Frobenius and Verschiebung. Since a premeasure $\mu$ is uniquely determined by its Fourier coefficients $c_{\nu} (\mu)$ for $\nu \ge 1$, the map
\[
 w : P (\T) \hookrightarrow W (\C)
\]
is injective. We end this section with an example. For every prime number $p$ the Artin--Hasse exponential is the element
\[
 E_p = \exp \Big( \sum^{\infty}_{\nu =0} \frac{z^{p^{\nu}}}{p^{\nu}} \Big)
\]
of $W (\Q) \subset W (\C)$. It has the remarkable property that it is $p$-integral, i.e. $E_p \in W (\Z_{(p)})$, where $\Z_{(p)} = \{ m / n \in \Q \tei p \nmid n \}$. Moreover it is an idempotent i.e. $E_p \odot E_p = E_p$. We can ask whether $E_p$ or more generally the element
\[
 E_N = \exp \Big( \sum^{\infty}_{\nu=0} \frac{z^{N^{\nu}}}{N^{\nu}} \Big)
\]
for $N \ge 2$ is a premeasure. This is the case. 

\begin{prop} \label{a31}
 For any $N \ge 2$ there is a unique premeasure $\mu \in P (\T)$ with $w (\mu) = E_N$. We have $\mu \in P^+_{\kappa_1} (\T) \cap P^-_{\kappa_1} (\T)$ and $N_* \mu = \mu$ and $\mu * \mu = \mu$. The premeasure $\mu$ is atomless and it is not a (signed) measure. Its cummulative mass function and Herglotz transform are:
\[
 \hmu (\theta) = - \frac{1}{\pi} \sum^{\infty}_{\nu =0} N^{-\nu} \sin N^{\nu} \theta \quad \text{and} \quad h_{\mu} (z) = - 2 \sum^{\infty}_{\nu=0} z^{N^{\nu}} \; .
\]
The function
\[
 f_{\mu} = \exp (-h_{\mu}) = \prod^{\infty}_{\nu=0} \exp (2 z^{N^{\nu}})
\]
is a unit in $\Ah_1$ which is not in the Nevanlinna class $\Nh$. The $\kappa_1$-singular measure $\mu_s = \sigma_{f_{\mu}}$ vanishes. 
\end{prop}

\begin{proof}
 The uniqueness is clear since $\omega$ is injective on $P (\T)$. Consider the atomless signed measure $\sigma \in M^0 (\T)$ defined by $\sigma = - \pi^{-1} \cos \theta \, d\theta$. We have $\hsigma (\theta) = - \pi^{-1} \sin \theta$. Let $\Sh = \{ N^{\nu} \tei \nu \ge 0 \}$ be the monoid generated by $N$. By corollary \ref{t822} a) we know that $\mu := \Psi_{\Sh} (\sigma) \in P_{\kappa_1} (\T)$. By definition
\[
 \hmu = \sum^{\infty}_{\nu =0} N^{-\nu} (N^{-\nu})^* \hsigma 
\]
or more explicitly
\begin{equation}
 \label{eq:104}
\hmu (\theta) = - \frac{1}{\pi} \sum^{\infty}_{\nu=0}  N^{-\nu} \sin N^{\nu} \theta
\end{equation}
This is a continuous function on $\T$ and hence $\mu$ is atomless. By equation \eqref{eq:6.65} we have for $\eta \in \T$
\[
 \widehat{N_* \sigma} (\eta^N) =  \frac{1}{2\pi i} \sum_{\zeta^N = 1} (\zeta - \overline{\zeta}) - \frac{1}{2\pi i} \sum_{\zeta^N = 1} (\zeta \eta - \overline{\zeta \eta}) = 0 \; .
\]
Hence we have $N_* \sigma = 0$ and therefore $\sigma \in Z (\Sh , M^0 (\T))$. Now proposition \ref{t63} gives $\mu \in \Psi_{\Sh} (Z (\Sh , M^0 (\T))) \subset H^0 (\Sh , P (\T))$, i.e. $N_* \mu = \mu$.

We have
\[
 \mu_0 := \int_{\T} \hmu d\lambda = 0
\]
and hence formulas \eqref{eq:7.72} and \eqref{eq:104} give the equations
\begin{align*}
 G_{\mu} & = h_{\hmu \lambda} = 2 \sum^{\infty}_{n=1} c_n (\hmu \lambda) z^n \\
& = - \frac{1}{\pi i} \sum^{\infty}_{\nu=0} N^{-\nu} z^{N^{\nu}} \; .
\end{align*}
Thus we have
\[
 w (\mu) = \exp (-\pi i G_{\mu}) = E_N \; .
\]
The map $w : \Dh' (\T) \to W (\C)$ is injective on real distributions $T$ with $c_0 (T) = 0$. Since $\mu$ and $\mu * \mu$ are such distributions the equation
\[
 w (\mu * \mu) = E_n \odot E_N = E_N = w (\mu)
\]
implies that $\mu * \mu = \mu$. 

The Herglotz transform of $\mu$ is given by
\[
 h_{\mu} = 2 \pi i z G'_{\mu} = - 2 \sum^{\infty}_{\nu=0} z^{N^{\nu}} \; .
\]
It follows that we have
\[
 f_{\mu} = \exp (-h_{\mu}) = \prod^{\infty}_{\nu=0} \alpha (z^{N^{\nu}})
\]
where $\alpha = \Phi_{\Sh} (f_{\mu}) = f_{\mu} / N^* f_{\mu} = \exp (2z)$. Since $\mu \in P_{\kappa_1} (\T)$ the function $f_{\mu}$ is in $\Nh^1_1$ by \eqref{eq:8.92}. However theorem \ref{t925n} a) easily gives a more precise assertion: Since $\alpha (z) = \exp (2z)$ and its inverse are in $H^{\infty} (D)$, it follows that both $f_{\mu}$ and $f^{-1}_{\mu}$ are in $\Ah^1_1$. Thus $f_{\mu}$ is a unit in $\Ah_1$ and by \eqref{eq:8.92} we have $\mu \in P^+_{\kappa_1} (\T) \cap P^-_{\kappa_1} (\T)$. Theorem \ref{t828} further implies that the $\kappa_1$-singular measure $\mu_s = \sigma_f$ vanishes. By example \ref{t514} the function $f_{\mu}$ is not in $\Nh$. Hence $\mu$ cannot be a (signed) measure. Purely measure theoretically this can be seen as follows: If $\mu$ were a measure, then by $N$-invariance we would have $\mu = c \lambda + \mu_s$ where $c \in \R$ is a constant and $\mu_s$ is a singular measure. This would imply
\begin{align*} 
 \sigma & = \Phi_{\Sh} (\mu) = \mu - N^* \mu \overset{\eqref{eq:6.62}}{=} \mu - N^{-1} \Tr_N (\mu)\\
& = \mu_s - N^{-1} \Tr_N (\mu_s) \; .
\end{align*}
Thus $\sigma = -\pi^{-1} \cos \theta \, d\theta$ would have to be singular which is not the case. 
\end{proof}

\begin{rem}
Similar assertions hold for the elements
\[
 E_{\Sh} = \exp \Big( \sum_{N \in \Sh} \frac{z^N}{N} \Big) \in W (\Q)
\]
where $\Sh$ is the semigroup generated by $N_1 , \ldots , N_s \ge 2$.
\end{rem}

\end{document}